\documentclass[11pt]{article}
\usepackage[english]{babel}
\usepackage{amsmath}
\usepackage{amsthm}
\usepackage{amsfonts}
\usepackage{amssymb}
\usepackage{euscript}
\usepackage[cp1251]{inputenc}
\usepackage[pdftex]{graphicx}
\usepackage [autostyle, english = american]{csquotes}
\usepackage{tikz-cd}
\usepackage{float}
\usepackage[pdfpagelabels,bookmarksdepth=3]{hyperref}
\hypersetup{colorlinks=true,citecolor=black,urlcolor=black,linkcolor=black}

\newcommand{\xrightarrowdbl}[2][]{%
    \leftarrow\mathrel{\mkern-14mu}\xrightarrow[#1]{#2}
}

\newtheorem{theorem}{Theorem}[section]
\newtheorem{lemma}[theorem]{Lemma}

\theoremstyle{definition}
\newtheorem{definition}[theorem]{Definition}

\newtheorem{conjecture}[theorem]{Conjecture}

\newtheorem{example}[theorem]{Example}

\theoremstyle{remark}
\newtheorem{remark}[theorem]{Remark}

\numberwithin{equation}{section}

\begin{document}

\begin{center}
\textbf{\large{The first step towards symplectic homotopy theory}}
\end{center}
\begin{center}
\textbf{Vardan Oganesyan}\\
University of California Santa Cruz,\\
E-mail address: vardanmath@gmail.com
\end{center}


\textbf{Abstract.} We consider two categories related to symplectic manifolds:

\vspace{0.03in}
\noindent
1. Objects are symplectic manifolds and morphisms are symplectic embeddings.

\vspace{0.03in}
\noindent
2. Objects are symplectic manifolds endowed with compatible almost complex structure and morphisms are pseudoholomorphic maps.

\vspace{0.04in}
\noindent

We define new homotopy theories for these categories.  In particular, we give  definitions of homotopy equivalent symplectic manifolds and define  new cohomology theories. Theses cohomology theories are functorial, homotopy invariant and have other interesting properties.

We also construct triangulated persistence category of symplectic manifolds. This allows to apply machinery developed Biran, Cornea, and Zhang and define distances between symplectic manifolds.

\tableofcontents
\newpage

\section{Introduction}

This paper was inspired by motivic homotopy theory of algebraic varieties and by the following question of Fukaya (asked in \cite{Fukaya})

\begin{center}
Can one quantize Adams spectral sequence?
\end{center}
This leads us to a different but similar question:

\begin{center}
Can we define the homotopy theory of symplectic manifolds? Can we define symplectic stable homotopy theory and the Adams spectral sequence?
\end{center}

In this paper we make the first step towards symplectic homotopy theory, but there is still a lot of work required to be done to answer the above questions.

We consider the following categories:

\vspace{0.06in}
\noindent
1. Objects are symplectic manifolds and morphisms are symplectic embeddings.

\vspace{0.06in}
\noindent
2. Objects are symplectic manifolds endowed with compatible almost complex structure and morphisms are pseudoholomorphic maps.

\vspace{0.1in}

We denote these categories by the same symbol $\mathcal{S}ymp$ and the reader can replace $\mathcal{S}ymp$ by any of the mentioned categories.

These categories are not rich enough to do homotopy theory. We do not have all limits and colimits (in sense of category theory) and as a result we can not even define simple operations $X/W$, $X\vee Y$, $X \wedge Y$. Also, all naive definitions of homotopy equivalent morphisms do not result in an equivalence relation. There are more disadvantages of these categories, but we already see that we need to add more objects and morphisms into the mentioned categories to be able to do homotopy theory.

We can try to solve all the mentioned problems  by considering symplectic structures on nonsmooth manifolds, but this is technically very complicated and seems to be not flexible enough.

We have a natural method to add morphisms and objects into our category. This method was introduced by Grothendieck and is known in algebraic geometry. Let us discuss a philosophical aspect here and give rigorous definitions in Section $\ref{aprspacetodo}$. In functional analysis we have generalized functions. Roughly speaking, generalized functions are functionals on the space of functions. Our set of functions can be embedded into the space of functionals, i.e. each function defines a generalized function. Of course in order to obtain something useful we need to work with a certain subset of `nice' functionals. It turns out that similar idea works for any small category $\mathcal{C}$. Let $\mathcal{C}^{op}$ be the opposite category. We consider a category of functors from $\mathcal{C}^{op}$ to a category with required properties (sets, abelian groups). This category of functors is called a category of presheves on $\mathcal{C}$ and is denoted by $PrSh_{\mathcal{C}}$. It turns out that $PrSh_{\mathcal{C}}$ has all limits and colimits and we are able to define analogues of quotients, smash products and other operations. On the other hand, $PrSh_{\mathcal{C}}$ is very big and we loose all information about $\mathcal{C}$. In other words, $\mathcal{C}$ may have some colimits (not all, but some of them) and the category of presheves does not respect existing colimits. To solve the mentioned problem we need to consider a category of sheaves.

We need to apply this technique to the mentioned symplectic categories. We denote the resulting categories by $Sh_{\mathcal{S}ymp}$, where $Sh$ stands for sheaves. Roughly speaking, each element of $Sh_{\mathcal{S}ymp}$ defines a sheaf on all symplectic manifolds and $\mathcal{S}ymp$ is naturally embedded into $Sh_{\mathcal{S}ymp}$. As a result, we replace our manifolds $X$ by functors $\mathcal{F}_X$ with no information loss:  $X$, $Y$ are isomorphic if and only if $\mathcal{F}_X$ is naturally isomorphic to $\mathcal{F}_Y$ (see Section $\ref{aprspacetodo}$). The advantage is that now we can define quotients, wedge products, and symplectic spheres (with the understanding that those are sheaves, not manifolds). Moreover, now we are able to define homotopy equivalent maps (see Section $\ref{homequivalence})$.

We also define  new cohomology groups. These cohomology groups have all required properties: functoriality, homotopy invariance,... We discuss all details in Sections $\ref{symphomgroup}$ and $\ref{homequivalence2}$. Note that there are not so many functorial nvariants in symplectic geometry.

In fact, we get that any object  $X \in \mathcal{S}ymp$ defines a cohomology theory and any morphism $f: X \rightarrow Y$ defines a map between cohomology theories associated with $X$ and $Y$.

\vspace{0.07in}

We also construct triangulated persistence category of symplectic manifolds (see Section $\ref{triangperscat}$). This allows to apply machinery developed Biran, Cornea, and Zhang (see $\cite{Octav1}$, $\cite{Octav2}$) and define distances between symplectic manifolds.

\vspace{0.11in}

Let us note that this technique is not completely new in symplectic geometry. Similar methods was used by Pardon in \cite{Pardon} to construct the fundamental class.

In fact, homotopy theory of some other `nonflexible' categories exist. In $1990 - 2000$ Morel, Suslin, and Voevodsky defined $A^1-$homotopy theory of algebraic varieties.  This theory helped to solve many open problems in algebraic geometry; the interested reader can find all details in \cite{motivic1, motivic2, motivic3}. Our definitions are partially inspired by these papers.

\vspace{0.09in}

Since the technique used in this paper is not widely known by symplectic geometers, we provide all definitions and give geometric interpretation of all results.  We only assume that the reader is familiar with basic concepts of category theory.

\vspace{0.13in}

\noindent \textbf{Acknowledgments.}
The author thanks Artem Kotelskiy, Egor Shelukhin, Ivan Panin, Octav Cornea, Viktor Ginzburg for many helpful discussions.

\section{Some reasons to read this paper}

In this section we discuss some relations of known invariants and open problems to the cohomology groups defined in this paper. These relations will be studied in further papers and in this paper we only mention them.

\vspace{0.05in}
The reader may ask if it is useful to state some open problems in terms of the defined cohomology groups. The answer is yes. As soon as we define a cohomology theory with all the nice properties (functoriality, homotopy invariance, \ldots ), we immediately get some tools to study these problems. For example, we get exact sequences, spectral sequences and other tools, like in algebraic topology.
In this paper we make the first step and only define cohomology groups and prove that our groups have all the mentioned nice properties.

\vspace{0.05in}
In this Section we simplify our definitions to be able to explain the basic ideas.

\subsection{Questions and applications of motivic pseudoholomorphic groups $JH_{\bullet}^M(Y, X)$}\label{question1}

Let $X$, $Y$, $Z$, $W$ be arbitrary symplectic manifolds endowed with compatible almost complex structures. Let $M = (M, \omega_M, J_M,  p_0, p_1)$ be an arbitrary connected symplectic manifold with two marked points and $\omega_M$-compatible almost complex structure $J_M$. Here $M$ plays a role of the unit segment.

We assign abelian groups to any triple $X, Y, M$
\begin{equation*}
X, Y, M \; \rightarrow \; JH_{n}^M(Y, X), \;\; \text{where $n \geqslant 0$}.
\end{equation*}

Let $f: Y \rightarrow Z$ and $g: X \rightarrow W$ be pseudoholomorphic maps. These maps induce homomorphisms between cohomology groups

\begin{equation*}
f^{*}: JH_{\bullet}^M(Z, X) \rightarrow JH_{\bullet}^M(Y, X), \quad g_{*}: JH_{\bullet}^M(Y, X) \rightarrow JH_{\bullet}^M(Y, W).
\end{equation*}

\begin{remark}
In fact, similar cohomology groups can be defined for any almost complex manifolds, not necessarily symplectic. The construction is absolutely the same, but we consider only symplectic manifolds with compatible almost complex structure.
\end{remark}

Let us consider a few applications and questions.

\vspace*{0.1in}
\noindent
\textbf{1.} Assume that $Y = pt$ is a point and $M = \mathbb{C}P^1$ endowed with the standard symplectic and almost complex structures, where $p_0 = [0:1]$ and $p_1 = [1:0]$.

\begin{theorem}\label{cohomkahler}(see Section $\ref{computations}$)
If $X$ is a Kahler manifold and $JH_{0}^{\mathbb{C}P^1}(pt, X) = \mathbb{Z}$, then $X$ is algebraic.
\end{theorem}

The proof of this theorem is simple and is based on a paper of Campana (see $\cite[p.~212]{Campana})$.

\vspace*{0.1in}
\noindent
\textbf{2.} We have the following lemma:

\begin{lemma}\label{curvepassinglemma}(see Section $\ref{computations}$)
If there is a pseudoholomorphic sphere passing through any two points of $X$, then $JH_0^{\mathbb{C}P^1}(pt, X) = \mathbb{Z}$. Here $X$ is not necessarily Kahler.
\end{lemma}

\vspace*{0.1in}
\noindent
\textbf{3.} We define homotopy between pseudoholomorphic maps in Section $\ref{homequivpseud}$. This homotopy equivalence results in equivalence relation.

If $X$, $W$ are homotopy equivalent (in sense of Section $\ref{homequivpseud}$), then $JH^M_{\bullet}(Y, X) = JH^M_{\bullet}(Y,W)$ for any $Y$ and $M$.

If $Y$, $Z$ are homotopy equivalent (in sense of Section $\ref{homequivpseud}$), then $JH^M_{\bullet}(Y, X) = JH^M_{\bullet}(Z, X)$ for any $X$ and $M$.

This can help to translate some geometric properties . For example, if $X, W$ are homotopy equivalent (in sense of Section $\ref{homequivpseud}$) Kahler manifolds and $X$ is algebraic, then $W$ is also algebraic. This follows from Lemma $\ref{cohomkahler}$.

\vspace*{0.1in}
\noindent
\textbf{4.} (see Section $\ref{cancelpseud}$)  If there exists a pseudoholomorphic map $\gamma: \;  M \times M \rightarrow M$ such that
\begin{equation*}
\gamma(p_0 \times M) = p_0, \quad  \gamma(p_1 \times M) = id_M,
\end{equation*}
then we have following isomorphisms
\begin{equation*}
JH_{\bullet}^M(Y, X \times M) = JH_{\bullet}^M(Y, X), \quad JH_{\bullet}^M(Y \times M, X) = JH_{\bullet}^M(Y, X).
\end{equation*}

\vspace*{0.08in}
\noindent
\textbf{5.} Let $X$ be a smooth algebraic variety (or Kahler manifold). Then, cohomology groups $JH_{\bullet}^M(Y, X)$ define new invariants of algebraic varieties (Kahler manifolds).

\vspace*{0.08in}
\noindent
\textbf{6.} Cohomology groups $JH^M_{\bullet}(Y, X)$ depend on almost complex structures of $X, Y$ and $M$. On the other hand, we mentioned that homotopy equivalent (in sense of Section $\ref{homequivpseud}$) manifolds define the same cohomology groups.

\vspace{0.08in}
\noindent
\textbf{Question.} How do groups $JH_{\bullet}^M(Y, X)$ depend on almost complex structures? Let $X_0 = (X, \omega_X, J_0)$, $X_1 = (X, \omega_X, J_1)$, where $J_0, J_1$ are $\omega_X$-compatible. When are manifolds $X_0$, $X_1$ are homotopy equivalent (in sense of Section $\ref{homequivpseud}$).

\subsection{Questions and applications of motivic symplectic groups $SH_{\bullet}^M(Y, X)$}\label{question2}

Let $X$, $Y$, $Z$, $W$ be arbitrary symplectic manifolds. Let $M = (M, \omega_M,  p_0, p_1)$ be an arbitrary connected symplectic manifold with two marked points. Here $M$ plays a role of the unit segment.

We assign abelian groups to any triple $X, Y, M$
\begin{equation*}
X, Y, M \; \rightarrow \; SH_{n}^M(Y, X), \;\; \text{where $n \geqslant 0$}.
\end{equation*}

Let $f: Y \rightarrow Z$ and $g: X \rightarrow W$ be symplectic maps. These maps induce homomorphisms between cohomology groups
\begin{equation*}
f^{*}: SH_{\bullet}^M(Z, X) \rightarrow SH_{\bullet}^M(Y, X), \quad g_{*}: SH_{\bullet}^M(Y, X) \rightarrow SH_{\bullet}^M(Y, W).
\end{equation*}

\noindent
\textbf{1.} Consider $f_t: Y \rightarrow Z$ and $g_t: X \rightarrow W$, where $f_t$, $g_t$ are symplectic for any $t$. Then, we have
\begin{equation*}
f_0^{*} = f_1^{*}, \quad (g_0)_{*} = (g_1)_{*}.
\end{equation*}
In other words, if $f_0$, $f_1$ and $g_0$, $g_1$ are symplectically isotopic, then these maps induce equal maps of cohomology groups (see Theorem $\ref{deformtheorem}$).

Seidel in $\cite{Seidel1}$ and $\cite{Seidel2}$ constructed symplectomorphisms on four-manifolds that are smoothly, but not symplectically, isotopic to the identity. Our new cohomology groups probably provides a new method to find similar examples in any dimension.

Consider symplectomorphisms $g_0, g_1: X \rightarrow X$. Assume that $g_0$ and $g_1$ are smoothly isotopic. For any $Y$ and $M$ we have the isomorphisms
\begin{equation*}
(g_0)_{*}, \; (g_1)_{*}: \; SH_{\bullet}^M(Y, X) \rightarrow SH_{\bullet}^M(Y, X).
\end{equation*}
If $(g_0)_{*} \neq (g_1)_{*}$, then $g_0$ and $g_1$ are not symplectically isotopic.

The problem is that we do not have enough tools yet to compute these cohomology groups and compare the isomorphisms.

\begin{conjecture}
Let $g_0, g_1: X \rightarrow X$ be symplectomorphisms. If $g_0$, $g_1$ are smoothly isotopic and are not symplectically isotopic, then there exist $Y$ and $M$ such that $(g_0)_{*} \neq (g_1)_{*}$.
\end{conjecture}

\noindent
\textbf{2} We define homotopy between symplectic embeddings in Section $\ref{homequivsymp}$. This homotopy equivalence generalizes the standard symplectic isotopy and results in equivalence relation. So, if two maps are symplectically isotopic, then they are hmotopic in sense of Section $\ref{homequivsymp}$ (see Theorem $\ref{deformtheorem}$).

\vspace{0.08in}
\noindent
\textbf{3.} The question of whether there is a unique smooth symplectic representative in $\mathbb{C}P^2$ of each homology class up to symplectic isotopy is known as the symplectic isotopy problem. This problem is still open and was studied by many people (see $\cite{Gromov2}$, $\cite{Sikorav}$, $\cite{Tian}$, $\cite{Starkston}$).

Let $D_1$, $D_2$ be disks, or polydisks, or similar manifolds. Assume that $dim(D_1) = dim(D_2)$. It is very hard to study the space of symplectic embeddings of $D_1$ into $D_2$. In general, it is not even known if the spaces of symplectic embeddings are connected.

Any symplectic embedding $Y \rightarrow X$ represents an element of $SH^M_{0}(Y, X)$.

\begin{lemma}\label{sympapplemma}(see Section $\ref{computations}$)
Consider arbitrary symplectic embeddings $f, g: Y \rightarrow X$. Suppose that for some $n$ there exists symplectic embedding of $M$ into $X^n$. If $f$, $g$ represent different elements in $SH_0^M(Y, X)$, then images $f(Y)$, $g(Y)$ are not symplectically isotopic through embeddings.

In particular, if some $f$ and $g$ represent different elements in $SH_0^M(D_1, D_2)$, then the space of symplectic embeddings is not connected.

\end{lemma}

This lemma shows that as soon as we have enough tools to compute the cohomology groups (like in algebraic topology), we will be able to study the space of symplectic embeddings.

\vspace{0.1in}
\noindent
\textbf{4.} (see Section $\ref{cancelsymp}$ and Example $\ref{prodembedd}$)  Suppose that there is a symplectic map
\begin{equation*}
\gamma: M \times M \rightarrow M \times M, \;\;\;\; \gamma(p_1 \times q) = (q \times p_0), \;\; \forall \; q\in M.
\end{equation*}
If $\gamma$ is symplectically isotopic to the identity and $dim(X), dim(Y) > 0$, then we have the following isomorphisms
\begin{equation*}
SH_{\bullet}^M(Y, X \times M) = SH_{\bullet}^M(Y, X), \quad SH_{\bullet}^M(Y \times M, X) = SH_{\bullet}^M(Y, X).
\end{equation*}
In fact, $M$ needs to satisfy one more technical condition, but we skip it. The reader can find all details in Section $\ref{cancelsymp}$.

\vspace{0.1in}
\noindent
\textbf{5.} The groups $SH^M_{\bullet}(Y, X)$ may be used to measure complexity of symplectomorphisms. Let $f: X \rightarrow X$ be a symplectomorphism. Then, we have the isomorphisms $f_{*}$ of cohomology groups. The complexity of $f$ may be measured by complexity of the isomorphism $f_{*}$. As we mentioned, if $f$ is symplecically isotopic to the identity, then $f_{*} = (id_X)_{*}$.

\vspace*{0.1in}
\noindent
\textbf{6.} Let $X_1 = (X, \omega_1)$, $X_2 = (X, \omega_2)$ be symplectic manifolds, where $\omega_1, \omega_2$ are connected by a path of symplectic forms.

\begin{conjecture}
We can choose some $Y$, $M$ to distinguish  $X_1$ and $X_2$ using $SH_{\bullet}^M(Y, X_1)$ and $SH_{\bullet}^M(Y, X_2)$.
\end{conjecture}

\section{Motivation}

\subsection{Theorem of Dold-Thom}\label{motivationdold}

Let $(X, e)$ be some nice based topological space. Denote by $SP^k(X, e)$ the symmetric product, i.e.

\begin{equation*}
SP^{k}(X, e) = X^k/\Sigma_k = (\underbrace{X\times \ldots \times X}_k)/\Sigma_k,
\end{equation*}
where $\Sigma_k$ is the symmetric group acting on $X^k$ by permuting the factors. We denote elements  $a \in SP^{k}(X,e)$ by unordered set $a = \{p_1,\ldots,p_k\}$. We have the following embedding
\begin{equation*}
SP^{k}(X,e) \hookrightarrow SP^{k+1}(X,e), \;\;\; \{p_1,\ldots, p_k\} \hookrightarrow \{e, p_1,\ldots, p_k\}
\end{equation*}
So, we get a sequence of embeddings
\begin{equation*}
SP^0(X,e) = e \in X=SP^1(X,e) \hookrightarrow SP^2(X,e) \hookrightarrow \ldots
\end{equation*}
We define the infinite symmetric product
\begin{equation*}
SP(X,e) = \bigcup_{k=0}^{\infty}SP^{k}(X,e) = colim(SP^k(X,e)),
\end{equation*}
where a set $V \subset SP(X,e)$ is closed if and only if $V \cap SP^k(X,e)$ is closed for any $k$. In other words, points of $SP(X,e)$ are infinite unordered sets $\{\ldots,e,p_1,\ldots,p_k,e\ldots\}$, where only finitely many terms are not the basepont $e$.

\vspace{0.1in}

\begin{remark}
For simplicity, we write $\{p_1,\ldots,p_k\}$ instead of $\{\ldots,e,p_1,\ldots,p_k, e,\ldots\}$, where $p_j \neq e$ for any $j=1,\ldots,k$.
\end{remark}

\vspace{0.07in}

The infinite symmetric product is commutative, associative monoid with unity (a group without inverses). Indeed, we define
\begin{equation*}
\begin{gathered}
\{p_1,\ldots, p_r\} + \{q_1,\ldots,q_s\} = \{p_1,\ldots,p_r,q_1,\ldots,q_s\},\\
\{e\} + \{p_1,\ldots, p_r\} = \{p_1,\ldots, p_r\}.
\end{gathered}
\end{equation*}
Let $\Delta^n$ be the standard $n-$simplex. Consider a monoid
\begin{equation*}
Map(\Delta^n, SP(X,e)),
\end{equation*}
where $Map$ stands for all continuous maps. Since there are $n+1$ embeddings  $i_k: \Delta^{n-1} \hookrightarrow \Delta^n$ as facets, we get $n+1$ induced maps (restrictions)
\begin{equation*}
\begin{gathered}
\partial_k: Map(\Delta^n, SP(X,e)) \rightarrow Map(\Delta^{n-1}, SP(X,e)), \\
 \partial_k =f\circ i_k, \;\; f\in Map(\Delta^n, SP(X,e)).
\end{gathered}
\end{equation*}
Let $Map(\Delta^n, SP(X,e))^{+}$ be the Grothendieck group of the monoid $Map(\Delta^n, SP(X,e))$ (e formally add inverses respecting the structure of the monoid). Denote
\begin{equation*}
TC_n = Map(\Delta^n, SP(X,e))^{+}
\end{equation*}
If $f \in Map(\Delta^n, SP(X,e))$, then we denote its inverse in $TC_{n}$ by $-f$. We define the following map:
\begin{equation*}
d: TC_n \rightarrow TC_{n-1}, \;\;\; d = \sum\limits_{k=0}^{n}(-1)^k\partial_k.
\end{equation*}
Direct computations show that $d^2 = 0$ and we get a chain complex
\begin{equation*}
\rightarrow TC_{n+1} \rightarrow TC_n \rightarrow TC_{n-1} \rightarrow \ldots \rightarrow TC_0 \rightarrow 0.
\end{equation*}

\begin{example}
Consider $TC_0(X) = Map(\Delta^0, SP(X, e))^+$, where $\Delta^0$ is a point. We see that $TC_0$ is isomorphic to $\mathbb{Z}[X]/\mathbb{Z}[e]$, i.e. formal finite linear combinations of points of $X$ modulo multiples of $e$.
\end{example}

The theorem below was proved by Dold and Thom in $1958$.

\begin{theorem}(Dold-Thom, \cite{DoldThom})
Let $H(TC_{*}, d)$ be the homology groups of the complex $(TC_{*},d)$. There is an isomorphism between $H(TC_{*}, d)$ and the ordinary reduced singular homology groups $\widetilde{H}_{*}^{sing}(X, \mathbb{Z})$.
\end{theorem}

\begin{remark}
Dold and Thom also proved that $\widetilde{H}_{*}(X, \mathbb{Z})$ is isomorphic to $\pi_{*}(SP(X, e))$.
\end{remark}

\vspace{0.06in}

The complex of Dold-Thom has an alternative construction. Let $X$ be some nice topological space (not based). Consider
\begin{equation*}
\coprod_{k=1}SP^k(X),
\end{equation*}
where $\sqcup$ is the disjoint union. Note that $SP^0(X)$ is not included into the disjoint union.   We define the addition on $\coprod_{k=1}SP^k(X)$ as before
\begin{equation*}
\{p_1,\ldots, p_r\} + \{q_1,\ldots,q_s\} = \{p_1,\ldots,p_r,q_1,\ldots,q_s\}.
\end{equation*}
and get the addition on $Map(\Delta^n, \coprod_{k=1}SP^k(X))$. We say that an element of $f \in Map(\Delta^n, SP^n(X,e))$ is irreducible if
\begin{equation*}
f \neq f_1 + f_2, \quad f_1, f_2 \in Map(\Delta^n, \coprod_{k=1}SP^k(X)),
\end{equation*}
We denote the set of irreducible maps by $Ir(Map(\Delta^n, \coprod_{k=1}SP^k(X)))$. Then, we consider a group of finite formal linear combinations
\begin{equation*}
K_n(X) = \mathbb{Z}[Ir(Map(\Delta^n, \coprod_{k=1}SP^k(X)))].
\end{equation*}
In the same way, we can define a group
\begin{equation}\label{requiredfor}
K(Y, X) = \mathbb{Z}[Ir(Map(Y, \coprod_{k=1}SP^k(X)))],
\end{equation}
where $Y$ is a topological spaces. Indeed, the structure of $\Delta_n$ is not used to define $K_n$.

\begin{example}
Consider a map $f: \Delta^0 \rightarrow SP^k(X)$. Since $\Delta^0$ is a point, we get that $f(\Delta^0)$ is a point $\{x_1, \ldots, x_k\}$. Let $f_i: \Delta^0 \rightarrow X$ be a map such that $f_i(\Delta^0) = x_i$. Then,
\begin{equation*}
f = f_1 + \ldots + f_k.
\end{equation*}
This shows that $f$ is irreducible if and only if $k=1$. Therefore, $K_0$ is isomorphic to finite formal linear combinations of points of $X$ with integer coefficients.
\end{example}

We define the restriction $\partial_k$ and the differential $d$ by the same formulas (given above) and get a chain complex
\begin{equation*}
\rightarrow K_{n+1} \rightarrow K_n \rightarrow K_{n-1} \rightarrow \ldots \rightarrow K_0 \rightarrow 0
\end{equation*}

\begin{theorem}(Dold-Thom, \cite{DoldThom})
Let $H(K_{*}, d)$ be the homology groups of the complex $(K_{*},d)$. There is an isomorphism between $H(K_{*}, d)$ and the ordinary singular homology groups $H_{*}^{sing}(X, \mathbb{Z})$.
\end{theorem}

We have the following question:

\begin{center}
Assume that spaces $X, Y$ have some extra structures (symplectic manifolds, Riemannian manifold, complex manifolds...) and consider the complexes defined above, where maps respect their extra structures. Can we use the ideas discussed above to  define cohomology theory for spaces with extra?
\end{center}

Suslin, Morel and Voevodskiy used the complex, defined above, to construct Suslin homology and motivic homotopy theory of algebraic varieties.

\subsection{Motivic cohomology groups of algebraic varieties}

The reader may skip this section. In this section we give semi-true facts about motivic cohomology groups, but show how the ideas explained in the previous section can be used to define new invariants of algebraic varieties.
\\

Let $X$ be an algebraic variety. We define an algebraic $n-$simplex
\begin{equation*}
\Delta^n_{alg} = \{(z_0, \ldots, z_{n}) \in \mathbb{C}^{n+1} \; | \; z_0 + \ldots + z_{n} = 1\}
\end{equation*}
We define embedding of $(n-1)-$simplex into $n-$simplex
\begin{equation*}
i_k(\Delta^{n-1}) = \Delta^n \cap \{z_k = 0\}, \;\;\; k=0, \ldots, n.
\end{equation*}

Let us note that $SP^N(X)$ is also an algebraic variety and assume that we know how to define infinite symmetric product $SP(X)$. Let $Map_{alg}(\cdot, \cdot)$ be the set of algebraic maps between algebraic varieties. Arguing as in the previous section, we note that $Map_{alg}(\Delta^n, SP(X))$ is a monoid. We consider the Grothendieck group and a chain complex
\begin{equation*}
\begin{gathered}
C_n = Map_{alg}(\Delta^n, SP(X))^{+}, \\
d: C_n \rightarrow C_{n-1}, \;\;\; d = \sum\limits_{k=1}^{n+1} (-1)^k\partial_k, \;\;\; \partial_k(f) = f \circ i_k, \;\;\; d^2 = 0.
\end{gathered}
\end{equation*}
The homology groups $H^{sus}_{*}(X) = H(C_{*}(X), d)$ are called the Suslin homology of $X$.

\begin{example}
Let $S_g$ be a Riemann surface of genus $g$ and $J(S_g)$ be its Jacobian. We have
\begin{equation*}
H_0^{sus}(S_g) = \mathbb{Z} \times J(S_g).
\end{equation*}
In particular, if $T$ is $2-$torus, then $H_0^{sus}(T) = \mathbb{Z} \times T$.
\end{example}

Let $Y$ be an algebraic variety and $U \subset Y$ be a Zarisky open subset. As in the previous section, we get a presheaf of chain complexes
\begin{equation*}
U \rightarrow Map_{alg}^{+}(\Delta^{*} \times U, \; SP(X)).
\end{equation*}
This presheaf is a sheaf (we do not discuss topology) and we get hypercohomology groups of this sheaf. As a result, we get that

\begin{center}
Any variety $X$ assign hepercohomology groups to all varieties $Y$. This assignment is called the motive of $X$.
\end{center}

Many nontrivial algebraic invariants are hidden in these hypercohomology groups.

\subsection{Some problems to define symplectic motivic cohomology groups}

Let $X$ be a symplectic manifold. We can not mimic algebraic construction of motivic cohomology groups and define symplectic analogue. There are the following problems:

\vspace{0.1in}
\noindent
1. The symmetric product $SP^N(X)$ is not a smooth manifold. Singularities in algebraic geometry come in a natural way, but not in symplectic geometry. To be able to study symplectic (or some other maps) between $Y$ and $SP^N(X)$ we need to define a symplectic structure on $SP^N(X)$. In fact, $SP^N(X)$ is stratified by smooth symplectic maps and we can try to work with stratified symplectic manifolds. This is very inconvenient because it's hard to deal with maps between stratified spaces and it's hard to prove that something is functorial. So, we need to find some convenient symplectic analogue of $SP^N(X)$.

\vspace{0.1in}
\noindent
2. We need to define symplectic simplices. Standard topological simplices do not work for us. Algebraic simplices $\Delta^n_{alg}$ (defined in the previous section) are symplectic manfolds and inherit symplectic structure from $\mathbb{C}^{n+1}$. Unfortunately, this symplectic structure looks ugly.

There is a bigger problem. When we define some cohomology groups, we want them to be homotopy invariant. The author do not know how to define homotopy invariant cohomology groups of symplectic manifolds using algebraic simplices.

To avoid the mentioned problems, we consider cubical complexes instead of simplicial complexes.

\section{Preliminaries}\label{preliminaries}

\subsection{Category $\mathcal{S}\mathcal{S}ymp$}

We define a symplectic category $\mathcal{S}\mathcal{S}ymp$, where

\vspace{0.06in}
\noindent
1. Objects are symplectic manifolds $(X, \omega_X)$ without boundary.

\vspace{0.06in}
\noindent
2. The set of morphisms $\mathcal{S}\mathcal{S}ymp(X, Y)$ between $X$ and $Y$ is the set of symplectic embeddings, i.e $f \in \mathcal{S}\mathcal{S}ymp(X, Y)$ if $f^{*}\omega_Y = \omega_X$ and $f$ is embedding.

\subsection{Category $\mathcal{J}\mathcal{S}ymp$}

We consider a category, where

\vspace{0.06in}
\noindent
1. Objects are symplectic manifolds $(X, \omega_X, J_X)$  with $\omega_X$-compatible almost complex structure $J_X$. We assume that manifolds do not have boundaries.

\vspace{0.06in}
\noindent
2. The set of morphisms $\mathcal{J}\mathcal{S}ymp(X, Y)$ between $X$ and $Y$ is the set of pseudomorphic maps, i.e $f \in \mathcal{J}\mathcal{S}ymp(X, Y)$ if $df \circ J_X = J_Y \circ df$.

\begin{remark}
As we mentioned, the construction of cohomology groups holds true for any almost complex manifolds, not necessarily symplectic. The construction is the same, but we consider only symplectic manifolds with compatible almost complex structure.
\end{remark}

\section{Appropriate spaces to do homotopy theory}\label{aprspacetodo}

As it is discussed in the introduction, the categories $\mathcal{S}\mathcal{S}ymp$ and $\mathcal{J}\mathcal{S}ymp$ are not rich enough to do homotopy theory. In this section we construct an appropriate category to do homotopy theory (see introduction for the philosophy behind the construction).

\subsection{Notation}

In this section we prove all results simultaneously for the categories $\mathcal{S}\mathcal{S}ymp$ and $\mathcal{J}\mathcal{S}ymp$. To simplify notation we denote by $\mathcal{S}ymp$ either $\mathcal{S}\mathcal{S}ymp$, or $\mathcal{J}\mathcal{S}ymp$. So, if some lemma works for all these categories, then we say that it works for $\mathcal{S}ymp$. We denote elements of morphisms $\mathcal{S}ymp(X, Y)$ between $X$ and $Y$ by $f$, $g \ldots$

\subsection{Category of presheaves on $\mathcal{S}ymp$}\label{sheavesonsymp}

Denote by $\mathcal{S}ymp^{op}$ the opposite category and by $\mathcal{S}et$ the category of sets. Let us consider the category of presheaves $PrSh_{\mathcal{S}ymp}$, i.e. functors
\begin{equation*}
\mathcal{F}: \mathcal{S}ymp^{op} \rightarrow \mathcal{S}et.
\end{equation*}
Recall that $\mathcal{S}ymp$ stands for either  $\mathcal{S}\mathcal{S}ymp$, or $\mathcal{J}\mathcal{S}ymp$.

\begin{remark}
We can also consider functors to the category of topological spaces. All theorems below work for presheaves of topological spaces if we assume that $\mathcal{S}\mathcal{S}ymp(Y, X)$ and $\mathcal{J}\mathcal{S}ymp$ induce topology from the space of all continuous maps from $Y$ to $X$. To simplify this paper we consider presheaves of sets.
\end{remark}

Morphisms of the category $PrSh_{\mathcal{S}ymp}$ are natural transformations. Let us recall that a natural transformation $\eta$ from $\mathcal{F}_1$ to $\mathcal{F}_2$ is a family of maps $\eta_X: \mathcal{F}_1(X) \rightarrow \mathcal{F}_2(X)$, for all $X \in \mathcal{S}ymp$, such that the following diagram commutes for any $f \in \mathcal{S}ymp(X, Y)$
\begin{equation*}
\begin{gathered}
\mathcal{F}_1(X) \xrightarrow{\eta_X}  \mathcal{F}_2(X) \\
\Big\uparrow f^{*} \;\;\;\;\;\;\;\;\;\;\;\; \Big\uparrow f^{*} \\
\mathcal{F}_1(Y) \xrightarrow{\eta_Y}  \mathcal{F}_2(Y)
\end{gathered}
\end{equation*}
We say that $\eta$ is natural isomorphism if $\eta_X$ is a bijection for any $X$.

Let us unwind this definition. Let $\mathcal{F}$ be a presheaf and $X$ be a symplectic manifold. For any open subsets $V \subset U$ of $X$ we have sets $\mathcal{F}(X)$, $\mathcal{F}(U)$, $\mathcal{F}(V)$. Let $i_V: V \hookrightarrow U$, $i_U: U \hookrightarrow X$  be  embeddings. Since $\mathcal{F}$ is a functor from the opposite category, we get maps of sets
\begin{equation*}
\begin{gathered}
\mathcal{F}(i_U): \mathcal{F}(X) \rightarrow \mathcal{F}(U), \;\;\; \mathcal{F}(i_V): \mathcal{F}(U) \rightarrow \mathcal{F}(V), \\
\mathcal{F}(i_U \circ i_V) = \mathcal{F}(i_V) \circ \mathcal{F}(i_U).
\end{gathered}
\end{equation*}
This means that we have a presheaf on $X$. As a result we see that
\begin{equation*}
\begin{gathered}
\text{ any object of $PrSh_{\mathcal{S}ymp}$ defines a presheaf of sets} \\
\text{on any manifold from $\mathcal{S}ymp$}.
\end{gathered}
\end{equation*}
If $f \in \mathcal{F}(X)$, then we denote
\begin{equation*}
\mathcal{F}(i_U)(f) = f|_U.
\end{equation*}

A functor $\mathcal{F}$ is called representable if there exists $X \in \mathcal{S}ymp$ such that for any $Z \in \mathcal{S}ymp$ we have
\begin{equation*}
\mathcal{F}(Z) = \mathcal{S}ymp(Z,X).
\end{equation*}
We denote this functor by $\mathcal{F}_X$. For any $f\in \mathcal{S}ymp(Y, X)$ and $h\in \mathcal{S}ymp(Y, Z)$ we get
\begin{equation*}
\begin{gathered}
f_{*}: \mathcal{F}_Y(Z) \rightarrow \mathcal{F}_X(Z), \;\;\;\; f_{*}:  \mathcal{S}ymp(Z, Y) \rightarrow \mathcal{S}ymp(Z, X)  \\
h^{*}: \mathcal{F}_X(Z) \rightarrow \mathcal{F}_X(Y), \;\;\;\; h^{*}:  \mathcal{S}ymp(Z, X) \rightarrow \mathcal{S}ymp(Y, X) \\
\end{gathered}
\end{equation*}
From the Yoneda Lemma we have
\begin{equation*}
PrSh_{\mathcal{S}ymp}(\mathcal{F}_X, \mathcal{F}_Y) = \mathcal{S}ymp(X, Y).
\end{equation*}
So, we identify $X$ with the functor $\mathcal{F}_X$ and see that $\mathcal{S}ymp$ is a full subcategory of $PrSh_{\mathcal{S}ymp}$.

\begin{lemma}
Manifolds $X, Y \in \mathcal{S}ymp$ are isomorphic if and only if $\mathcal{F}_X$ is naturally isomorphic to $\mathcal{F}_Y$, i.e there is a natural bijection between $\mathcal{S}ymp(Z, X)$ and $\mathcal{S}ymp(Z, Y)$ for any $Z \in \mathcal{S}ymp$.
\end{lemma}

\begin{proof}
Assume that $f: X \rightarrow Y$ is an isomorphism.  Then,  $f_{*}: \mathcal{S}ymp(Z, X) \rightarrow \mathcal{S}ymp(Z, Y)$ provides the required natural bijection.

Assume that there is natural bijection $\eta_Z: \mathcal{S}ymp(Z, X) \rightarrow \mathcal{S}ymp(Z, Y)$ for any $Z$. If $Z = X$, then  we get the bijection
\begin{equation*}
\mathcal{S}ymp(X, X) \xrightarrow{\eta_X} \mathcal{S}ymp(X, Y).
\end{equation*}
If $Z = Y$, then we have the bijection
\begin{equation*}
\mathcal{S}ymp(Y, X) \xrightarrow{\eta_Y} \mathcal{S}ymp(Y, Y).
\end{equation*}
Denote $f = \eta_X(id_X)$ and $g = \eta_Y^{-1}(id_Y)$.  By definition we have the following commutative diagram
\begin{equation*}
\begin{gathered}
\mathcal{S}ymp (X, X) \xrightarrow{\eta_X} \mathcal{S}ymp(X, Y) \\
\uparrow f^{*} \qquad \qquad \uparrow f^{*} \\
\mathcal{S}ymp(Y, X) \xrightarrow{\eta_Y} \mathcal{S}ymp(Y, Y)
\end{gathered}
\end{equation*}
We see that $\eta_X \circ f^{*} = f^{*} \circ \eta_Y$. Let us apply this identity to $g$ and obtain
\begin{equation*}
\eta_X \circ g(f) = f \Leftrightarrow g(f) = \eta_X^{-1} \circ f = id_X.
\end{equation*}
We get $f(g) = id_Y$ by considering similar diagram and replacing $f$ by $g$.

\end{proof}

The category $PrSh_{\mathcal{S}ymp}$ has all limits and colimits. Let us show how to find colimits of presheaves. Assume that we want to find $colim(\mathcal{F}_{\alpha})$, where $\alpha$ belongs to some index set. Since $\mathcal{F}_{\alpha}(Z) \in \mathcal{S}et$ and all colimits exist in the category of sets, we get that $colim(\mathcal{F}_{\alpha}(Z))$ exists for any $Z$. One can show that
\begin{equation*}
colim(\mathcal{F}_{\alpha})(Z) = colim(\mathcal{F}_{\alpha}(Z)).
\end{equation*}
All limits can be computed in the same way.

We see that $PrSh_{\mathcal{S}ymp}$ is some kind of generalization of the category $\mathcal{S}ymp$.  Unfortunately, this category  is too rich and we need to get rid of some objects. There is a standard argument showing that $PrSh_{\mathcal{S}ymp}$ destroys geometry.   Assume that $U, V \subset X$ are open subsets and $U \cup V = X$.  Let $i_U$, $i_V$ be corresponding embeddings of $U \cap V$ into $U$ and $V$, respectively. Then, the fact that $U \cup V = X$ in terms of category theory means that $X$ is the colimit of the following diagram
\begin{equation}\label{unionofsets}
\begin{gathered}
  U \cap V \longrightarrow U \\
\quad \big\downarrow \qquad \quad \big\downarrow  \\
\quad V \; \longrightarrow \; X
\end{gathered}
\end{equation}
In other words, $U \sqcup V/(i_U(x) \sim i_V(x)) = X$, where $\sqcup$ stands for the disjoint union. In $PrSh_{\mathcal{S}ymp}$ we have the corresponding representable functors $\mathcal{F}_X$, $\mathcal{F}_U$, $\mathcal{F}_V$, $\mathcal{F}_{U \cap V}$. Since we replace our objects by functors, it is natural to expect that $\mathcal{F}_X$ is the colimit of the following diagram
\begin{equation}\label{desiredcolimit}
\begin{gathered}
\quad \quad \quad \mathcal{F}_{U \cap V} \longrightarrow \mathcal{F}_U \\
\big\downarrow \\
\mathcal{F}_V
\end{gathered}
\end{equation}
i.e. is `union' of functors $\mathcal{F}_U$, $\mathcal{F}_V$. Unfortunately, the colimit $\widetilde{\mathcal{F}}$ of diagram $(\ref{desiredcolimit})$ evaluated at a manifold $Y$ is given by
\begin{equation*}
\widetilde{\mathcal{F}}(Y) = \mathcal{F}_U(Y) \cup \mathcal{F}_V(Y),
\end{equation*}
where $\cup$ is union of sets (not disjoint). We see that $\widetilde{\mathcal{F}}(Y)$ consists of maps from $Y$ to $U$ and from $Y$ to $V$. Note that $\mathcal{F}_{X}(Y)$ consists of maps  from $Y$ to $X$ and the image of $Y$ is not supposed to be entirely in $U$ or $V$. So, we do not want $\widetilde{\mathcal{F}}$ to be the colimit.

Since symplectic manifolds are constructed by gluing open symplectic manifolds, it is very important to avoid the mentioned problem. So, we lose our geometry when we consider presheaves. We want the representable functor $\mathcal{F}_X$  to be obtained by `gluing'  functors $\mathcal{F}_U$, $\mathcal{F}_V$. The problem mentioned above says that presheaves do not preserve existing colimits.

\subsection{Category of sheaves on $\mathcal{S}ymp$}\label{sheavesonsymp3}

As we already discussed, any element of $PrSh_{\mathcal{S}ymp}$ defines a presheaf of sets on any $X \in \mathcal{S}ymp$. We need to consider a category of sheaves $Sh_{\mathcal{S}ymp}$ on $\mathcal{S}ymp$. Sheaves allow to add new colimits taking into account already existing ones.

\begin{definition}
A presheaf $\mathcal{F} \in PrSh_{\mathcal{S}ymp}$ is a sheaf if $\mathcal{F}$ defines a sheaf of sets on any $X \in \mathcal{S}ymp$. Let us give more details. Let $\cup_{\alpha} U_{\alpha} = X$ be a covering of $X$ by open subsets. A presheaf $\mathcal{F} \in PrSh_{\mathcal{S}ymp}$ is a sheaf if and only if  the following is satisfied for any $X$ and any covering:

\vspace{0.09in}
\noindent
1. If $f, g \in \mathcal{F}(X)$ satisfy $f|_{U_{\alpha}} = g|_{U_{\alpha}}$ for any $\alpha$, then $f=g$.

\vspace{0.09in}
\noindent
2. If $f_{\alpha} \in \mathcal{F}(U_{\alpha})$  satisfy $f_{\alpha}|_{U_{\alpha} \cap U_{\beta}} = f_{\beta}|_{U_{\alpha} \cap U_{\beta}}$ for any $\alpha, \beta$, then there exists $f \in \mathcal{F}(X)$ such that $f|_{U_{\alpha}} = f_{\alpha}$.

\vspace{0.09in}
\noindent
3. $\mathcal{F}(\emptyset) = pt$, where $\emptyset$ is the empty set and $pt$ is one-element set.
\vspace*{0.11in}
\noindent
We denote sheaves on $\mathcal{S}ymp$ by $Sh_{\mathcal{S}ymp}$.
\end{definition}

Let us give an equivalent definition

\begin{definition}
A presheaf $\mathcal{F} \in PrSh_{\mathcal{S}ymp}$ is a sheaf if for any $X$ and any open covering $\cup_{\alpha} U_{\alpha} = X$ by open sets
the following is an equalizer diagram
\begin{equation*}
\mathcal{F}(X) \rightarrow \prod\limits_{\alpha}\mathcal{F}(U_{\alpha}) \rightrightarrows \prod\limits_{\beta, \gamma}\mathcal{F}(U_{\beta} \cap U_{\gamma}),
\end{equation*}
where $\prod$ is the product of sets.
\end{definition}

\vspace*{0.05in}
\begin{remark}
Any representable presheaf $\mathcal{F}_X$ is a sheaf, i.e. $\mathcal{F}_X \in Sh_{\mathcal{S}ymp}$.
\end{remark}

\vspace*{0.08in}

Let us show that  $\mathcal{F}_X$ is the colimit of diagram $(\ref{desiredcolimit})$ in the category $Sh_{\mathcal{S}ymp}$. Let $\mathcal{F} \in Sh_{\mathcal{S}ymp}$ be an arbitrary sheaf. Since $X$ is the colimit of diagram $(\ref{unionofsets})$ , we get that (by definition of sheaves) $\mathcal{F}_{U \cap V}$ is colimit of the following diagram:
\begin{equation*}
\begin{gathered}
\mathcal{F}(U \cap V) \longleftarrow \mathcal{F}(U) \\
\big\uparrow \quad \quad \;\; \quad \big\uparrow \\
\mathcal{F}(V) \longleftarrow \mathcal{F}(X)
\end{gathered}
\end{equation*}
From Yoneda lemma we obtain that $\mathcal{F}(X) = Sh_{\mathcal{S}ymp}(\mathcal{F}_X, \mathcal{F})$ and the diagram above is equivalent to
\begin{equation*}
\begin{gathered}
Sh_{\mathcal{S}ymp}(\mathcal{F}_{U \cap V}, \mathcal{F}) \longleftarrow Sh_{\mathcal{S}ymp}(\mathcal{F}_{U }, \mathcal{F}) \\
\big\uparrow  \qquad \qquad \qquad \big\uparrow \\
Sh_{\mathcal{S}ymp}(\mathcal{F}_{ V}, \mathcal{F}) \longleftarrow Sh_{\mathcal{S}ymp}(\mathcal{F}_{X}, \mathcal{F})
\end{gathered}
\end{equation*}

Since $\mathcal{F}$ is arbitrary, we get that $\mathcal{F}_X$ is the colimit of  diagram $(\ref{desiredcolimit})$ in the category $Sh_{\mathcal{S}ymp}$. In the same way, we can prove similar fact about finite number of open subsets.

 \vspace{0.05in}

So, without loss of any information, we make the replacement (let us recall that $\mathcal{S}ymp$ stands for either $\mathcal{S}\mathcal{S}ymp$, or $\mathcal{J}\mathcal{S}ymp$)

\begin{equation*}
 X \xrightarrowdbl{replace \;  by \;  sheaf } \mathcal{S}ymp(-, X)= \mathcal{F}_X \in Sh_{\mathcal{S}ymp}.
\end{equation*}

We have the following theorem:

\begin{theorem}\label{theoremsheafification} (see $\cite{sheaves1}$) There exists a functor $Sheaf$
\begin{equation*}
Sheaf: \;\; PrSh_{\mathcal{S}ymp} \rightarrow Sh_{\mathcal{S}ymp}
\end{equation*}
called  sheafification functor such that

\vspace{0.05in}
\noindent
1.  $\mathcal{F} \rightarrow Sh(\mathcal{F})$ is isomorphism if and only if $\mathcal{F}$ is a sheaf

\vspace{0.05in}
\noindent
2.  To compute a colimit in $Sh_{\mathcal{S}ymp}$ we need first to compute it in $PrSh_{\mathcal{S}ymp}$ and then apply the sheafification functor $Sheaf$.

\end{theorem}

\subsection{Some properties of $Sh_{\mathcal{S}ymp}$}\label{sheavesonsymp2}

Let us show how to define different topological operations in the category $Sh_{\mathcal{S}ymp}$. Recall that $\mathcal{S}ymp$ stands for either $\mathcal{S}\mathcal{S}ymp$  or $\mathcal{J}\mathcal{S}ymp$.

Below we assume that $X, Y, Z \in \mathcal{S}ymp$.

\vspace{0.07in}
\noindent
\textbf{Intersection and union of submanifolds.} Let $W_1, W_2 \subset X$ be symplectic submanifolds (pseudoholomorphic, in case of $\mathcal{J}\mathcal{S}ymp$). The intersection $W_1 \cap W_2$  and union $W_1 \cup W_2$ are not always smooth manifolds. We consider functors $\widetilde{\mathcal{F}_{W_1 \cap W_2}}$, $\widetilde{\mathcal{F}_{W_1 \cup W_2}}$ such that
\begin{equation*}
\begin{gathered}
\widetilde{\mathcal{F}_{W_1 \cap W_2}}(Z) = \mathcal{F}_{W_1}(Z) \cap \mathcal{F}_{W_2}(Z), \\
\widetilde{\mathcal{F}_{W_1 \cup W_2}}(Z) = \mathcal{F}_{W_1}(Z) \cup \mathcal{F}_{W_2}(Z).
\end{gathered}
\end{equation*}
In general, these functors are not sheaves and we need to sheafify them (see Theorem $\ref{theoremsheafification}$). We define
\begin{equation*}
\mathcal{F}_{W_1 \cap W_2} = Sheaf(\widetilde{\mathcal{F}_{W_1 \cap W_2}}), \quad \mathcal{F}_{W_1 \cup W_2} = Sheaf(\widetilde{\mathcal{F}_{W_1 \cup W_2}}).
\end{equation*}
So, the sheaves $\mathcal{F}_{W_1 \cap W_2}$, $\mathcal{F}_{W_1 \cup W_2}$ are analogues of intersection and union, respectively. If $W_1 \cap W_2$, $W_1 \cup W_2$ are smooth, then these sheaves are representable.

There is another way to define these sheaves. We can define $\mathcal{F}_{W_1 \cap W_2}$ as a limit of the following diagram (fiber product) in the category $Sh_{\mathcal{S}ymp}$

\begin{equation*}
\begin{gathered}
\mathcal{F}_{W_1 \cap W_2} \longrightarrow \mathcal{F}_{W_1} \\
\big\downarrow \quad \quad \;\; \quad \big\downarrow \\
\mathcal{F}_{W_2} \longrightarrow \mathcal{F}_X
\end{gathered}
\end{equation*}
where horizontal and vertical arrows to $\mathcal{F}_X$ are induced by embeddings of $W_1, W_2$ into $X$. We can define $\mathcal{F}_{W_1 \cup W_2}$ as a colimit of the following diagram in the category of $Sh_{\mathcal{S}ymp}$
\begin{equation*}
\begin{gathered}
\mathcal{F}_{W_1 \cap W_2} \longrightarrow \mathcal{F}_{W_1} \\
\big\downarrow \quad \quad \;\; \quad \big\downarrow \\
\mathcal{F}_{W_2} \longrightarrow \mathcal{F}_{W_1 \cup W_2}
\end{gathered}
\end{equation*}

\vspace{0.07in}
\noindent
\textbf{Quotient.} Let $ W \subset X$ be a symplectic submanifold (pseudoholomorphic, in case of $\mathcal{J}\mathcal{S}ymp$). Since
\begin{equation*}
\mathcal{S}ymp(Y, W) = \mathcal{F}_W(Y) \subset \mathcal{F}_X(Y) = \mathcal{S}ymp(Y, X)
\end{equation*}
for any $Y$, we define a presheaf
\begin{equation*}
(\widetilde{\mathcal{F}_X/\mathcal{F}_W})(Y) =  \mathcal{S}ymp(Y,X)/\mathcal{S}ymp(Y, W).
\end{equation*}

In general, $\widetilde{\mathcal{F}_X/\mathcal{F}_W}$ is not a sheaf. We sheafify this presheaf (see Theorem $\ref{theoremsheafification}$) and get a sheaf
\begin{equation*}
\mathcal{F}_{X/W} = Sheaf(\widetilde{\mathcal{F}_X/\mathcal{F}_W}).
\end{equation*}
We define the quotient  $X/W$ as the sheaf $\mathcal{F}_{X/W}$.

Roughly speaking, we identify maps from $Y$ to $X$ if they are different only on $W$.

In general, consider $\mathcal{F}, \mathcal{G} \in Sh_{\mathcal{S}ymp}$ such that $\mathcal{G}(Y) \subset \mathcal{F}(Y)$ for any $Y$. We define
\begin{equation}\label{quotientsymp}
\begin{gathered}
(\widetilde{\mathcal{F/G}})(Y) = \mathcal{F}(Y)/\mathcal{G}(Y), \\
\mathcal{F/G} = Sheaf(\widetilde{\mathcal{F/G}}).
\end{gathered}
\end{equation}

\vspace*{0.09in}
\noindent
\textbf{Smash product.} Let $(X_i, x_i)$ be symplectic (almost complex) manifolds with marked points $x_i \in X_i$, where $i=1,\ldots,n$. Denote
\begin{equation*}
\begin{gathered}
X = (X_1 \times \ldots \times X_n, \; \oplus_{i=1}^n\omega_{X_i}), \\
W_i = X_1 \times \ldots \times X_{i-1} \times x_i \times X_{i+1} \times \ldots \times X_n \subset X
\end{gathered}
\end{equation*}
Define a presheaf $\mathcal{F}_{W_1} \cup \ldots \cup \mathcal{F}_{W_n}$ such that
\begin{equation*}
(\mathcal{F}_{W_1} \cup \ldots \cup \mathcal{F}_{W_n})(Z) = \mathcal{F}_{W_1}(Z) \cup \ldots \cup \mathcal{F}_{W_n}(Z).
\end{equation*}
We define the smash product by the following formula:
\begin{equation}\label{smash}
\bigwedge\limits^{n}_{i=1}(X_i, x_i) = \mathcal{F}_X/(\mathcal{F}_{W_1} \cup \ldots \cup \mathcal{F}_{W_n}).
\end{equation}

\section{Symmetric products in $\mathcal{S}\mathcal{S}ymp$,  $\mathcal{J}\mathcal{S}ymp$}

\subsection{Topological symmetric products and sheaves}\label{topcorr}

In this section we discuss how to define topological symmetric products in terms of sheaves. This section should be considered as a philosophical motivation and we skip some details. We use similar ideas to define symmetric products in other categories.

\vspace{0.07in}

Let $X$ be some nice topological space. Let $X^n$ be the $n$th power of $X$ and let $\Sigma_n$ be the symmetric group acting on $X^n$ by permuting the factors. We define the symmetric product as in Section $\ref{motivationdold}$
\begin{equation*}
SP^n(X) = X^n/\Sigma_n.
\end{equation*}

Let $Y$ be a nice topological space. Consider the set of continuous maps $C^0(Y, X^n)$. Any map $f \in C^0(Y, X^n)$ is defined by $n$ ordered set of continuous maps $(f_1, \ldots, f_n)$, where $f_i: Y \rightarrow X$. The symmetric group $\Sigma_n$ acts on $X^n$  and as a result acts on $C^0(Y, X^n)$. We take a quotient $C^0(Y, X^n)/\Sigma_n$ and note that elements of the quotient are unordered set of  maps. Two elements $\{f_1, \ldots, f_n\}, \{g_1, \ldots, g_n\} \in C^0(Y, X^n)/\Sigma_n$ are equal if
\begin{equation*}
\{f_1(y), \ldots, f_n(y) \} = \{ g_1(y), \ldots, g_n(y)  \} \;\;\; \text{for any $y \in Y$},
\end{equation*}
where the equality is equality of unordered sets.

\begin{example}
Let $\mathbb{R}$ be the real line. Consider maps from $\mathbb{R}$ to $\mathbb{R}$
\begin{equation*}
f_1(x) = x, \;\;\; f_2(x) = -x, \;\;\; g_1(x) = |x|, \;\;\; g_2(x) = -|x|.
\end{equation*}
We see that $\{f_1, f_2\} = \{g_1, g_2\} \in C^0(\mathbb{R}, \mathbb{R}^2)/\Sigma_2$.
\end{example}

Any element $\{f_1, \ldots, f_n\} \in C^0(Y, X^n)/\Sigma_n$ defines a subset $F \subset Y \times X$ in the following way
\begin{equation*}
F = \bigcup_{i=1}^n\bigcup_{y \in Y}(y, f_i(y).
\end{equation*}
Note that if $f_1(y), \ldots, f_n(y)$ are distinct for any $y$, then the projection $F \rightarrow Y$ is trivial $n$-covering.

In terms of Section $\ref{sheavesonsymp}$ we have a presheaf on the category of topological spaces
\begin{equation*}
Y \rightarrow C^0(Y, X^n)/\Sigma_n.
\end{equation*}
In particular, we get a presheaf on each $Y$ by assigning
\begin{equation*}
U \rightarrow C^0(U, X^n)/\Sigma_n \quad \text{to each open $U \subset Y$}.
\end{equation*}

Let us show that $C^0(U, X^n)/\Sigma_n$ is not a sheaf. Let $U_1, U_2 \subset Y$ be open subsets such that $U_1 \cup U_2 = Y$. Consider a subset $F \subset Y \times X$ such that the projection $\pi: F \rightarrow Y$ is some nontrivial $n$-covering. Assume that the covering is trivial over $U_1$ and $U_2$. Then, the sections of these coverings provide the following elements
\begin{equation*}
\{f_1, \ldots, f_n\} \in C^0(U_1, X^n)/\Sigma_n,  \quad \{g_1, \ldots, g_n\} \in C^0(U_2, X^n)/\Sigma_n.
\end{equation*}
Moreover, $\{f_1, \ldots, f_n\} = \{g_1, \ldots, g_n\}$ on $U_1 \cap U_2$. Since $F \rightarrow Y$ is not trivial, we see that $F$ is not defined by any $h \in  C^0(Y, X^n)/\Sigma_n$. This shows that we can not glue maps $\{f_1, \ldots, f_n\}$ and $\{g_1, \ldots, g_n\}$ and the presheaf $C^0(\cdot, X^n)/\Sigma_n$ is not a sheaf.

We sheafify $C^0(\cdot, X^n)/\Sigma_n$  and get a sheaf
\begin{equation*}
TCor_n(Y, X) = Sheaf(C^0(\cdot, X^n)/\Sigma_n)(Y),
\end{equation*}
where $T$ stands for topological and $Cor$ for correspondence. We denote elements of $TCor_n(Y, X)$ by letters $F, G, \ldots$ We also abuse notation and denote subsets of $Y \times X$ also by symbols $F, G, \ldots$ We will see below that these objects are related.

\begin{example}
Elements of $TCor_1(Y, X)$ are standard continuous maps from $Y$ to $X$.
\end{example}

\begin{example}
We see that $C^0(Y, X^n)/\Sigma_n \subset TCor_n(Y, X)$.
\end{example}

\begin{lemma}
Any element $F \in TCor_n(Y, X)$ provides a map $Y \rightarrow SP^n(X)$ and defines a subset of $F \subset Y \times X$.
\end{lemma}
\begin{proof}
For any point $y \in Y$ there exists a small neighborhood $U$ and unordered set of maps $\{f_1, \ldots, f_n\} \in C^0(U, X^n)/\Sigma_n$ such that $F|_U = \{f_1, \ldots, f_n\}$. We define
\begin{equation*}
\begin{gathered}
Y \rightarrow SP^n(X), \quad y \rightarrow \{f_1(y), \ldots, f_n(y)\}\\
F = \bigcup_{i=1}^n\bigcup_{y \in U}(y, f_i(y)) \subset Y \times X.
\end{gathered}
\end{equation*}
\end{proof}

Subsets of $Y \times X$ do not always define elements of $TCor_n(Y, X)$ uniquely. If $f \in C^0(Y, X)$ is a continuous map, then $\{f, f\} \in TCor_2(Y, X)$ and subsets of $Y \times X$ corresponding to $f$ and $\{f, f\}$ coincide. In fact, elements of $TCor_n$ have more geometric definition if we fix $n$ and count graphs with multiplicities.

\begin{definition}\label{sympcorr}
Let $F \subset Y \times X$ be a subset and $\pi: F \rightarrow Y$ be the projection. We say that $F$ is topological correspondence of order $n$ between $Y$ and $X$ if any point $y \in Y$ has a neighborhood $U \subset Y$ and $n$ maps $f_i: U \rightarrow X$ such that
\begin{equation}\label{corrdef}
\pi^{-1}(U) = \bigcup_{y \in Y}\bigcup_{i=1}^n(y, f_i(y)), \;\;\;\; f_i \in C^{0}(U, X)
\end{equation}
We assume that maps $f_1,\ldots, f_n$ are unordered. These maps are not supposed to be distinct. We abuse notation and denote elements of $TCor$ and topological correspondences by the same symbols $F, G, \ldots$
\end{definition}

If $f_1(y), \ldots, f_n(y)$ are distinct for any $y\in Y$, then the projection $F \rightarrow Y$ is some $n$-covering (not necessarily trivial).

\begin{lemma}\label{corrsheafbij}
There is a bijection between topological correspondences of order $n$ and elements of $TCor_n(Y, X)$.
\end{lemma}

\begin{proof}
Let $F \subset Y \times X$ be a topological correspondence. By definition, for any small  $U \subset Y$ we have $F|_U = \cup_{y \in Y}\cup_{i=1}^n(y, f_i(y))$. We see that $\{f_1, \ldots, f_n\} \in C^0(U, X^n)/\Sigma_n$. Since these unordered maps are provided by $F$, we get that they coincide on the intersections of open subsets. Gluing these elements for all open subsets we obtain an element of $TCor_n(Y, X)$.

Let $F \in TCor_n(Y, X)$. For any small $U \subset Y$ we have $F|_U = \{f_1, \ldots, f_n\}$.  We define $F|_{U} = \cup_{y \in Y}\cup_{i=1}^n(y, f_i(y)$. Since $F\in TCor_n(Y, X)$ is a sheaf, we see that subsets of $F|_{U} \subset Y \times X$ coincide on the intersections of $U \subset Y$.
\end{proof}

\noindent
Define
\begin{equation*}
\widetilde{TCor}(Y, X) = \coprod_{n=0}TCor_n(Y, X),
\end{equation*}
where $\coprod$ is the disjoint union and $TCor_0(Y, X)$ is one-element set (we assume that $X^0$ is a point).

Let $(X, e)$ be a based nice topological space. We would like to define infinite symmetric product $SP(X, e)$ (see Section \ref{motivationdold}) in terms of sheaves. Let $F \in SCor_{n+1}(Y, X)$, $G \in SCor_n(Y, X)$ be correspondences and $U \subset Y$ be a small open subset such that $F|_U = \{f_1, \ldots, f_{n+1}\}$ and $G|_U = \{g_1, \ldots, g_n\}$.  We say that $F$ and $G$ are equivalent if
\begin{equation*}
f_{n+1} = e, \quad \{f_1, \ldots, f_n\} = \{g_1, \ldots, g_n\}
\end{equation*}
for any small $U \subset Y$. In other words, $F$ is equivalent to $G$ if they differ by the constant correspondence $Y \times e$. We take quotient by this equivalence relation and define $\widetilde{TCor}(Y, X)/\sim$. We get a presheaf on any $Y$ by assigning
\begin{equation*}
U \rightarrow \widetilde{TCor}(U, X)/\sim, \quad \text{ to each open $U \subset Y$}.
\end{equation*}
Let us show that $\widetilde{TCor}(\cdot, X)/\sim$ is not a sheaf on $Y$. We can glue sections on finite number of open subsets, but we have a problem with gluing infinitely many sections. Consider infinite number of open subsets $U_1, U_2 \ldots \subset Y$ such that $\cup U_i = Y$ and $U_i \cap U_j \neq \emptyset$ if and only if $|i-j| \leqslant 1$ (see the picture below). Define $F_n \in \widetilde{TCor}(U_n, X)/\sim \;$ as follows
\begin{equation*}
F_n = \{f_{n,1}, \ldots, f_{n,n}\} \in \widetilde{TCor}(U_n, X)/\sim, \;\;\; f_{n, i} \not\equiv e \;\; \text{on} \;\; U_n \;\; \text{for any $i$}.
\end{equation*}
Assume that $f_{n,n} \equiv e$ on $U_{n-1} \cap U_n$, but not on $U_n$. Suppose that $F_{n-1} \sim F_{n}$ on $U_{n-1} \cap U_{n}$. We can glue $F_{n-1}$ and $F_{n}$ to get a correspondence on $U_{n-1} \cup U_{n}$. Let us glue all the topological correspondences $F_n$ and denote the final element by $F$. Since $F$ can not be locally represented by the same number of maps (need more and more maps when the number $n$ increases), we get that $F \notin \widetilde{TCor}(Y, X)/\sim$. So, $\widetilde{TCor}(\cdot, X)/\sim$ is not a sheaf and is just a presheaf.

Consider the sheafification
\begin{equation*}
TCor(Y, X) = Sheaf(\widetilde{TCor}(\cdot, X)/\sim)(Y).
\end{equation*}

\begin{figure}[h]
\centering
\includegraphics[width=0.58\linewidth]{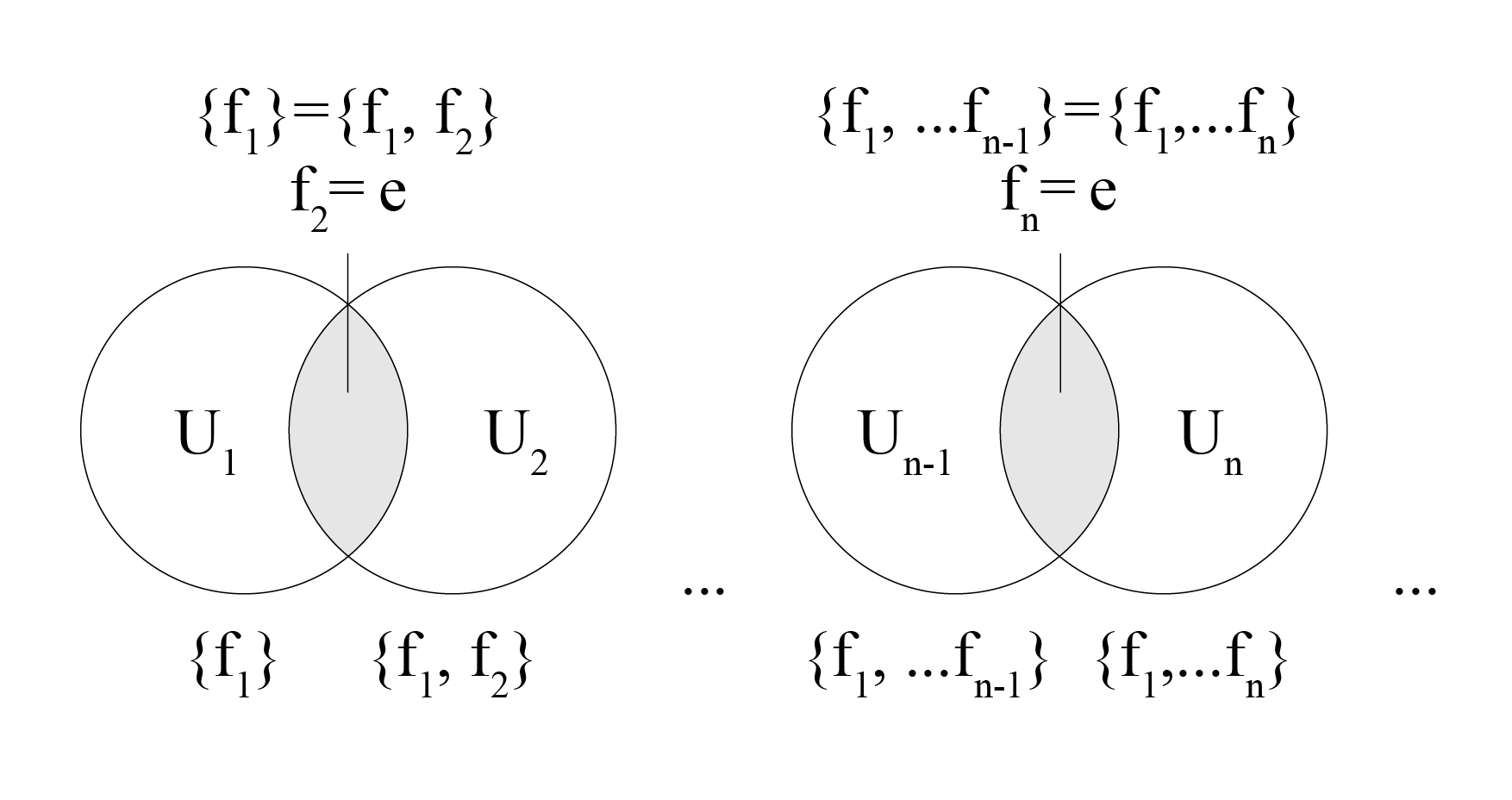}
\end{figure}

Consider an arbitrary element $F \in TCor(Y, X)$. Let $y \in Y$ be an arbitrary point. Then there exists a neighborhood $U \subset Y$ of $y$ and $n=n(y)$ (depends on $y$) maps
\begin{equation*}
f_1, \ldots, f_n: U \rightarrow X, \quad f_1, \ldots, f_n \not\equiv e
\end{equation*}
such that  $F|_U = \{f_1, \ldots, f_n\}$. The element $F$ provides a well defined map
\begin{equation*}
\begin{gathered}
f: Y \rightarrow SP(X, e) \\
f(y) = \{\ldots, e, f_1(y), \ldots, f_n(y), e, \ldots\}.
\end{gathered}
\end{equation*}

In Section $\ref{motivationdold}$ we constructed two chain complexes using $SP(X, e)$ and $\coprod_{k=1}SP^k(X)$. We already constrcuted an analogue of $SP(X, e)$ in terms of sheaves. The analogue of the union can be constructed in a similar way.

Let $U \subset Y$ be an open subset. Then, the assignment
\begin{equation*}
U \rightarrow \coprod_{k=1}TCor_k(U, X)
\end{equation*}
defines a presheaf on $Y$. The following simple example shows that this is not a sheaf. Let $U_1, U_2 \subset Y$ be open subsets such that $U_1 \cap U_2 = \emptyset$. Consider $F \in \coprod_{k=1}TCor_k(U_1, X)$ and $G \in \coprod_{k=1}TCor_k(U_2, X)$, where $n \neq k$. There is no element $H \in \coprod_{k=1}TCor_k(U_1 \cup U_2, X)$ such that $H|_{U_1} = F$ and $H|_{U_2} = G$, i.e. we can not glue elements $F$ and $G$.

We sheafify $\coprod_{k=1}TCor_k(\cdot, X)$ and note that sheafifed $\coprod_{k=1}TCor_k(Y, X)$ is an analogue of $Map(Y, \coprod_{k=1}SP^k)$ in terms of sheaves.

\vspace{0.1in}

We showed in this section that maps to symmetric products can be defined as global sections of some sheaf. In fact, the same technique can be applied to define symmetric products in the categories $\mathcal{S}\mathcal{S}ymp$ and $\mathcal{J}\mathcal{S}ymp$. As soon as we have the symmetric products, we can argue as in Section $\ref{motivationdold}$ to define (co)homology groups of symplectic manifolds.

\subsection{Symmetric products in $\mathcal{S}\mathcal{S}ymp$ and symplectic correspondences $SCor$}\label{sympcorrsheaves}

In this section we work with the category $\mathcal{S}\mathcal{S}ymp$. In the beginning of this section we argue as in Section $\ref{topcorr}$.

Recall that the symmetric products have singularities and are not symplectic manifolds. We argue as in Section $\ref{sheavesonsymp2}$ and represent symplectic manifolds as sheaves on $\mathcal{S}\mathcal{S}ymp$.

Let $X = (X, \omega_X)$, $Y = (Y, \omega_Y)$ be symplectic manifolds. We assume that $Y \times X = (Y \times X, \; \omega_Y \oplus \omega_X)$ and $X^n = (X^n, \oplus^n\omega_X)$, where $X^n$ is the $n$th power of $X$ and $n\geqslant 1$. Consider a symplectic embedding
\begin{equation*}
f: Y \rightarrow X^n, \quad f(y) = (f_1(y), \ldots, f_n(y)).
\end{equation*}
The symmetric group $\Sigma_n$ acts by symplectomorphisms on $X^n$ (permuting the factors) and as a result acts on $\mathcal{S}\mathcal{S}ymp(Y, X^n)$. The quotient $\mathcal{S}\mathcal{S}ymp(Y, X^n)/\Sigma_n$ consists of $n$ unordered maps $\{f_1, \ldots, f_n\}$ from $Y$ to $X$ such that
\begin{equation*}
f_1^{*}\omega_X + \ldots + f_n^{*}\omega_X = \omega_Y.
\end{equation*}
Two elements $\{f_1, \ldots, f_n\}, \{g_1, \ldots, g_n\} \in \mathcal{S}\mathcal{S}ymp(Y, X^n)/\Sigma_n$ are equal if
\begin{equation*}
\{f_1(y), \ldots, f_n(y) \} = \{ g_1(y), \ldots, g_n(y)  \} \;\;\; \text{for any $y \in Y$},
\end{equation*}
where the equality is equality of unordered sets. Let $h \in \mathcal{S}\mathcal{S}ymp(Z, Y)$ be a symplectic embedding. We define
\begin{equation*}
\begin{gathered}
h^{*}: \mathcal{S}\mathcal{S}ymp(Y, X^n)/\Sigma_n \rightarrow \mathcal{S}\mathcal{S}ymp(Z, X^n)/\Sigma_n, \\
h^{*}\{f_1, \ldots, f_n\} = \{f_1 \circ h, \ldots, f_n \circ h\}.
\end{gathered}
\end{equation*}
Since
\begin{equation*}
(f_1 \circ h)^{*}\omega_X + \ldots + (f_n \circ h)^{*}\omega_X = h^{*}\omega_Y = \omega_Z
\end{equation*}
and $(h \circ h')^{*} = (h')^{*} \circ h^{*}$ for any $h' \in \mathcal{S}\mathcal{S}ymp(Z', Z)$, we get a presheaf of sets on the category $\mathcal{S}\mathcal{S}ymp$
\begin{equation*}
\mathcal{S}\mathcal{S}ymp^{op} \rightarrow \mathcal{S}et, \quad Y \rightarrow \mathcal{S}\mathcal{S}ymp(Y, X^n)/\Sigma_n.
\end{equation*}
Arguing as in Section $\ref{topcorr}$, we can show that this presheaf is not a sheaf. We sheafify this presheaf (see Theorem $\ref{theoremsheafification}$) and denote
\begin{equation*}
SCor_n(Y, X) = Sheaf(\mathcal{S}\mathcal{S}ymp(\cdot, X^n)/\Sigma_n)(Y).
\end{equation*}

As in the previous section, we can show that any element of $SCor_n(Y, X)$ defines a subset of $Y \times X$. If $F \in SCor_n(Y, X)$, then any point $y \in Y$ has a small open $U \subset Y$ such that $F|_U = \{f_1, \ldots, f_n\}$. We define
\begin{equation*}
\bigcup_{y \in Y} \bigcup_{i=1}^n(y,  f_i(y)).
\end{equation*}
As before, $SCor_n(Y, X)$ has more geometric definition.

\begin{definition}\label{sympcorr}
Let $F \subset Y \times X$ be a subset and $\pi: F \rightarrow Y$ be the projection. We say that $F$ is a symplectic correspondence of order $n$ between $Y$ and $X$ if any point of $Y$ has a neighborhood $U \subset Y$ and $n$ maps $f_i: U \rightarrow X$ such that
\begin{equation}\label{corrdef}
\begin{gathered}
\pi^{-1}(U) = \bigcup\limits_{i=1}^n(U, f_i(U)), \;\;\;\; f_i \in C^{\infty}(U, X)\\
f_1^{*}\omega_X + \ldots + f_n^{*}\omega_X = \omega_Y.
\end{gathered}
\end{equation}
\end{definition}

If $f_1(y), \ldots, f_n(y)$ are distinct for any $y\in Y$, then the projection $F \rightarrow Y$ is an $n$-covering.

\begin{lemma}
There is a bijection between symplectic correspondences of order $n$ and elements of $SCor_n(Y, X)$
\end{lemma}

\begin{proof}
The proof mimics the proof of Lemma $\ref{corrsheafbij}$.
\end{proof}

We denote the symplectic correspondence associated to $F \in SCor_n$ by the same symbol $F$.

If  $F \in SCor_n(Y, X)$ is defined by unordered maps $f_1, \ldots, f_n$ from $Y$ to $X$ (defined globally on $Y$, not locally), then we denote
\begin{equation*}
F = \{f_1, \ldots, f_n\}.
\end{equation*}

\begin{example}
We see that $\mathcal{S}\mathcal{S}ymp(Y, X^n)/\Sigma_n \subset SCor_n(Y, X)$. Geometrically, elements of $\mathcal{S}\mathcal{S}ymp(Y, X^n)/\Sigma_n$ are union of $n$ graphs of maps from $Y$ to $X$.
\end{example}

The example below shows that we need to be careful thinking about elements of $SCor_n(Y, X)$ as subsets of $Y \times X$.

\begin{example}
Let $D \subset \mathbb{C}$ be the unit disk of the complex plane, where $\mathbb{C}$ is endowed by the standard symplectic form $\omega_{\mathbb{C}}$. Consider
\begin{equation*}
f_1, f_2: \; D \rightarrow D, \quad f_1(z) = f_2(z) = \frac{z}{\sqrt{2}}.
\end{equation*}
We see that $f_1^{*}\omega_{\mathbb{C}} + f_2^{*}\omega_{\mathbb{C}} = \omega_{\mathbb{C}}$. This shows that $\{f_1, f_2\} \in SCor_2(D, D)$. The graphs of $f_1$ and $f_2$ coincide and the subset $F \subset D \times D$ corresponding to $\{f_1, f_2\}$ is a graph of $f_1$.
\end{example}

The example above shows that if we want o replace $SCor_n$ by subsets, then we need to consider the subsets with `multiplicities' (this word does not have a rigorous meaning). That is the reason to fix $n$ in the definition of symplectic correspondences.

\begin{example}\label{corrpoint}
We assume that a point $pt$ is a symplectic manifold and $2$-form on $pt$ is equal to zero. Elements of $SCor_n(pt, X)$ are unordered set of points $\{x_1, \ldots, x_n\}$, where $x_i \in X$.
\end{example}

Let us define isotropic correspondences. Since the construction is very similar to the construction above, we give less details. We define a presheaf $Iso_{X^n}$ on $\mathcal{S}\mathcal{S}ymp$ such that $Iso_{X^n}(Y)$ is the following set of unordered maps
\begin{equation*}
\begin{gathered}
f_1, \ldots, f_n \in C^{\infty}(Y, X), \quad f_1^{*}\omega_X + \ldots + f_n^{*}\omega_X = 0.
\end{gathered}
\end{equation*}
Note that maps $f_1, \ldots, f_n$ are unordered. We sheafify $Iso_{X^n}$ and denote
\begin{equation*}
SCor^0_n(Y, X) = Sheaf(Iso_{X^n})(Y).
\end{equation*}

\begin{example}
Unordered constant maps $g_i : Y \rightarrow x_i$ define an isotropic correspondence $\{g_1, \ldots, g_n\} \in SCor^0_n(Y, X)$.
\end{example}

The goal is to consider symplectic correspondences up to isotropic correspondences. Consider $F \in Scor_n(Y, X), \;$ $G \in Scor^0_k(Y, X)$ and small open $U \subset Y$ such that  $F|_U = \{f_1, \ldots, f_n\}$ and $G|_{U} = \{g_1, \ldots, g_k\}$. Note that
\begin{equation*}
f_1^{*}\omega_X + f_n^{*}\omega_X + g_1^{*}\omega_X + \ldots + g_k^{*}\omega_X = \omega_Y + 0 = \omega_Y.
\end{equation*}

We define $F + G \in SCor_{n+k}(Y, X)$ as a symplectic correspondence such that
\begin{equation*}
(F + G)|_U = \{f_1, \ldots, f_n, g_1, \ldots, g_k \}.
\end{equation*}
Consider
\begin{equation*}
\widetilde{SCor}(Y, X) = \coprod_{n=0} SCor_n(Y, X).
\end{equation*}
We say that two symplectic correspondences $F_1\in SCor_n(Y,X)$ and $F_2 \in SCor_{n+k}(Y, X)$ are equivalent if there exists $G \in SCor^0_k(Y, X)$ such that
\begin{equation*}
F_2 = F_1 + G.
\end{equation*}
Let us consider $\widetilde{SCor}(\cdot, X))/\sim$, where $\sim$ is the mentioned equivalence relation.

\begin{example}
Assume that $f: Y \rightarrow X$ is a symplectic embedding and $x_1, x_2 \in X$ are considered as  constant maps from $Y$ to $X$. Note that $\{x_1, x_2\} \in SCor^0_2(Y, X)$. Then, $\{f, x_1, x_2\} \in Scor_3(Y, X)$ and $\{f, x_1, x_2\} = \{f\}$ as elements of $\widetilde{SCor}(Y, X))/\sim$.
\end{example}

\begin{lemma}\label{presheaflemmasympcor}
\begin{equation*}
Y \rightarrow (\widetilde{SCor}(Y,X))/\sim.
\end{equation*}
is a presheaf on $\mathcal{S}\mathcal{S}ymp$.
\end{lemma}
\begin{proof}
We need to show that $\widetilde{SCor}(Y,X)/\sim$ is a functor from $\mathcal{S}\mathcal{S}ymp^{op}$ to the category of sets. Consider $h \in \mathcal{S}\mathcal{S}ymp(Z, Y)$ and a symplectic correspondence $F$ between $Y$ and $X$. Let $z \in Z$ be an arbitrary point and let $V \subset Z$ be a small neighborhood of $z$. There exists a neighborhood $U \subset Y$ of $h(z)$ such that $F|_U = \{f_1, \ldots, f_n\}$. If $V$ is small, then $h(V) \subset U$. We define $h^{*}F|_V = \{f_1 \circ h, \ldots, f_n \circ h\}$. Consider $F + G$, where $G \in SCor^0_k(Y, X)$. Then, $(F + G)|_U =  \{f_1, \ldots, f_n, g_1, \ldots, g_k\}$ and  we get
\begin{equation*}
\begin{gathered}
h^{*}(F + G)|_U = \{f_1 \circ h, \ldots, f_n \circ h, g_1 \circ h, \ldots, g_k \circ h\} = \\
\{f_1 \circ h, \ldots, f_n \circ h\} + \{g_1 \circ h, \ldots, g_k \circ h\}   \Rightarrow \\
h^{*}(F + G) =  h^{*}F + h^{*}G \sim h^{*}F, \quad h^{*}G \in SCor^0_k(Z, X).
\end{gathered}
\end{equation*}
This means that $h^{*}$ is well-defined on $\widetilde{SCor}(Y, X)/\sim$ and we have a map
\begin{equation*}
h^{*}: \widetilde{SCor}(Y, X)/\sim \;\;\; \rightarrow \;\;\;  \widetilde{SCor}(Z, X)/\sim.
\end{equation*}
If $h' \in \mathcal{S}\mathcal{S}ymp(Z', Z)$,  then in the similar way we define $(h')^{*}$ and $(h \circ h')^{*}$. Moreover, the definition shows that $(h \circ h')^{*} = (h')^{*} \circ h^{*}$. This proves that $\widetilde{SCor}(\cdot, X)$ is a presheaf on $\mathcal{S}\mathcal{S}ymp$.

\end{proof}

Let us show that $\widetilde{SCor}(Y,X))/\sim$ is not a sheaf. In fact, we do not need to show it and we can just sheafify the presheaf. Sheafification of a sheaf gives us the same sheaf. However, it is useful to prove that the presheaf is not a sheaf because it can help to undestand elements of the resulting sheaf.

We argue as in Section $\ref{topcorr}$, we consider infinite number of open subsets $U_1, U_2 \ldots \subset Y$ such that $\cup U_i = Y$ and $U_i \cap U_j \neq \emptyset$ if and on $|i-j| \leqslant 1$. Define $F_n \in \widetilde{SCor}(U_n, X)/\sim \;$ as follows:
\begin{equation*}
\begin{gathered}
F_n = \{f_{n,1}, \ldots, f_{n,n}\} \in \widetilde{SCor}(U_n, X)/\sim, \\
\text{$f_{n, i}$ is not isotropic on $U_n$  for any $i$}.
\end{gathered}
\end{equation*}
Assume that $f_{n,n}$ is isotropic on $U_{n-1} \cap U_n$, but not on $U_n$. Suppose that $F_{n-1} \sim F_{n}$ on $U_{n-1} \cap U_{n}$ (they are equal up to isotropic correspondence). We can glue $F_{n-1}$ and $F_{n}$ to get a correspondence on $U_{n-1} \cup U_{n}$. Let us glue all the topological correspondences $F_n$ and denote the final element by $F$. Since there is no number $m$ such that $F$ is defined by $m$ maps on each $U_n$, we get that $F$ does not belong to $\widetilde{SCor}(Y,X))/\sim$.

\begin{definition}
We sheafify (see Theorem $\ref{theoremsheafification}$) the presheaf $\widetilde{SCor}(\cdot, X)/\sim$ and denote
\begin{equation*}
SCor(Y, X) = Sheaf(\widetilde{SCor}(\cdot, X)/\sim)(Y).
\end{equation*}
\end{definition}

Let $F \in SCor(Y, X)$ be an arbitrary element. For any $y \in Y$ there is a neighborhood $U$ of $y$ and a number $n = n(y)$ (depends on $y$) such that
\begin{equation*}
F|_U = \{ f_1, \ldots, f_n \} \;\;\; \; \text{modulo isotropic correspondence}.
\end{equation*}
We say that a correspondence is bounded if there exists $N$ such that $n(y) < N$ for any $y$.

\begin{definition}\label{boundedcorr}
We denote the set of bounded correspondences by $SCor_b(Y, X)$.
\end{definition}
Note that there is a filtration on $SCor_b(Y, X)$. We define
\begin{equation}\label{filtration1}
SCor(Y, X; r) = \{\text{bounded correspondences with $n(y) \leqslant e^r$ for any $r$}\}
\end{equation}
The reason to consider $e^r$ is explained in Section $\ref{operationcorr}$. Note that $SCor(Y, X; r)$ is symplectic analogue of maps from $Y$ to $SP^{e^r}(X)$. Obviously, for compact $Y$ and $X$ we have $Scor_b(Y, X) = SCor(Y, X)$.

\begin{remark}
As we can see from Section $\ref{topcorr}$, the set $SCor(Y, X)$ is a symplectic analogue of maps $Y \rightarrow SP(X, e)$ and the sheaf $SCor(\cdot, X)$ is analogue of $SP(X,e)$.
\end{remark}

\begin{example}\label{corrpoint}
We discussed above that the set $SCor_n(pt, X)$ consists of unordered set of points $\{x_1, \ldots, x_n\}$. Note that maps $pt \rightarrow x_i$ are isotropic and symplectic correspondences. This means that,
\begin{equation*}
\{x_1, \ldots, x_n\} = \{x_1\} + \{x_2 + \ldots + x_n\} = \{x_1\} \;\;\; \text{as elements of $SCor$}.
\end{equation*}
This shows that $Scor(pt, X)$ consists of one element. It represens the class of points.
\end{example}

\begin{example}
Assume that $dim(Y)>0$. We see that
\begin{equation*}
\mathcal{S}Symp(Y, X^n)/\Sigma_n \subset SCor(Y, X), \quad \text{for any $n$ }.
\end{equation*}
In particular, $\mathcal{S}\mathcal{S}ymp(Y, X) \subset SCor(Y, X)$.
\end{example}

\begin{example}
There is no symplectic embedding of $(X \times X, \omega_X \oplus \omega_X)$ into $(X, \omega_X)$. Let $pr_1$, $pr_2$ be the projections of $X \times X$ onto the first and second factors, respectively.  The identity map of $X \times X$ can be written as $(pr_1, pr_2)$. This means that the unordered pair $\{pr_1, pr_2\}$ belongs to $Cor_2(X \times X, X) \subset SCor(X \times X, X)$.
\end{example}

\begin{example}\label{prodembedd}
If there exists symplectic embedding $f = (f_1, \ldots, f_n) \in \mathcal{S}\mathcal{S}ymp(Y, X^n)$, then $SCor(Y, X) \neq \emptyset$. Indeed, unordered set of maps $\{f_1, \ldots, f_n\} \in \mathcal{S}\mathcal{S}ymp(Y, X^n)/\Sigma_n \subset SCor_n(Y, X)$. As result $f$ belongs to $SCor(Y, X)$.
\end{example}

\noindent
There is the following theorem of Gromov:

\begin{theorem}(\cite[Corollary B, page 87]{Gromov})
If $Y$ is contractible and $dim(Y)<dim(Z)$, then every embedding (not necessarily symplectic) $Y \rightarrow Z$ is isotopic to a symplectic one.
\end{theorem}

This implies the following:

\begin{lemma}\label{contremb}
If $Y$ is contractible and there exists smooth (not necessarily symplectic) embedding of $Y$ into $X^n$, then $SCor(Y, X) \neq \emptyset$ (even if $dim(Y)> dim(X)$). In particular, $SCor(\mathbb{C}, X) \neq \emptyset$, $SCor(D, X) \neq \emptyset$ for any $X$, where $D$ is a ball, polydisk or similar manifold.
\end{lemma}

\begin{proof}
Lemma immediately follows from the theorem above and example $\ref{prodembedd}$.
\end{proof}

\begin{lemma}\label{exactcorr}
If $X$ is exact, $Y$ is not exact and $dim(Y) > 0$, then $SCor(Y, X) = \emptyset$.
\end{lemma}

\begin{proof}
Consider $F \in SCor(Y, X)$ and a small open subset $U \subset Y$ such that $F|_U = \{f_1, \ldots, f_n\}$. Since $X$ is exact, there exists $1$-form such that $\omega_X = d\alpha_X$. Then there is a well-defined $1$-form $\alpha_Y$ locally given by the following formula:
\begin{equation*}
\alpha_Y|_U = f_1^{*}\alpha_X|_{U} + \ldots + f_n^{*}\alpha_X|_U.
\end{equation*}
Moreover, $d\alpha_Y = \omega_Y$. This implies that $\omega_Y$ is exact, but $Y$ is not exact. Therefore $F$ does not exist.

\end{proof}

\subsection{Symmetric products in $\mathcal{J}\mathcal{S}ymp$ and pseudoholomorphic correspondences $JCor$}\label{pseudoholcorr}

We apply the ideas developed in Section $\ref{topcorr}$ to the category $\mathcal{J}\mathcal{S}ymp$ (as we do in Section $\ref{sympcorrsheaves}$).

\begin{remark}
In Section $\ref{motivationdold}$ we used two chain complexes to define singular homology groups. The first one was defined using infinite symmetric product and the second one using the disjoint union of symmetric products. To define cohomology groups of almost complex manifolds we use the disjoint union and do not need to take quotients. So, this section is technically easier.
\end{remark}

\noindent
Let
\begin{equation*}
\begin{gathered}
X = (X, \omega_X, J_X), \;\;\;\; Y= (Y, \omega_Y, J_Y), \\
Y \times X = (Y \times X, \omega_Y \oplus \omega_X, J_X \times J_Y), \;\;\;\; X^n = (X^n, \oplus^n\omega_X, \times^n J_X,)
\end{gathered}
\end{equation*}
where $X^n$ is the $n$th power of $X$. Consider the set $\mathcal{J}\mathcal{S}ymp(Y, X^n)$ of all pseudoholomorphic maps from $Y$ to $X^n$. The permutation group $\Sigma_n$ acts on $X^n$ permuting the factors and acts on $\mathcal{J}\mathcal{S}ymp(Y, X^n)$. Let us take the quotient $\mathcal{J}\mathcal{S}ymp(Y, X^n)/\Sigma_n$. An element of the quotient consists of $n$ unordered pseudoholomorphic maps $\{f_1, \ldots, f_n\}$ from $Y$ to $X$. Two elements $\{f_1, \ldots, f_n\}, \{g_1, \ldots, g_n\} \in \mathcal{J}\mathcal{S}ymp(Y, X^n)/\Sigma_n$ are equal if
\begin{equation*}
\{f_1(y), \ldots, f_n(y) \} = \{ g_1(y), \ldots, g_n(y)  \} \;\;\; \text{for any $y \in Y$},
\end{equation*}
where the equality is equality of unordered sets.

Let $h \in \mathcal{J}\mathcal{S}ymp(Z, Y)$ be a pseudoholomorphic map. We define
\begin{equation*}
\begin{gathered}
h^{*}: \mathcal{J}\mathcal{S}ymp(Y, X^n)/\Sigma_n \rightarrow \mathcal{J}\mathcal{S}ymp(Z, X^n)/\Sigma_n, \\
h^{*}\{f_1, \ldots, f_n\} = \{f_1 \circ h, \ldots, f_n \circ h\}.
\end{gathered}
\end{equation*}
This defines a presheaf on $\mathcal{J}\mathcal{S}ymp$
\begin{equation*}
\mathcal{J}\mathcal{S}ymp^{op} \rightarrow  \mathcal{S}et, \quad Y \rightarrow \mathcal{J}\mathcal{S}ymp(Y, X^n)/\Sigma_n.
\end{equation*}
We sheafify (see Theorem $\ref{theoremsheafification}$) this presheaf  and denote
\begin{equation*}
JCor_n(Y, X) = Sheaf(\mathcal{J}\mathcal{S}ymp(\cdot, X^n)/\Sigma_n)(Y).
\end{equation*}
We assume that $JCor_0(Y, X)$ is a sheaf defined locally by zero maps and define
\begin{equation*}
JCor_0(Y, X) = \emptyset.
\end{equation*}
If $F \in JCor_n(Y, X)$, then any point $y \in Y$ has a small open $U \subset Y$ such that $F$ over $U$ is given by $n$ unordered maps $\{f_1, \ldots, f_n\}$. We define
\begin{equation*}
\bigcup_{y \in Y} \bigcup_{i=1}^n(y, f_i(y)).
\end{equation*}
So, element of $JCor_n(Y, X)$ defines a subset of $Y \times X$. As in the previous sections, $JCor_n(Y, X)$ has more geometric definition.

\begin{definition}\label{sympcorr}
Let $F \subset Y \times X$ be a subset and $\pi: F \rightarrow Y$ be the projection. We say that $F$ is a pseudoholomorphic correspondence of order $n$ between $Y$ and $X$ if any point of $Y$ has a neighborhood $U \subset Y$ and $n$  pseudoholomorphic maps $f_i: U \rightarrow X$ such that
\begin{equation*}
\pi^{-1}(U) = \bigcup\limits_{i=1}^n(U, f_i(U)), \quad f_1, \ldots, f_n \in \mathcal{J}\mathcal{S}ymp(Y, X).
\end{equation*}
We define the empty set as an pseudoholomorphic correspondence of order $0$.
\end{definition}

Note that if $f_1(y), \ldots, f_n(y)$ are distinct for any $y\in Y$, then the projection $F \rightarrow Y$ is an $n$-covering.

\begin{lemma}
There is a bijection between pseudoholomorphic correspondences of order $n$ and elements of $JCor_n(Y, X)$
\end{lemma}

\begin{proof}
The proof mimics the proof of Lemma $\ref{corrsheafbij}$.
\end{proof}

If  $F \in JCor_n(Y, X)$ is defined by unordered maps $f_1, \ldots, f_n$ from $Y$ to $X$ (defined globally on $Y$, not locally), then we denote
\begin{equation*}
F = \{f_1, \ldots, f_n\}.
\end{equation*}
We consider
\begin{equation*}
\widetilde{JCor}(Y, X) = \coprod_{n=0}JCor_n(Y, X),
\end{equation*}
where $\coprod$ is the disjoint union.

Assume that $Y$ is not connected and $Y_1, Y_2$ are its connected components. Consider $F_1 \in SCor_{n_1}(Y_1, X)$ and $F_2 \in SCor_{n_2}(Y_2, X)$, where $n_1 \neq n_2$. We see that $F_1$ and $F_2$ agree on the intersection of $Y_1, Y_2$, but there is no $F \in \widetilde{SCor}(Y, X)$ such that $F|_{Y_1} = F_1$ and $F|_{Y_2} = F_2$. This shows that $\widetilde{JCor}(\cdot, X)$ is not a presheaf.

We apply the sheafification functor $Sheaf$ (see Theorem $\ref{theoremsheafification}$)  and define
\begin{equation*}
JCor(Y, X) = Sheaf(\widetilde{JCor}(\cdot, X))(Y).
\end{equation*}

Let $F \in JCor(Y, X)$ be an arbitrary element. For any $y \in Y$ there exists a neighborhood $U$ of $y$ and a number $n = n(y)$ (depends on $y$) such that $F|_U = \{f_1, \ldots, f_n\}$.

\begin{example}
Since $JCor_1(Y, X) = \mathcal{J}\mathcal{S}ymp(Y, X)$, we obtain that all standard pseudoholomorphic maps belong to $JCor(Y, X)$. In particular, all constant maps belong to $JCor(Y, X)$.
\end{example}

\begin{example}
Let $S_g$ be a surface of genus $g$, $X$ be a K$\ddot{a}$hler manifold and $f: S_g \rightarrow X$ be a holomorphic map. Consider a meromorphic function $\pi: S_g \rightarrow \mathbb{C}P^1$. Note that $\pi$ defines a branched covering of $\mathbb{C}P^1$. Define a subset of $\mathbb{C}P^1 \times X$
\begin{equation*}
F = \bigcup_{p \in \mathbb{C}P^1}\bigcup_{y \in \pi^{-1}(p)}(p, f(y)).
\end{equation*}
If $p \in \mathbb{C}P^1$ is not a branch point of $\pi$, then there is a small neighborhood $U$ of $p$ such that $\pi^{-1}(U) = \sqcup_{i=1}^nV_i$. Then, $F|_U = \{f|_{V_1}, \ldots, f|_{V_n}\}$. If $p \in \mathbb{C}P^1$ is a branch point and $\pi^{-1}(p) = \{q_1, \ldots, q_m\}$, where $q_i$ has multiplicity $k_i$. In a small neighborhood $U$ of $p$ we have $\pi^{-1}(U) = \sqcup_{i=1}^mV_i$.
\begin{equation*}
F|_U = \{\underbrace{f|V_1, \ldots, f|_{V_1}}_{k_1}, \ldots, \underbrace{f|_{V_m}, \ldots, f|_{V_m}}_{k_m}   \}.
\end{equation*}
This shows that $F$ is a pseudoholomorphic correspondence between $\mathbb{C}P^1$ and $X$.
\end{example}

The set $JCor_n(Y,X)$ is `pseudoholomorphic analogue' of  the set of maps from $Y$ to the symmetric product  $SP^n(X)$. The sheaf $JCor(\cdot, X)$ is `pseudoholomorphic analogue' of $SP^n(X)$.

\section{Abelian groups of symmetric products $SC$ and $JC$}\label{abelcorr}

In this Section we construct analogue of the Grothendieck groups $Map(Y, SP(X,e))^{+}$ (see Section $\ref{motivationdold}$).

\subsection{Group of pseudoholomorphic correspondences $JC$}\label{abelpseudcorr}

Consider
\begin{equation*}
f = (f_1, \ldots, f_n) \in \mathcal{J}\mathcal{S}ymp(Y, X^n), \quad g = (g_1, \ldots, g_k) \in \mathcal{J}\mathcal{S}ymp(Y, X^k).
\end{equation*}
We can combine these maps and get
\begin{equation*}
\begin{gathered}
(f,g) = (f_1, \ldots, f_n, g_1, \ldots, g_k): \;\; Y \rightarrow X^{n+k}, \quad (f, g) \in \mathcal{J}\mathcal{S}ymp(Y, X^{n+k}), \\
(f, g)(y) = (f(y), g(y)).
\end{gathered}
\end{equation*}
If we have unordered sets of pseudoholomorphic maps
\begin{equation*}
f=\{f_1, \ldots, f_n\} \in \mathcal{J}\mathcal{S}ymp(Y, X^n)/\Sigma_n, \;\; g=\{g_1, \ldots, g_k\} \in  \mathcal{J}\mathcal{S}ymp(Y, X^k)/\Sigma_k,
\end{equation*}
then we define a sum $f+g$ as combination of these maps (they are unordered) into an unordered set
\begin{equation*}
f + g = \{f_1, \ldots, f_n, g_1, \ldots, g_k\} \in \mathcal{J}\mathcal{S}ymp(Y, X^{n+k})/\Sigma_{n+k}.
\end{equation*}
We give the following definition:

\begin{definition}
Consider correspondences $F, G \in JCor(Y, X)$. Let $y \in Y$ be an arbitrary point and $U$ be a small neighborhood of $y$ such that $F|_U = \{f_1, \ldots, f_n\}$ and $G|_U = \{g_1, \ldots, g_k\}$. We define $F+G$ as a sheaf locally defined as follows
\begin{equation*}
(F + G)|_U = \{f_1, \ldots, f_n, g_1, \ldots, g_k \}.
\end{equation*}
Recall that the emptyset is pseudoholomorphic correspondence of order $0$ and we get that $F + \emptyset = F$.
\end{definition}

We see that $(JCor(Y, X), +)$ is associative, commutative monoid (a group without inverses), where the zero elements is the empty set. Then, there is the Grothedieck group of the monoid.

\begin{definition}
Let $JCor(Y, X)^{+}$ be the Grothedieck group of the monoid $(JCor(Y, X), +)$. We denote
\begin{equation*}
JC(Y, X) = JCor(Y, X)^{+}.
\end{equation*}
Note that $JC(Y, X)$ is abelian.
\end{definition}

\begin{remark}
There is another way to define $JC(Y, X)$. Let us say that $F \in JCor(Y, X)$ is irreducible if $F \neq F_1 + F_2$, where $F_1, F_2 \neq \emptyset$. Denote the set of irreducible correspondences by $Ir(JCor(Y, X))$. Consider the set of formal finite linear combinations $\mathbb{Z}[Ir(JCor(Y, X))]$. Then,
\begin{equation*}
JC(Y, X) = \mathbb{Z}[Ir(JCor(Y, X))]/\mathbb{Z}[\emptyset].
\end{equation*}
This definition requires some technical lemmas. We do not want to go into details and want to avoid this definition.
\end{remark}

\begin{example}
Assume that $Y = pt$. Any element $F \in JCor(pt, X)$ has the form $F = \{x_1, \ldots, x_n\} = \{x_1\} + \ldots + \{x_n\}$. This shows that $JC(pt, X)$ is isomorphic to a group $\mathbb{Z}[X]$ of finite formal linear combinations of points of $X$.
\end{example}

\begin{example}
Note that $\mathcal{J}\mathcal{S}ymp(Y, X) \subset JCor(Y, X) \subset JC(Y, X)$.
\end{example}

Let us prove the following lemma:

\begin{lemma}
Consider $h \in \mathcal{J}\mathcal{S}ymp(Z, Y)$. Then, $h^{*}(F +G) = h^{*}F + h^{*}G$. This shows that $h^{*}: JCor(Y, X) \rightarrow JCor(Z, X)$ defines a homomorphism of monoids.
\end{lemma}

\begin{proof}
By definition $F + G$ locally is defined by unordered maps $\{f_1, \ldots, f_n, g_1, \ldots, g_k\}$ and $h^{*}(F + G)$ locally is defined by maps
\begin{equation*}
\{f_1 \circ h, \ldots, f_n \circ h, g_1 \circ h, \ldots, g_k \circ h \} = \{f_1 \circ h, \ldots, f_n \circ h\} + \{g_1 \circ h, \ldots, g_k \circ h  \}.
\end{equation*}
This implies that globally we have $h^{*}(F + G) = h^{*}F + h^{*}G$.

\end{proof}

\begin{lemma}
Any pseudoholomorphic map $h \in \mathcal{J}\mathcal{S}ymp(Z, Y)$ defines a homomorphism of groups
\begin{equation*}
h^{*}: \;\;\; JC(Y, X) \rightarrow JC(Z, X).
\end{equation*}
\end{lemma}

\begin{proof}
This immediately follows from the previous lemma and the definition of the Grothendieck group.
\end{proof}

Roughly speaking, in the previous lemma we define $h^{*}(-F) = -h^{*}F$.

\subsection{Group of symplectic correspondences $SC$}\label{abelsympcorr}

Definitions given in the previous section do not make sense in $SCor$. Consider
\begin{equation*}
f = (f_1, \ldots, f_n) \in \mathcal{S}\mathcal{S}ymp(Y, X^n), \quad g = (g_1, \ldots, g_k) \in \mathcal{S}\mathcal{S}ymp(Y, X^k).
\end{equation*}
Then
\begin{equation*}
\begin{gathered}
(f, g) \not\in \mathcal{S}\mathcal{S}ymp(Y, X^{n+k}), \\
f^{*}_1\omega_X + \ldots + f_n^{*}\omega_X + g_1^{*}\omega_X + \ldots + g_k^{*}\omega_X = 2\omega_Y \neq \omega_Y.
\end{gathered}
\end{equation*}
This shows that if $F, G \in SCor(Y, X)$, then $F + G \not\in SCor(Y, X)$.

\begin{definition}\label{grousympcorr}
Let $\mathbb{Z}[SCor(Y, X)]$ be a group of finite formal linear combinations. We define a group
\begin{equation*}
SC(Y, X) = \mathbb{Z}[SCor(Y, X)].
\end{equation*}
If $SCor(Y, X) = \emptyset$, then we assume that $SC(Y, X) = \mathbb{Z}[\emptyset] = 0$.
\end{definition}

\begin{definition}\label{abboundsymcorr}
We also have an abelian group $SC_b(Y, X) = \mathbb{Z}[SCor_b(Y, X)]$ of bounded correspondences.
\end{definition}

We defined filtration on $SCor_b(Y, X)$ by formula $(\ref{filtration1})$. This filtration can be  extended to filtration on $SC_b(Y, X)$. We define
\begin{equation}\label{filtration}
SC_b(Y, X; r) = \mathbb{Z}[SCor_b(Y,X; r)].
\end{equation}

\vspace*{0.06in}

Consider $h \in \mathcal{S}\mathcal{S}ymp(Z, Y)$. We define $h^{*}: SC(Y,X) \rightarrow SC(Z, X)$ by
\begin{equation*}
h^{*}(r_1F + r_2G) = r_1h^{*}F + r_2h^{*}G.
\end{equation*}
We get the following lemma

\begin{lemma}
The group $SC(\cdot, X)$ is a presheaf on $\mathcal{S}\mathcal{S}ymp$.
\end{lemma}

\begin{proof}
This follows immediately from the definition.
\end{proof}

\section{Some operations with symmetric products}\label{operationsection}

We give some definitions simultaneously for the categories $\mathcal{S}\mathcal{S}ymp$ and $\mathcal{J}\mathcal{S}ymp$. We denote by $Cor$ either $SCor$ or $JCor$.  So,  the reader can replace $Cor$ by $SCor$ or $JCor$.

\vspace{0.1in}
\noindent
In this section we construct the following operations:

\vspace{0.07in}
\noindent
1. Consider $F \in Cor(Z, Y)$, $G \in Cor(Y, X)$. We define a composition $G \circ F \in Cor(Z, X)$. Moreover, the composition can be extended to
\begin{equation*}
SC(Y, X) \circ SC(Z, Y) \rightarrow SC(Z, X), \quad JC(Y, X) \circ JC(Z, Y) \rightarrow JC(Z, X).
\end{equation*}

\vspace*{0.07in}
\noindent
2. Consider $F \in SCor(Y, X)$, $G \in SCor(Z, X)$. There is a wedge sum operation $F \vee G \in SCor(Y \times Z, X)$. The wedge sum can be extended
\begin{equation*}
SC(Y, X) \vee SC(Z, X) \rightarrow SC(Y \times Z, X).
\end{equation*}

\subsection{Composition of correspondences}\label{operationcorr}

\textbf{Topological motivation}. Let $Z, Y, X$ be some nice topological spaces. Consider symmetric products $SP^k(Y)$, $SP^n(X)$ and continuous maps
\begin{equation*}
\begin{gathered}
f: Z \rightarrow SP^k(Y), \quad g: Y \rightarrow SP^n(X), \\
f(z) = \{f_1(z), \ldots, f_k(z)\}, \quad g(y) = \{g_1(y), \ldots, g_n(y)\}.
\end{gathered}
\end{equation*}
We define
\begin{equation*}
g \circ f: Z \rightarrow SP^{nk}(X), \;\;\; g \circ f(z) = \{g_1 \circ f_1(z), g_2 \circ f_1, \ldots, g_k \circ f_k(z)\},
\end{equation*}
i.e. we  consider all possible compositions of $f_i$ and $g_j$.

Let us recall that the maps $f$, $g$, $, g\circ f$ define subsets $F \subset Z \times Y$, $G \subset Y \times X$ and $G \circ F \subset Z \times X$, respectively. Moreover, the projections $F \rightarrow Z$, $G \rightarrow Y$, $G \circ F \rightarrow Z$ are surjective. The set $G \circ F \subset Z \times X$ has a geometric illustration (see the picture below).

\begin{figure}[h]
\includegraphics[width=0.3\linewidth]{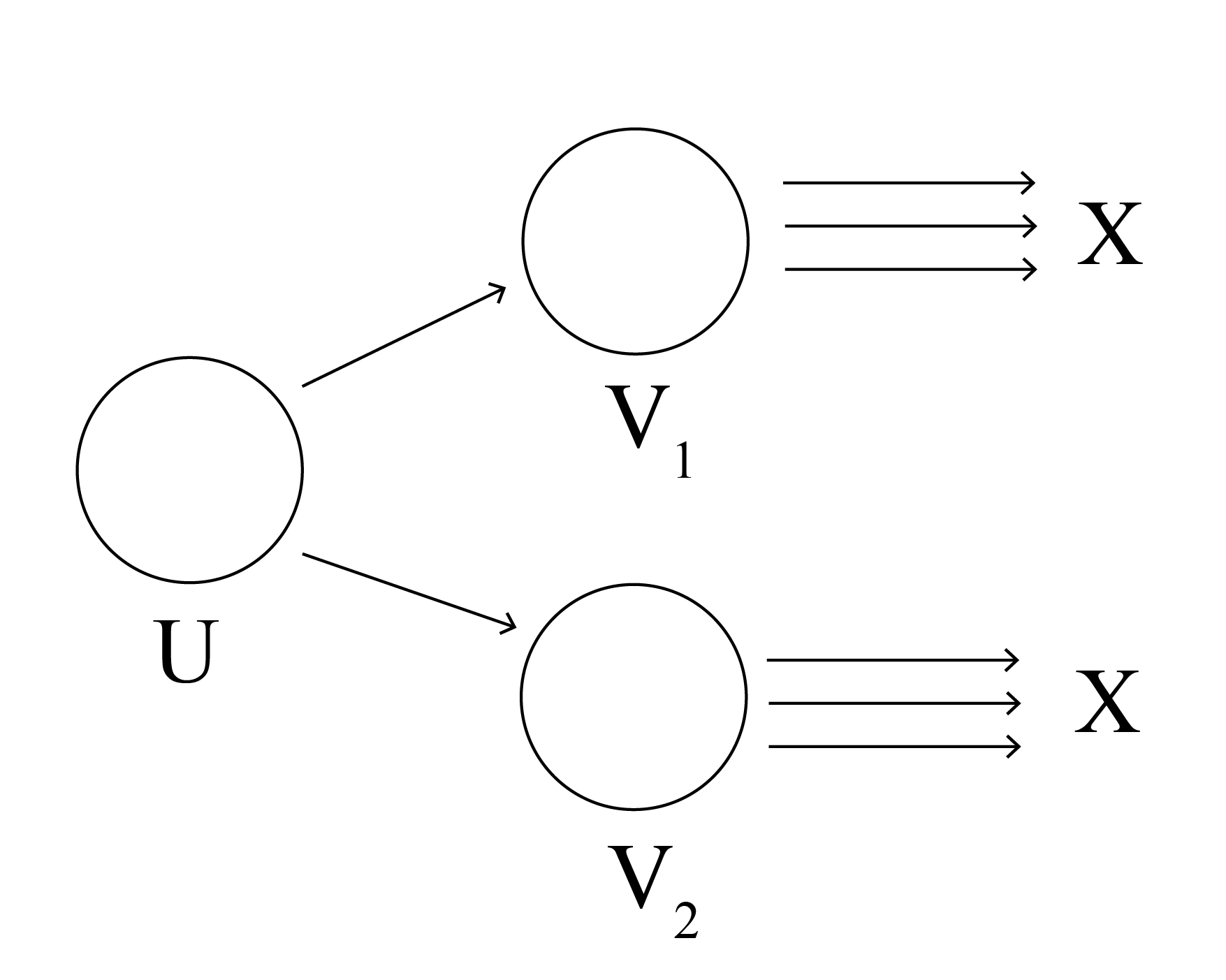}
\includegraphics[width=0.8\linewidth]{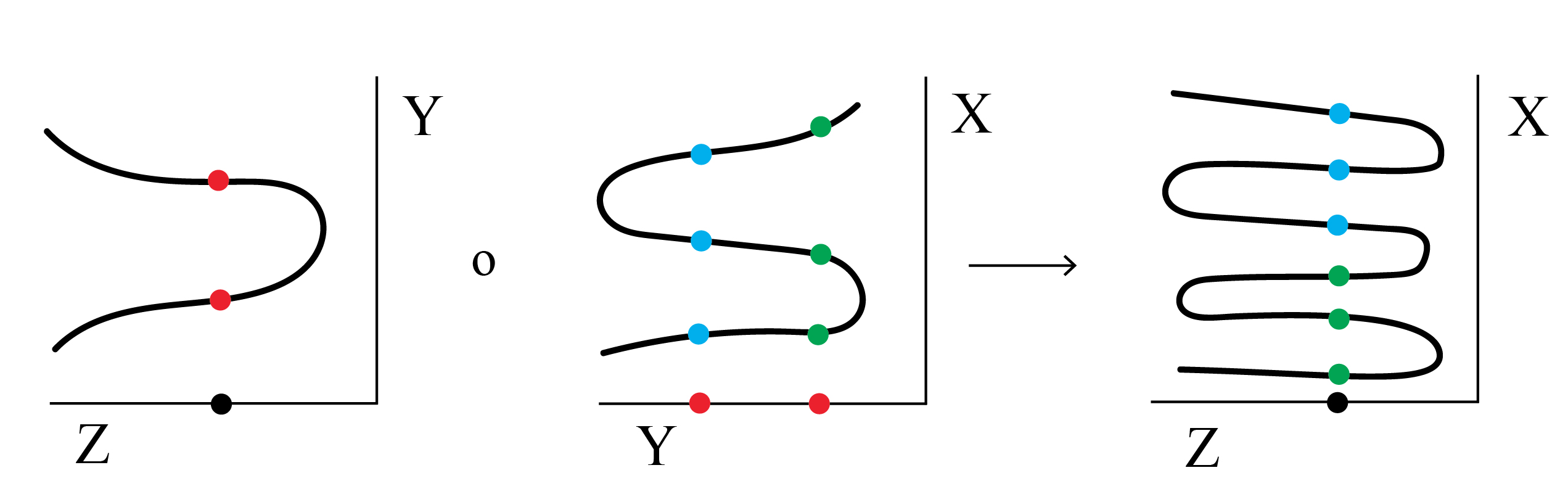}
\end{figure}

Our correspondences $SCor$, $JCor$ are analogues of symmetric products and in this section we define their composition.

\vspace{0.11in}
\noindent
\textbf{Composition of $JC$ and $SC$.} Consider $F \in Cor_k(Z, Y)$, $G \in Cor_n(Y, X)$. Let $U \subset Z$ be a small open subset such that $F|_U = \{f_1, \ldots, f_k \}$. Let $V_i \subset Y$ be small open subsets such that $f_i(U) \subset V_i$ and $G|_{V_i} = \{g_{i,1}, \ldots, g_{i, n}\}$. We define $G \circ F \in Cor_{nk}(Z, X)$ as a correspondence locally given by the following unordered maps
\begin{equation*}
(G \circ F)|_U = \{ \ldots, g_{i, j} \circ f_i, \ldots \}, \;\;\; \text{all possible compositions.}
\end{equation*}
In case of $Scor$, we see that
\begin{equation*}
\sum\limits_{i,j}(g_{i, j}\circ f_i)^{*}\omega_X = \sum\limits_{i}f_i^{*}\omega_Y = \omega_Z.
\end{equation*}
There is nothing to check in case of $JCor$ because composition of pseudoholomorphic maps is pseudoholomorphic.

\vspace*{0.1in}

As a result, we get
\begin{equation*}
Cor_k(Y, X) \circ Cor_n(Z, Y) \rightarrow Cor_{nk}(Z, X).
\end{equation*}

Let us extend our definition to $JCor$ and $SCor$

\begin{lemma}
There is a well-defined composition
\begin{equation*}
JCor(Y, X) \circ JCor(Z, Y) \rightarrow JCor(Z, X).
\end{equation*}
\end{lemma}

\begin{proof}
Consider $F \in JCor(Z, Y)$ and $G \in JCor(Y, X)$. Let $U \subset Z$, $V_i \subset Y$ be small open subsets such that $f_i(U) \subset V_i$ and
\begin{equation*}
F|_U = \{f_1, \ldots, f_k\}, \quad G|_{V_i} = \{g_{i, 1}, \ldots, g_{i, n_i}\}.
\end{equation*}
Since $V_i$ may belong to different connected components, we get that numbers $n_i$ are not supposed to be equal (see the definition of $JCor$).  We give the same definition
\begin{equation*}
(G \circ F)|_U = \{ \ldots, g_{i, j} \circ f_i, \ldots \}.
\end{equation*}
We define $G \circ F$  as a sheaf on $Z$ locally defined by the maps given above.

\end{proof}

Let us define the composition of elements of $SCor$. Since we take quotient by isotropic correspondences to define  $Scor$,  we need the following lemma.

\begin{lemma}\label{compwelldefined}
Consider $F \in SCor_k(Z, Y)$, $G \in SCor_n(Y, X)$. The following holds
\begin{equation*}
\begin{gathered}
(G + \; \text{isotropic correspondence}) \circ F = G \circ F + \; \text{isotropic correspondence} \\
G  \circ (F + \; \text{isotropic correspondence}) = G \circ F + \; \text{isotropic correspondence},
\end{gathered}
\end{equation*}
\end{lemma}

\begin{proof}
Consider an isotropic correspondence $H \in SCor^0_m(Z, Y)$. The correspondence  $F + H \in SCor_{k+m}(Z, Y)$ over some small $U \subset Z$ is given by
\begin{equation*}
(F + H)|_U = \{f_1, \ldots, f_k, h_1, \ldots, h_m\}.
\end{equation*}
Consider an open $W_{\ell} \subset Y$ such that $h_{\ell}(U) \subset W_{\ell}$ (we assume that $U$ and $W_{\ell}$ are small enough and $f_i(U) \subset V_i$) and
\begin{equation*}
G|_{W_{\ell}} = \{g_{\ell, 1}', \ldots, g_{\ell, n}'\}.
\end{equation*}
Then, by definition we have
\begin{equation*}
\begin{gathered}
(G \circ (F+H))|_U = \{ \ldots, g_{i, j} \circ f_i, \ldots, g_{\ell, j}' \circ h_{\ell}, \ldots \},\\
i=1, \ldots, k, \;\;\; j=1, \ldots, n, \ldots, \ell = 1, \ldots, m.\\
\end{gathered}
\end{equation*}
Note that
\begin{equation*}
\sum\limits_{\ell, j}(g_{\ell, j}' \circ h_{\ell})^*\omega_X = \sum\limits_{\ell}h_{\ell}^{*}\omega_Y = 0.
\end{equation*}
This means that $G \circ H \in SCor^0_{n+m}(Y, X)$, i.e. is an isotropic correspondence. Therefore,
\begin{equation*}
G \circ (F + H)  = G \circ F +  G \circ H = G \circ F\; + \text{isotropic correspondence}.
\end{equation*}

Consider an isotropic correspondence $H \in SCor^0_m(Y, X)$. Locally, over $V_i$ (recall that $f_i(U) \subset V_i$), we have
\begin{equation*}
(G + H)|_{V_i} = \{ \ldots, g_{i, j}, \ldots, h_{i, \ell}, \ldots\}.
\end{equation*}
By definition, we get
\begin{equation*}
((G  + H) \circ F)|_U = \{ \ldots, g_{i, j} \circ f_i, \ldots, h_{i, \ell} \circ f_i, \ldots   \}.
\end{equation*}
Arguing as above, we see that $(G + H) \circ F = G \circ F + H \circ F$ and $H \circ F \in SCor^0_{m+k}(Y, X)$. This means that
\begin{equation*}
(G + H) \circ F  = G \circ F + H \circ F = G \circ F \; + \text{isotropic correspondence}.
\end{equation*}

\end{proof}

\begin{lemma}
There is a well-defined operation
\begin{equation*}
SCor(Y, X) \circ SCor(Z, Y) \rightarrow SCor(Z, X).
\end{equation*}
\end{lemma}

\begin{proof} Consider $F \in SCor(Z, Y)$ and $G \in SCor(Y, X)$. Let $U \subset Z$, $V_i \subset Y$ be small open subsets such that $f_i(U) \subset V_i$ and
\begin{equation*}
\begin{gathered}
F|_U = \{f_1, \ldots, f_k\} \;\;\; \text{ modulo isotropic correspondences},\\
G|_{V_i} = \{g_{i, 1}, \ldots, g_{i, n_i}\} \;\;\; \text{ modulo isotropic correspondences},
\end{gathered}
\end{equation*}
where numbers $n_i$ are not supposed to be equal.  From Lemma $\ref{compwelldefined}$ we get that
\begin{equation*}
(G \circ F)|_U = \{ \ldots, g_{i, j} \circ f_i, \ldots \} \;\;\; \text{ modulo isotropic part correspondences}.
\end{equation*}
We define $G \circ F$ as a sheaf on $Z$ locally defined by the maps given above.

\end{proof}

Let $id_X$ be the identity map $X \rightarrow X$ and we denote the corresponding element in $Cor(X, X)$ by the same symbol $id_X$. The definition says that for any $F \in Cor(Y, X)$ we have
\begin{equation*}
id_X \circ F = F, \;\;\;\; F \circ id_Y = F.
\end{equation*}

The defined composition can be extended to $SC$ and $JC$ in the following way
\begin{equation}\label{linearcomp1}
\begin{gathered}
SC(Y, X) \circ SC(Z, Y) \rightarrow SC(Z, X) \;\;\; \text{and} \;\;\; JC(Y, X) \circ JC(Z, Y) \rightarrow JC(Z, X) \\
(\sum\limits_{i=1}^{n}r_i{F}_i) \circ (\sum\limits_{j=1}^mk_jG_j)  = \sum\limits_{i,j}r_ik_j(F_i\circ G_j),
\end{gathered}
\end{equation}
where $r_i \in \mathbb{Z}$.

\vspace{0.07in}

In the same way we can define composition of bounded correspondences $SC_b(Y, X)$. We defined filtration of bounded correspondences by formula $\ref{filtration}$. The following lemma is the reson to introduce $e^r$ to define the filtration.
\begin{lemma}
We have
\begin{equation*}
SC(Y, X; r) \circ SC(Z, Y; s) \subset SC(Z, X; r+s).
\end{equation*}
\end{lemma}

\begin{proof}
Let $F$ and $G$ be correspondences such that in small neighborhoods of $y$ and $z$ we have $n(y) \leqslant e^r$ and $m(z) \leqslant e^s$ for any $y \in Y$ and $z \in Z$. Then, in a small neighborhood of $z$ the correspondence $G \circ F$ is given by $n(y)m(z)$ maps. We have
\begin{equation*}
n(y)m(z) \leqslant e^re^s = e^{r+s} \;\; \Rightarrow \;\; G \circ F \in SC(Z, X; r+s).
\end{equation*}
\end{proof}

\subsection{Wedge sum of symplectic correspondences}\label{wedgesum}

Consider $F \in SCor_k(Y, X)$, $G \in SCor_n(Z, X)$. Let $U \subset Y$, $V \subset Z$ be small open subsets  such that
\begin{equation*}
F|_U = \{ f_1, \ldots, f_k \}, \quad  G|_V = \{ g_1, \ldots, g_n \}.
\end{equation*}
We define $F \vee G \in SCor_{n+k}(Y \times Z, X)$ as a correspondence defined over $U \times V$ as follows
\begin{equation*}
\begin{gathered}
(F \vee G)|_{U \times V} = \{f_1 \circ pr_Y, \ldots, f_k \circ pr_Y, g_1 \circ pr_Z, \ldots, g_n \circ pr_Z \}.
\end{gathered}
\end{equation*}
where $pr_Y$, $pr_Z$ are projections of $Y \times Z$ onto $Y$ and $Z$, respectively. We see that
\begin{equation*}
\begin{gathered}
(f_1 \circ pr_Y)^{*}\omega_X + \ldots + (f_n \circ pr_Y)^{*}\omega_X + (g_1 \circ pr_Z)^{*}\omega_X + \ldots + (g_m \circ pr_Z)^{*}\omega_X = \\
 pr_Y^{*}\omega_Y + pr_Z^{*}\omega_Z = \omega_Y \oplus \omega_Z.
\end{gathered}
\end{equation*}
We obtain
\begin{equation*}
SCor_k(Y, X) \vee SCor_n(Z, X) \rightarrow SCor_{n+k}(Y \times Z, X).
\end{equation*}

Since elements of $SCor(Y, X)$ are equivalence classes, we need to prove the following lemma:

\begin{lemma}\label{compwelldefined2}
Consider $F \in SCor_k(Y, Y)$, $G \in SCor_n(Z, X)$. The following holds
\begin{equation*}
\begin{gathered}
(F + \; \text{isotropic correspondence}) \vee G = F \vee G + \; \text{isotropic correspondence} \\
F  \vee (G + \; \text{isotropic correspondence}) = F \vee G + \; \text{isotropic correspondence},
\end{gathered}
\end{equation*}
\end{lemma}

\begin{proof}
Consider an isotropic correspondence $H \in SCor^0_m(Y, X)$. The correspondence  $F + H \in SCor_{k+m}(Y, X)$ over some small $U \subset Y$ is given by
\begin{equation*}
(F + H)|_U = \{f_1, \ldots, f_k, h_1, \ldots, h_m\}.
\end{equation*}
By definition we have
\begin{equation*}
\begin{gathered}
((F+H) \vee G)|_{U \times V} = \\
\{f_1 \circ pr_Y, \ldots, f_k \circ pr_Y, h_1 \circ pr_Y, \ldots, h_m \circ pr_Y,  g_1 \circ pr_Z, \ldots, g_n \circ pr_Z \} =\\
(F \vee G)|_{U \times V} + \{ h_1 \circ pr_Y, \ldots, h_m \circ pr_Y \} = \\
(F \vee G)|_{U \times V} \;\; + \;\; \text{isotropic correspondence} \Rightarrow \\
(F+H) \vee G  = F \vee G  \; + \; \text{isotropic correspondence $H \circ pr_Y$}.
\end{gathered}
\end{equation*}

Consider an isotropic correspondence $H \in SCor^0_m(Z, X)$. Over small $U \subset Z$ we have
\begin{equation*}
(G + H)|_{U} = \{ g_1, \ldots, g_n, h_1, \ldots, h_m\}.
\end{equation*}
By definition, we get
\begin{equation*}
\begin{gathered}
(F \vee (G + H))|_U = \\
\{f_1 \circ pr_Y, \ldots, f_k \circ pr_Y, g_1 \circ pr_Z, \ldots, g_n \circ pr_Z, h_1 \circ pr_Z, \ldots, h_m \circ pr_Z, \} =\\
(F \vee G)|_{U \times V} + \{ h_1 \circ pr_Z, \ldots, h_m \circ pr_Z \} = \\
(F \vee G)|_{U \times V} \;\; + \;\; \text{isotropic correspondence} \Rightarrow \\
F \vee (G + H) = F \vee G \; + \; \text{isotropic correspondence $H \circ pr_Z$}.
\end{gathered}
\end{equation*}

\end{proof}

\begin{lemma}
The $\vee$ operation induces a well-defined operation
\begin{equation*}
SCor(Y, X) \vee SCor(Z, Y) \rightarrow SCor(Y \times Z, X),
\end{equation*}
\end{lemma}

\begin{proof} Consider $F \in SCor(Y, X)$ and $G \in SCor(Z, X)$. Let $U \subset Y$, $V \subset Z$ be small open subsets such that
\begin{equation*}
\begin{gathered}
F|_U = \{f_1, \ldots, f_k\} \;\;\; \text{ modulo isotropic correspondences},\\
G|_{V} = \{g_{1}, \ldots, g_{n}\} \;\;\; \text{ modulo isotropic correspondences},
\end{gathered}
\end{equation*}
From Lemma $\ref{compwelldefined2}$ we see that (note that numbers $n, k$ are not constant and depend on points of $Y$ and $Z$)
\begin{equation*}
\begin{gathered}
(F \vee G)|_{U \times V} = \{ f_1 \circ pr_Y, \ldots, f_k \circ pr_Y, g_1 \circ pr_Z, \ldots, g_n \circ pr_Z \} \\
 \text{ modulo isotropic part correspondences}.
\end{gathered}
\end{equation*}
We define $F \vee G$ as a sheaf on $Y \times Z$ locally defined by the maps given above.

\end{proof}

Let us prove a simple lemma which will be used later.

\begin{lemma}\label{simplelemma}
Let $i: Y \times p \hookrightarrow Y \times Z$ be  embedding, where $p \in Z$ is a point. Consider $G \in SCor(Z, X)$. Then, for any $F \in SCor(Y, X)$,  we have
\begin{equation*}
i^{*}(F \vee G) = (F \vee G) \circ i  = F \in SCor(Y, X),
\end{equation*}
where equality above is equality in $SCor(Y, X)$
\end{lemma}

\begin{proof}
By definition,
\begin{equation*}
i^{*}: \;\; SCor(Y \times Z, \; X) \rightarrow SCor(Y \times p, \; X), \quad i^{*}(F \vee G) = (F \vee G) \circ i.
\end{equation*}
For a small $U \subset Y$ we have
\begin{equation*}
\begin{gathered}
\big((F \vee G) \circ i \big)|_U =  \{  f_1 \circ pr_Y, \ldots, f_k \circ pr_Y, g_1(p), \ldots, g_n(p)  \}.
\end{gathered}
\end{equation*}
Since $\{ g_1(p), \ldots, g_n(p)\}$ is isotropic correspondence (set of constant maps) and $(f_i \circ pr_Y)|_{Y \times p} = f_i$, we get that
\begin{equation*}
(F \vee G) \circ i = F + \text{isotropic correspondence} = F \;\; \text{as elements of $SCor(Y, X)$}.
\end{equation*}
\end{proof}

The wedge sum can be extended to $SC$ in the following way
\begin{equation}\label{linearcomp2}
\begin{gathered}
SC(Y, X) \vee SC(Z, X) \rightarrow SC(Y \times Z, X) \\
(\sum\limits_{i=1}^{n}r_i{F}_i) \vee (\sum\limits_{j=1}^mk_jG_j)  = \sum\limits_{i,j}r_ik_jF_i \vee G_j.
\end{gathered}
\end{equation}

\section{Homotopy equivalence in $\mathcal{S}\mathcal{S}ymp$ and in $\mathcal{J}\mathcal{S}ymp$}\label{homequivalence}

\subsection{Homotopy equivalence in $\mathcal{S}\mathcal{S}ymp$}\label{homequivsymp}

Let $M = (M, \omega_M, p_0, p_1)$ be a connected symplectic manifold with two marked points.

\vspace{0.05in}
\noindent
Consider  embeddings
\begin{equation*}
\begin{gathered}
i_0: Y = Y \times p_0 \hookrightarrow Y \times M, \;\;\; i_1: Y = Y \times p_1 \hookrightarrow Y \times M, \\
\end{gathered}
\end{equation*}
We see that $i_0, i_1 \in \mathcal{S}\mathcal{S}ymp(Y, Y \times M)$ and denote by the same symbols the associated correspondences in $i_0, i_1 \in SC(Y, Y \times M)$ (recall that $SCor(Y, Y\times M) \subset SC(Y, Y \times M)$.

\begin{definition}
We say that $F, G \in C(Y, X)$ are $M$-homotopic if there exists  $H \in SC(Y \times M, X)$ such that
\begin{equation*}
H \circ i_0 = F, \;\;\; H \circ i_1 = G.
\end{equation*}
We denote $M$-homotopy equivalent maps by $F \simeq G$.
\end{definition}

\begin{lemma}
Assume that there  exists $I \in SCor(M, X)$. Then, $M$-homotopy relation defines equivalence relation on $SC(Y, X)$.
\end{lemma}

\begin{remark}
If  we replace $SC(Y, X)$ by $\mathcal{S}ymp(Y, X)$, then the homotopy relation does not define equivalence relation. For example, if $dim(X) = dim(Y)$, then any map $f \in \mathcal{S}\mathcal{S}ymp(Y, X)$ is not homotopy equivalent to itself. We even can not replace $SC$ by $SCor(Y, X)$ (transitivity is not satisfied).
\end{remark}

\begin{proof} Let $\gamma \in \mathcal{S}ymp(M, M)$ be a map such that
\begin{equation*}
\gamma(p_0) = p_1, \quad \gamma(p_1) = p_0.
\end{equation*}
Note that $\gamma$ exists for any connected $M$ and we assumed that $M$ is connected. Denote by $id_Y \times \gamma$ the obvious map from $Y \times M$ to $Y \times M$.

\vspace{0.09in}
\noindent
\emph{Reflexivity}. We need to show that $F \simeq F$, where $F \in SC(Y, X)$. Lemma $\ref{simplelemma}$ says that $(F \vee I) \circ i_0 = F$ as elements of $SCor(Y, X)$. In the same way, $(F \vee I) \circ i_1 = F$. Since $F \vee I \in SCor(Y \times M, X)$, we get $F \vee I$ is the required homotopy.

\vspace{0.06in}
\noindent
\emph{Symmetry.} We need to show that $F \simeq G \Rightarrow G \simeq F$. Assume that a homotopy between $F$, $G$ is provided by $H$ and $H \circ i_0 = F$, $H \circ i_1 = G$. Then, $H \circ (id_Y \times \gamma)$ provides a homotopy between $G$ and $F$. Indeed,

\begin{equation*}
\begin{gathered}
(H \circ (id_Y \times \gamma)) \circ i_0(Y) = (H \circ (id_Y \times \gamma))(Y \times p_0) = H|_{Y \times p_1}  = G, \\
(H \circ (id_Y \times \gamma)) \circ i_1 (Y) = (H \circ (id_Y \times \gamma))(Y \times p_1) = H|_{Y \times p_0} = F.
\end{gathered}
\end{equation*}

\vspace{0.06in}
\noindent
\emph{Transitivity.} Assume that $F_1 \simeq F_2$ and $F_2 \simeq F_3$, where $F_1, F_2, F_3 \in C(Y, X)$. We need to show that $F_1 \simeq F_3$. We have

\begin{equation*}
\begin{gathered}
H_1 \in SC(Y \times M, X), \;\;\; H_2 \in SC(Y \times M, X), \\
H_1 \circ i_0 = F_1, \;\; H_1 \circ i_1 = F_2, \;\; H_2 \circ i_0 = F_2, \;\; H_2 \circ i_1 = F_3.
\end{gathered}
\end{equation*}
Consider
\begin{equation*}
H = H_1 - H_2 \circ (id_Y \times \gamma) + F_3 \vee I \in SC(Y \times M, X).
\end{equation*}
We obtain  (see Lemma $\ref{simplelemma}$)
\begin{equation*}
\begin{gathered}
H \circ i_0 = F_1 - F_3 + F_3 = F_1, \\
H \circ i_1 = F_2 - F_2 + F_3 = F_3.
\end{gathered}
\end{equation*}
So, transitivity is proved.
\end{proof}

Let us prove the following theorem

\begin{theorem}\label{deformtheorem}
Let $\varphi_t \in \mathcal{S}\mathcal{S}ymp(Y, X)$ be a symplectic isotopy. If there exists  $I \in SCor(M, X)$, then $\varphi_0$ is $M$-homotopic to $\varphi_1$.
\end{theorem}

\begin{proof}
Let $\ell: M \rightarrow [0,1]$ be a function that is equal to $0$ in a neighborhood of $p_0$ and is equal to $1$ in a neighborhood of $p_1$. Let $U \subset M$ be a small open subsets such that $I|_U$ defined by an unordered set of maps from $M$ to $X$. Let us define the following unordered set of maps on $Y \times U$
\begin{equation*}
H|_{Y \times U} = \{ \varphi_{\ell(q)}(y), \; I|_U \}, \quad \text{where $y \in Y$ and $q \in U$}.
\end{equation*}
Since $\varphi_t$ is symplectomorphism, we have $\varphi_t^{*}\omega_X = \omega_Y$ for any $t$. By definition, maps $I|_U$ pull back $\omega_X$ to $\omega_M$. This shows that the maps defining $H|_{Y \times U}$ pull back $\omega_X$ to $\omega_Y \oplus \omega_M$ and
\begin{equation*}
H|_{X \times U} \in SCor(Y \times U, X).
\end{equation*}
Let $U_{\alpha}$ be a covering of $M$ by small open subsets. Since $I \in SCor(M, X)$, we can glue correspondences $H|_{Y \times U_{\alpha}}$ and get a correspondence
\begin{equation*}
H \in SCor(Y \times M, X).
\end{equation*}
We see that $H \circ i_0 \in SCor(Y, X)$ and
\begin{equation*}
H|_{Y \times p_0} = \{ \varphi_{\ell(p_0)}(x), \; I|_{p_0} \}, \quad \text{where $I|_{p_0}$ is a set of maps $p_0 \rightarrow X$}.
\end{equation*}
This means that $I|_{p_0}$ is an isotropic correspondence. Since $\ell(p_0) = 0$, we get
\begin{equation*}
\begin{gathered}
H|_{Y \times p_0} = \varphi_0 + I|_{p_0} = \varphi_0 \;\; + \;\; \text{isotropic correspondence} = \\
\varphi_0 \;\; \text{as elements of $SCor(Y, X)$}.
\end{gathered}
\end{equation*}
In the same way
\begin{equation*}
\begin{gathered}
H|_{Y \times p_1} = \varphi_1 + I|_{p_1} = \varphi_1 \;\; + \;\; \text{isotropic correspondence} = \\
\varphi_1 \;\; \text{as elements of $SCor(Y, X)$}.
\end{gathered}
\end{equation*}
Therefore, $H$ provides an $M$-homotopy between $\varphi_0$ and $\varphi_1$.

\end{proof}

\subsection{Homotopy equivalence in $\mathcal{J}\mathcal{S}ymp$}\label{homequivpseud}

Let $M = (M, \omega_M, J_M, p_0, p_1)$ is a connected symplectic manifold  endowed with $\omega_M-$compatible almost complex structure $J_M$ and two marked points..

\vspace{0.05in}
\noindent
Consider  pseudoholomorphic embeddings
\begin{equation*}
\begin{gathered}
i_0: Y = Y \times p_0 \hookrightarrow Y \times M, \;\; i_1: Y = Y \times p_1 \hookrightarrow Y \times M, \\
\end{gathered}
\end{equation*}
We see that $i_0, i_1 \in \mathcal{J}\mathcal{S}ymp(Y, Y \times M)$ and denote by the same symbols the associated correspondences in $i_0, i_1 \in JC(Y, Y \times M)$ (recall that $JCor(Y, Y\times M) \subset JC(Y, Y \times M)$.

\begin{definition}
We say that $F, G \in JC(Y, X)$ are $M$-homotopic if there exists  $H \in JC(Y \times M, X)$ such that
\begin{equation*}
H \circ i_0 = F, \;\;\; H \circ i_1 = G.
\end{equation*}
We denote $M$-homotopy equivalent maps by $F \simeq G$.
\end{definition}

\vspace*{0.04in}

\begin{lemma}
Assume that there  exists $\gamma \in \mathcal{J}\mathcal{S}ymp(M, M)$ such that
\begin{equation*}
\gamma(p_0) = p_1, \quad \gamma(p_1) = p_0.
\end{equation*}
Then, $M$-homotopy relation defines equivalence relation on $JC(Y, X)$.
\end{lemma}

\vspace*{0.05in}

\begin{remark}
As in the previous section, if  we replace $JC(Y, X)$ by $\mathcal{J}\mathcal{S}ymp(Y, X)$, then the homotopy relation does not define equivalence relation.
\end{remark}

\begin{proof} Denote by $id_Y \times \gamma$ the obvious map from $Y \times M$ to $Y \times M$. Let $\pi: Y \times M \rightarrow Y$ be the projection and note that $\pi$ is pseudoholomorphic.

\vspace{0.08in}
\noindent
\emph{Reflexivity}. We need to show that $F \simeq F$, where $F \in JC(Y, X)$. Consider $F \circ \pi \in JC(Y \times M, X)$. We see that $F \circ \pi \circ i_0 = F \circ \pi \circ i_1 = F$. So, $F \circ \pi$ is the required homotopy.

\vspace{0.06in}
\noindent
\emph{Symmetry.} We need to show that $F \simeq G \Rightarrow G \simeq F$. Assume that a homotopy between $F$, $G$ is provided by $H$ and $H \circ i_0 = F$, $H \circ i_1 = G$. Then
$H \circ (id_Y \times \gamma)$ provides a homotopy between $G$ and $F$. Indeed,
\begin{equation*}
\begin{gathered}
(H \circ (id_Y \times \gamma)) \circ i_0(Y) = (H \circ (id_Y \times \gamma))(Y \times p_0) = H|_{Y \times p_1}  = G, \\
(H \circ (id_Y \times \gamma)) \circ i_1 (Y) = (H \circ (id_Y \times \gamma))(Y \times p_1) = H|_{Y \times p_0} = F.
\end{gathered}
\end{equation*}

\vspace{0.06in}
\noindent
\emph{Transitivity.} Assume that $F_1 \simeq F_2$ and $F_2 \simeq F_3$, where $F_1, F_2, F_3 \in JC(Y, X)$. We need to show that $F_1 \simeq F_3$. We have
\begin{equation*}
\begin{gathered}
H_1 \in JC(Y \times M, X), \;\;\; H_2 \in JC(Y \times M, X), \\
H_1 \circ i_0 = F_1, \;\; H_1 \circ i_1 = F_2, \;\; H_2 \circ i_0 = F_2, \;\; H_2 \circ i_1 = F_3.
\end{gathered}
\end{equation*}
Consider
\begin{equation*}
H = H_1 - H_2 \circ (id_Y \times \gamma) + F_3 \circ \pi \in JC(Y \times M, X).
\end{equation*}
We have
\begin{equation*}
\begin{gathered}
H \circ i_0 = F_1 - F_3 + F_3 = F_1, \\
H \circ i_1 = F_2 - F_2 + F_3 = F_3.
\end{gathered}
\end{equation*}
Transitivity is proved.
\end{proof}

\section{Cohomology groups}\label{symphomgroup}

In this section we define cohomology groups motivated by Section $\ref{motivationdold}$.

As before, we give some definitions simultaneously for $\mathcal{S}\mathcal{S}ymp$ and $\mathcal{J}\mathcal{S}ymp$. To give all definitions simultaneously we denote
\begin{equation*}
\begin{gathered}
\text{$\mathcal{S}\mathcal{S}ymp$ and $\mathcal{J}\mathcal{S}ymp$ by $\mathcal{S}ymp$}, \\
\text{$SCor$ and $JCor$ by $Cor$}, \\
\text{$SC$ and $JC$ by $C$}.
\end{gathered}
\end{equation*}
If we work with $SCor$, then $M$ is a connected symplectic manifold with two marked points $(M, \omega_M, p_0, p_1)$.

If we work with $JCor$, then $M$ is a connected symplectic manifold $(M, \omega_M, J_M, p_0,p_1)$ with fixed $\omega_M$-compatible almost complex structure $J_M$ and two marked points.

\vspace{0.05in}
\noindent
We also consider
\begin{equation*}
\begin{gathered}
M^n = (\underbrace{M \times \ldots M}_{n}, \; n\oplus\omega_M), \;\; M^0 = pt \;\;\; \text{in case of $\mathcal{S}\mathcal{S}ymp$ and $SCor$}, \\
M^n = (\underbrace{M \times \ldots M}_{n}, \; n\oplus\omega_M, \times^nJ_M), \;\; M^0 = pt \;\;\;  \text{in case of $\mathcal{J}\mathcal{S}ymp$ and $JCor$}
\end{gathered}
\end{equation*}

\subsection{Cohomology groups}

We define
\begin{equation*}
JC_n^{M}(Y, X) = JC(Y \times M^n, X), \quad SC_n^{M}(Y, X) = SC(Y \times M^n, X).
\end{equation*}
Since $JC_n^M(Y, X)$ always contains constant correspondences, we get that it is never empty. On the other hand, $SCor(Y \times M, X)$ may be empty (see Lemma $\ref{exactcorr}$), but we assumed  that $\mathbb{Z}[\emptyset] = 0$ (see definition $\ref{grousympcorr}$). This shows that $JC_n^M(Y, X)$ and $SC_n^M(Y, X)$ are never empty.

\vspace{0.08in}

To give definitions simultaneously we denote
\begin{equation*}
\text{$JC_n^{M} \;$ and $\; SC_n^{M} \;$ by $ \;C_n^{M}$}
\end{equation*}
We denote elements of $C_n^M(Y, X)$ by Greek letters $\sigma$ or $\tau$.

\vspace{0.1in}
\noindent
Let $d_{i, \varepsilon} \in \mathcal{S}ymp(M^n, \; M^{n+1})$ be embeddings defined as follows
\begin{equation*}
d_{i, \varepsilon}(M^{n}) = (\underbrace{M \times \ldots \times M}_{i} \times p_{\varepsilon} \times M \ldots \times M), \;\;\; \varepsilon = \{0,1\}.
\end{equation*}
Let $id_Y$ be the identity map from $Y$ to $Y$. We see that
\begin{equation*}
id_Y \times d_{i, \varepsilon} \in \mathcal{S}ymp(Y \times M^n \rightarrow Y \times M^{n+1}).
\end{equation*}
and we denote by the same symbol the associated correspondence in $Cor(Y \times M^n, \; Y \times M^{n+1})$. We have the induced restriction maps (see Sections $\ref{sympcorrsheaves}$, $\ref{pseudoholcorr}$)
\begin{equation*}
\begin{gathered}
d^{i, \varepsilon} = (id_Y \times d_{i, \varepsilon})^{*} : Cor(Y \times M^{n+1}, \; X) \rightarrow Cor(Y \times M^n, \; X), \\
\sigma \rightarrow \sigma \circ (id_Y \times d_{i, \varepsilon}).
\end{gathered}
\end{equation*}
We extend $d^{i, \varepsilon}$ linearly to $C_{n+1}^M(Y, X)$ and define a differential
\begin{equation*}
\begin{gathered}
d^{i, \varepsilon}\sum_{j=1}^{n}r_j\sigma_j = \sum\limits_{j=1}^nr_jd^{i, \varepsilon}\sigma_j, \;\;\;\; \sigma_j \in C_{n+1}^M(Y, X), \;\; r_j \in \mathbb{Z} \\
d = \sum\limits_{\varepsilon =0}^1\sum\limits_{i=0}^{n-1}(-1)^{\varepsilon + i }d^{i, \varepsilon}.
\end{gathered}
\end{equation*}
For simplicity we denote
\begin{equation*}
d^{i, \varepsilon}\sigma = \sigma|_{Y \times M^{i} \times p_{\varepsilon} \times M^{n-i-1}}.
\end{equation*}

\begin{example}
Let us show that $d^2(C_2^M(Y, X)) = 0$. We have
\begin{equation*}
\begin{gathered}
d\sigma|_{Y \times M \times M} = \sigma|_{Y \times p_0 \times M} - \sigma|_{Y \times p_1 \times M} - \sigma|_{Y \times M \times p_0} + \sigma|_{Y, M \times p_1}, \\
d^2\sigma|_{Y, M \times M} = \sigma|_{Y, p_0 \times p_0} - \sigma|_{Y, p_0 \times p_1}- \sigma|_{Y, p_1 \times p_0} + \sigma|_{Y, p_1 \times p_1} - \\ \sigma|_{Y, p_0 \times p_0} + \sigma|_{Y, p_1 \times p_0} + \sigma|_{Y, p_0 \times p_1} - \sigma|_{Y, p_1 \times p_1} = 0.
\end{gathered}
\end{equation*}
\end{example}

\vspace{0.11in}

In general, for each $i < j$ we see that $d^2$ contains two equal terms with opposite signs
\begin{equation*}
d\sigma|_{Y \times M^n} = (-1)^{i + \varepsilon}\sigma|_{Y \times M^{i} \times p_{\varepsilon} \times M^{n-i-1}} \;\; + \;\; (-1)^{j + \varepsilon}\sigma|_{Y \times M^{j} \times p_{\varepsilon} \times M^{n-j-1}} \; + \ldots \\
\end{equation*}
\begin{equation*}
\begin{gathered}
d^{2}\sigma(Y \times M^n) = (-1)^{i+j + \varepsilon -1}\sigma|_{Y \times M^{i} \times p_{\varepsilon} \times M^{j-i-1} \times p_0 \times M^{n-j-1}} \;\; + \\
(-1)^{i+j + \varepsilon}\sigma|_{Y \times M^{i} \times p_{\varepsilon} \times M^{j-i - 1} \times p_{\varepsilon} \times M^{n-j-1}} \; + \ldots
\end{gathered}
\end{equation*}
We assume that $M^{j-i-1}$ is omitted in the formula above if $i-j-1 = 0$. This shows that
\begin{equation*}
d^2 = 0.
\end{equation*}

So, we get a chain complex $(C_{\bullet}^{M}(Y, X), d)$ and denote its homology groups by
\begin{equation*}
H(C_{\bullet}^{M}(Y, X), d) = H_{\bullet}^{M}(Y, X).
\end{equation*}

\vspace{0.1in}

\begin{definition}
We say that groups $H_{\bullet}^{M}(Y, X)$ are $(X, M)-$motivic cohomology of $Y$.

In case of $\mathcal{S}\mathcal{S}ymp$ and $SCor$ we denote the cohomology groups by  $SH_{\bullet}^{M}(Y, X)$.

In case of $\mathcal{J}\mathcal{S}ymp$ and $JCor$ we denote the cohomology groups by  $JH_{\bullet}^{M}(Y, X)$.
\end{definition}

\vspace{0.1in}

Lemma $\ref{chainmap}$ below shows that these cohomology groups are contrvariant with respect to $Y$. That is the reason to call cohomology, not homology.

\subsection{Functoriality}\label{functoriality}

We  defined the composition (see Section $\ref{operationcorr}$)
\begin{equation*}
C(Y, X) \circ C(Z, Y) \rightarrow C(Z, X)
\end{equation*}
Since $C_n^M(Z, Y) = C( Z \times M^n, \; Y)$, we have
\begin{equation*}
C(Y, X) \circ C_n^M(Z, Y) \rightarrow C_n^M(Z, X)\\
\end{equation*}
In other words, every $G \in C(Y, X)$ defines
\begin{equation*}
G_{*}: C_n^M(Z, Y) \rightarrow C_n^M(Z, X).
\end{equation*}

Consider the identity map $id_{M^n}: M^n \rightarrow M^n$ and the associated correspondence (diagonal of $M^n \times M^n$) by the same symbol. If $F \in Cor(Z, Y)$, then
\begin{equation*}
F \times id_{M^n} \in Cor(Z \times M^n, \; Y \times M^n).
\end{equation*}
This shows that multiplying by $id_{M^n}$ we can transfer any element of $Cor(Z, Y)$ into element of $Cor(Z \times M^n, Y \times M^n)$. So, we have
\begin{equation*}
\begin{gathered}
C_n^M(Y, X) \circ C(Z, Y) \rightarrow C_n^M(Z, X), \\
\sigma \rightarrow \sigma \circ (F \times id_{M^n}) \; \;\;\; F \in C(Z, Y), \;\; \sigma \in C_n^M(Y, X).
\end{gathered}
\end{equation*}
This means that any element $F \in C(Z, Y)$ defines (we abuse notation and denote it by $F^{*}$)
\begin{equation*}
F^{*}: C_n^M(Y, X) \rightarrow C_n^M(Z, X).
\end{equation*}

\begin{lemma}\label{chainmap} Consider $F \in C(Z, Y)$ and $G \in C(Y, X)$. We have the induced maps
\begin{equation*}
F^{*}: H_{\bullet}^M(Y, X) \rightarrow H_{\bullet}^M(Z, X), \;\;\; G_{*}: H_{\bullet}^M(Z, Y) \rightarrow H_{\bullet}(Z, X).
\end{equation*}
In particular, $F$ and $G$ may be the standard symplectic maps from $\mathcal{S}ymp(Z, Y)$ and $\mathcal{S}ymp(Y, X)$, respectively.
\end{lemma}

\begin{proof} Let us prove the lemma for $F \in Cor(Z, Y)$ and prove for linear combinations after. We need to show that $F^{*}$ commutes with the differential. Consider  $\sigma \in C_n^M(Y, X) = Cor(Y \times M^n, X)$. By definition, we have
\begin{equation*}
\begin{gathered}
d = \sum\limits_{\epsilon=1}^1 \sum\limits_{i=0}^{n-1}(-1)^{\varepsilon + i }d^{i, \varepsilon}, \;\;\;\; d^{i, \varepsilon}\sigma = \sigma \circ (id_Y \times d_{i, \varepsilon}), \\
F^{*}(d^{i, \varepsilon}\sigma) = \sigma \circ (id_Y \times d_{i, \varepsilon}) \circ (F \times id_{M^n}) = \sigma \circ (F \times d_{i, \varepsilon}) = \\
\sigma \circ (F \times id_{M^{n-1}}) \circ (id_Z \times d_{i, \varepsilon}) = d^{i, \varepsilon}F^{*}\sigma.
\end{gathered}
\end{equation*}
So, we obtain that $d$ commutes with $F^{*}$. Hence, there is the induced map
\begin{equation*}
F^{*}: H_{\bullet}^M(Y, X) \rightarrow H_{\bullet}^M(Z, X).
\end{equation*}
Suppose that  $F = r_1F_1 + r_2F_2$. Since each term is a chain map, we get that the sum is also a chain map.

\vspace{0.08in}

We need to prove that $G_{*}$ commutes with $d$. We have
\begin{equation*}
G_{*}(d^{i, \varepsilon}\sigma) = G \circ \big(\sigma \circ (id_Z \times d_{i, \varepsilon})\big) = \big( G \circ \sigma \big) \circ (id_Z \times d_{i, \varepsilon}) = d^{i, \varepsilon}(G_{*}\sigma).
\end{equation*}
This means that we have the induced map
\begin{equation*}
G_{*}: H_{\bullet}^M(Z, Y) \rightarrow H_{\bullet}^M(Z, X).
\end{equation*}
\end{proof}

\section{Homotopy invariance of cohomology groups}\label{homequivalence2}

\subsection{Homotopy invariance of $H_{\bullet}^M(Y, X)$}

Again, in this section, we prove some theorems simultaneously for $\mathcal{S}\mathcal{S}ymp$ and $\mathcal{J}\mathcal{S}ymp$. To prove all theorems simultaneously we denote
\begin{equation*}
\begin{gathered}
\text{$\mathcal{S}\mathcal{S}ymp$ and $\mathcal{J}\mathcal{S}ymp$ by $\mathcal{S}ymp$}, \\
\text{$SCor$ and $JCor$ by $Cor$}, \\
\text{$SC$ and $JC$ by $C$}, \\
\text{$SH_{\bullet}^M(Y, X)$ and $JH_{\bullet}^M(Y, X)$ by $H_{\bullet}^M(Y, X)$}
\end{gathered}
\end{equation*}

Consider embeddings (they are symplectic and pseudoholomorphic)
\begin{equation}\label{embmaps}
\begin{gathered}
i_0: Y = Y \times p_0 \hookrightarrow Y \times M, \;\; i_1: Y = Y \times p_1 \hookrightarrow Y \times M, \\
j_0: X = X \times p_0 \hookrightarrow X \times M, \;\; j_1: X = X \times p_1 \hookrightarrow X \times M, \\
\end{gathered}
\end{equation}
We denote by the same symbols the associated correspondences.

\vspace{0.07in}

\begin{theorem}\label{homotopyinv1}
Assume that maps $F_0, F_1 \in C(Y, Z)$ are $M$-homotopic. Then $F^{*}_0, F^{*}_1$ induce the same map
\begin{equation*}
F_0^{*} = F_1^{*}: H_{\bullet}^M(Z, X) \rightarrow H_{\bullet}^M(Y, Z).
\end{equation*}
In particular, $i_0^{*} = i_1^{*}: H_{\bullet}^M(Y \times M, X) \rightarrow H_{\bullet}^M(Y, X)$.

Similarly, assume that  maps $G_0, G_1 \in Cor(X, Z)$ are $M$-homotopic. Then $(G_0)_{*}$ and $(G_1)_{*}$ induce the same map
\begin{equation*}
\begin{gathered}
(G_0)_{*} = (G_1)_{*}: H_{\bullet}^M(Y, X) \rightarrow H_{\bullet}^M(Y, Z) \\
\end{gathered}
\end{equation*}
In particular,  $(j_0)_{*} = (j_1)_{*}: H_{\bullet}^M(Y, X) \rightarrow H_{\bullet}^M(Y, X \times M)$.
\end{theorem}

\begin{proof}
Denote by $H \in C(Y \times M, Z)$ the homotopy between $F_0, F_1$. We have our chain maps
\begin{equation*}
F_0^{*}, F_1^{*}: C_n^M(Z, X) \rightarrow C_n^M(Y, X), \;\;\; H^{*}: C_n^M(Z, X) \rightarrow C_n^M(Y \times M, X).
\end{equation*}
Any element $\sigma  \in C_n^M(Y \times M, X)$ can be considered as an element of $C_{n+1}^M(Y, X)$. So, composing this with $H^{*}$ we get
\begin{equation*}
K: C_n^M(Z, X) \xrightarrow{H^{*}} C_n^M(Y \times M, \; X) \rightarrow C_{n+1}^M(Y, X).
\end{equation*}
Direct computation shows that
\begin{equation*}
dK = F_0^{*} - F_1^{*} \pm Kd.
\end{equation*}
This means that $F_0^{*}, F_1^{*}$ are chain homotopic. If $Z = Y \times M$, $F_0 = i_0$, $F_1 = i_1$, $H = id_Y \times id_M$, then we get that $i_0^{*}=i_1^{*}$.
\\

Let us prove the second part. Denote by $H \in SC(X \times M, Z)$ the homotopy between $G_1, G_2$. We have the chain maps
\begin{equation*}
(G_0)_{*}, (G_1)_{*}: C_n^M(Y, X) \rightarrow C_n^M(Y, Z), \;\;\;  H_{*}: C_n^M(Y, X \times M) \rightarrow C_n^M(Y, Z).
\end{equation*}
Let us recall that multiplying by the diagonal $id_M \subset M \times M$ we can transfer any element of $C(Y, X)$ into $C(Y \times M, X \times M)$. We have
\begin{equation*}
C_n^M(Y, X) = C(Y \times M^n, \; X) \xrightarrow{\times id_{M}} C(Y \times M^n \times M, \; X\times M) = C_{n+1}^M(Y, X \times M).
\end{equation*}
Composing with $H_{*}$ we get
\begin{equation*}
K: C_{n}^M(Y, X) \xrightarrow{\times id_M} C_{n+1}^M(Y, X \times M) \xrightarrow{H_{*}} C_{n+1}(Y, Z).
\end{equation*}
We see that
\begin{equation*}
\begin{gathered}
d^{n+1, 0}(H_{*} \circ (\sigma \times id_M) ) = H_{*} \circ (\sigma \times p_0) = (G_0)_{*} \circ \sigma \\
d^{n+1, 1}(H_{*} \circ (\sigma \times id_M) ) = H_{*} \circ (\sigma \times p_1) = (G_1)_{*} \circ \sigma.
\end{gathered}
\end{equation*}
This implies that
\begin{equation*}
dK = (G_0)_{*} - (G_1)_{*} \pm Kd.
\end{equation*}
So, $(G_0)_{*}, (G_1)_{*}$ are chain homotopic. The last statement can be proved by assuming that $Z = X \times M$, $G_0 = j_0$, $G_1 = j_1$, $H = id_X \times id_M$.

\end{proof}

The following Lemma follows immediately from Theorem $\ref{deformtheorem}$ and Thereom $\ref{homotopyinv1}$.

\begin{lemma}\label{cohominvhom}
Consider a symplectic isotopies $\varphi_t \in \mathcal{S}\mathcal{S}ymp(X, X)$, $\psi_t \in \mathcal{S}\mathcal{S}ymp(Y, Y)$. If $SCor(M, X) \neq \emptyset$, then $(\varphi_0)_{*} = (\varphi_1)_{*}$ and $\psi_0^{*} = \psi_1^{*}$.
\end{lemma}

\subsection{Canceling $M$ in $JH_{\bullet}^M(Y \times M, X)$ and $JH_{\bullet}^M(Y, X \times M)$}\label{cancelpseud}

This section is technically much easier than the next one because the projection $X \times M \rightarrow X$ is pseudoholomorphic, but not symplectic.

\begin{theorem}\label{invarianttheorem}
Assume that there exists a pseudoholomorphic map $\gamma: \;  M \times M \rightarrow M$ such that
\begin{equation*}
\gamma(p_0 \times M) = p_0, \quad  \gamma(p_1 \times M) = id_M.
\end{equation*}
Let $pr_X$, $pr_Y$ be the projections of $X \times M$ onto $X$ and $Y \times M$ onto $Y$, respectively. Then,
\begin{equation*}
(pr_X)_{*}: JH_{\bullet}^M(Y, X \times M) \rightarrow JH_{\bullet}^M(Y, X).
\end{equation*}
is an isomorphism and the inverse is $(j_0)_{*}$.

Similarly,
\begin{equation*}
(pr_Y)^{*}: JH_{\bullet}^M(Y, X)  \rightarrow JH_{\bullet}^M(Y \times M, X)
\end{equation*}
is an isomorphism and the inverse is $(i_0)^{*}$.

\end{theorem}

\begin{remark}
Conditions of this theorem are satisfied when $M = (\mathbb{C}, \; 0, \; 1)$. The map $\gamma$ is defined in the following way
\begin{equation*}
\gamma(z_1, z_2) = z_1z_2.
\end{equation*}
Also, this conditions are satisfied for some polydisks in $\mathbb{C}^n$ and some other open sets of $\mathbb{C}^n$.
\end{remark}

\begin{proof}

Let $i_0$, $i_1, j_0, j_1$ be maps defined by $(\ref{embmaps})$. We obviously have the following
\begin{equation*}
\begin{gathered}
pr_Y \circ i_0 = id_Y  \;\; \Rightarrow  \;\; i_0^{*} \circ pr_Y^{*} = id^{*}_Y. \\
pr_X \circ j_0 =  id_X \;\;  \Rightarrow \;\;  (pr_X)_{*} \circ (j_0)_{*} = (id_X)_{*}.
\end{gathered}
\end{equation*}
We need to prove that
\begin{equation}\label{requiredtoprove2}
\begin{gathered}
(j_0)_{*} \circ (pr_X)_{*}  = (id_{X \times M})_{*},\\
pr_Y^{*} \circ i_0^{*} = id_{Y\times M}^{*}.
\end{gathered}
\end{equation}

First let us prove the lemma below.

\begin{lemma}\label{commutativediag2}
The following diagram is commutative
\begin{equation*}
\begin{gathered}
\begin{tikzcd}
&[2em] X \times M &[2em] \arrow{l}[above]{j_0} X
  \\[2em]
  X \times M \arrow{ru}[left]{id_{X \times M}} \arrow{r}{j_1 \times id_M \;\;}  & X \times M \times M \arrow{u}[right]{id_X \times \gamma } &  \arrow{l}[above]{j_0 \times id_M} X \times M \arrow{u}[right]{pr_X}
\end{tikzcd}
\end{gathered}
\end{equation*}
\end{lemma}

\begin{proof}
Let $(x, q) \in X \times M$ be an arbitrary point. The following computations prove the lemma.
\begin{equation*}
\begin{gathered}
(id_X \times \gamma)\circ (j_0 \times id_M)(x, q) = (x, \gamma(p_0, q)) = (x, p_0), \\
j_0 \circ pr_X(x, q) = j_0(x) = (x, p_0), \\
(id_X \times \gamma) \circ (j_1 \times id_M)(x, q) = (x, \gamma(p_1, q)) = (x, q).
\end{gathered}
\end{equation*}
\end{proof}

\noindent
Since $(j_0)_{*} = (j_1)_{*}$ (see Lemma $\ref{homotopyinv1}$), we obtain that
\begin{equation*}
(j_0)_{*} \circ (pr_X)_{*} = (id_X \times \gamma)_{*} \circ (j_0 \times id_M)_{*} =  (pr_X \times \gamma)_{*} \circ (j_1 \times id_M)_{*} = id_{X \times M}.
\end{equation*}
As a result, we proved the first equality from $(\ref{requiredtoprove2})$. Let us prove the second equality.

\begin{lemma}
The following diagram is commutative
\begin{equation*}
\begin{gathered}
\begin{tikzcd}
&[2em] Y \times M &[2em] \arrow{l}[above]{i_0} Y
  \\[2em]
  Y \times M \arrow{ru}[left]{id_{Y \times M}} \arrow{r}{i_1 \times id_M \;\;}  & Y \times M \times M \arrow{u}[right]{id_Y \times \gamma } &  \arrow{l}[above]{i_0 \times id_M} Y \times M \arrow{u}[right]{pr_Y}
\end{tikzcd}
\end{gathered}
\end{equation*}
\end{lemma}

\begin{proof}
The proof coincides withe the proof of Lemma $\ref{commutativediag2}$. We need to replace $j_0, j_1, X$ by $i_0, i_1, Y$.
\end{proof}

Since $(i_0)^{*} = (i_1)^{*}$ (see Theorem \ref{homotopyinv1}), we have
\begin{equation*}
(pr_Y)^{*} \circ (i_0)^{*} =  (i_0 \times id_M)^{*} \circ (id_Y \times \gamma)^{*}=  (i_1 \times id_M)^{*} \circ (id_Y \times \gamma)^{*} = id_{Y \times M}^{*}.
\end{equation*}
This proves the second equality from $(\ref{requiredtoprove2})$.
\end{proof}

\subsection{Canceling $M$ in $SH_{\bullet}^M(Y \times M, X)$ and $SH_{\bullet}^M(Y, X \times M)$}\label{cancelsymp}

As we mentioned above this section is technically more complicated, but ideas are the same.

\begin{theorem}\label{invarianttheorem}
Assume that there exists a symplectic isotopy $\gamma_t \in \mathcal{S}\mathcal{S}ymp(M \times M, \; M \times M)$ such that $\gamma_0 = id_{M \times M}$ and
\begin{equation*}
\gamma_1: M \times M \rightarrow M \times M, \;\;\;\; \gamma_1(p_1 \times q) = (q \times p_0), \;\; \forall \; q\in M.
\end{equation*}
If there exists $J \in SCor(M, X)$, then the following map is isomorphism
\begin{equation*}
(j_0)_{*}: \; SH_{\bullet}^M(Y, X) \rightarrow SH_{\bullet}^M(Y, X \times M).
\end{equation*}
and the inverse is $(id_X \vee J)_{*}$.

If there exists  $I \in SCor(M, Y)$, then the following map is isomorphism
\begin{equation*}
i_0^{*}: \;   SH_{\bullet}^M(Y \times M, X)  \rightarrow SH_{\bullet}^M(Y, X).
\end{equation*}
and the inverse is $(id_Y \vee I)^{*}$.
\end{theorem}

\begin{remark}
Conditions of this theorem are satisfied when $M = (\mathbb{C}, \; 1, \; -1, \; \frac{i}{2}dz\wedge d\overline{z})$ or $M = (D_r, \; a, \; -a, \; \frac{i}{2}dz\wedge d\overline{z})$, where $D_r \subset \mathbb{C}$ is a disk of radius $r$. The map $\gamma_1$ obviously exists. For example, consider
\begin{equation*}
\begin{gathered}
\gamma_t = \begin{pmatrix}
cos(\frac{\pi t}{2}) & sin(\frac{\pi t}{2}) \\
-sin(\frac{\pi t}{2}) & cos(\frac{\pi t}{2})
\end{pmatrix}.
\end{gathered}
\end{equation*}
Easy to see that $\gamma_t \in \mathcal{S}\mathcal{S}ymp(M \times M, \; M \times M)$ is the required isotopy. Lemma $\ref{contremb}$ says that the correspondences $I$ and $J$ exist.
\end{remark}

\begin{proof}
Denote by $id_Y$, $id_X$ the correspondences associated with the identity maps of $Y$ and $X$, respectively. We have the following symplectic correspondences:
\begin{equation*}
id_Y \vee I \in Cor(Y \times M, Y), \;\;\; id_X \vee J \in Cor(X \times M, X)
\end{equation*}
Let $i_0$, $i_1, j_0, j_1$ be maps defined by $(\ref{embmaps})$. Lemma $\ref{simplelemma}$ says that (equality below is considered in $SCor$, i.e. up to isotropic correspondences)
\begin{equation*}
\begin{gathered}
(id_Y \vee I) \circ i_0  = id_Y  \;\; \Rightarrow  \;\; i_0^{*} \circ (id_Y \vee F)^{*} = id^{*}_Y. \\
(id_X \vee J) \circ j_0 =  id_X \;\;  \Rightarrow \;\;  (id_X \vee J)_{*} \circ (j_0)_{*} = (id_X)_{*}.
\end{gathered}
\end{equation*}
We need to prove that
\begin{equation}\label{requiredtoprove}
\begin{gathered}
(j_0)_{*} \circ (id_X \vee J)_{*}  = (id_{X \times M})_{*},\\
(id_Y \vee I)^{*} \circ i_0^{*} = id_{Y\times M}^{*}.
\end{gathered}
\end{equation}

\begin{lemma}\label{commutativediag}
The following two diagrams are commutative
\begin{equation*}
\begin{gathered}
\begin{tikzcd}
&[2em] X \times M
  \\[2em]
 X \times M  \arrow{ru}[left]{id_{X \times M}} \arrow{r}{j_1 \times id_M \;\;}  & X \times M \times M \arrow{u}[right]{\big(id_{X \times M} \vee (j_0 \circ J)\big) \circ (id_X \times \gamma_1)}
\end{tikzcd}
\hspace{0.1in}
\begin{tikzcd}
X \times M
&[2em] X  \arrow{l}[above]{j_0}  \\[2em]
X \times M \times M  \arrow{u}[right]{id_{X \times M} \vee (j_0 \circ J)}
& X \times M \arrow{u}[right]{id_X \vee J}  \arrow{l}[above]{j_0 \times id_M}
\end{tikzcd}
\end{gathered}
\end{equation*}
\end{lemma}

\begin{proof}
First, note that
\begin{equation*}
\begin{gathered}
j_1 \times id_M: \;\; X \times M \rightarrow X \times p_1 \times M, \\
\big((id_X \times \gamma_1) \circ (j_1 \times id_M) \big): \;\; X \times M  \rightarrow  X \times M \times p_0.
\end{gathered}
\end{equation*}
Let $U \subset M$ be a small open subset such that
\begin{equation*}
J|_U = \{f_1, \ldots, f_n\}, \;\;\; f_i \in C^{\infty}(U, X).
\end{equation*}
Let us prove that the first diagram is commutative. We have
\begin{equation*}
\begin{gathered}
\big(id_{X \times M} \vee (j_0 \circ J)\big)|_{X \times U \times p_0} = \{id_{X \times U}, \; f_1(p_0) \times p_0, \ldots, f_n(p_0) \times p_0   \} = \\
id_{X \times U} + \text{isotropic correspondence} = id_{X \times U} \;\; \text{as elements of $SCor(X \times U, X \times U)$}.
\end{gathered}
\end{equation*}
As a result, we obtain
\begin{equation*}
\begin{gathered}
\big(id_{X \times M} \vee (j_0 \circ J)\big) \circ (id_X \times \gamma_1) \circ (j_1 \times id_M) = id_{X \times M}.
\end{gathered}
\end{equation*}
Let us prove that the second diagram is commutative.
\begin{equation*}
\begin{gathered}
\big(j_0 \circ (id_X \vee J) \big)|_{X \times U} = \{j_0 \circ id_X, \; j_0 \circ f_1, \ldots, j_0 \circ f_n\}, \\
\big( (id_{X \times M} \vee (j_0 \circ J) )\circ (j_0 \times id_M) \big)|_{X \times U} = (id_{X \times M} \vee (j_0 \circ J) )|_{X \times p_0 \times U} = \\
\{j_0 \circ id_{X}, \; j_0 \circ f_1, \ldots, j_0 \circ f_n  \}.
\end{gathered}
\end{equation*}
\end{proof}

We can not glue these two diagrams into one commutative diagram. It turns out that we can glue them on cohomology level. Since $\gamma_1$ is symplectically isotopic to identity, we get from Theorem $\ref{cohominvhom}$ that
\begin{equation*}
\begin{gathered}
(\gamma_1)_{*} = (id_{M \times M})_{*}, \quad (id_X \times \gamma_1)_{*} = (id_{X \times M \times M})_{*}, \\
\bigm(\big(id_{X \times M} \vee (j_0 \circ J)\big) \circ (id_X \times \gamma_1) \bigm)_{*} =  \big(id_{X \times M} \vee (j_0 \circ J)\big)_{*} \circ (id_X \times \gamma_1)_{*} =  \\
\big(id_{X \times M} \vee (j_0 \circ J)\big)_{*}.
\end{gathered}
\end{equation*}
This implies that we have the following commutative diagram
\begin{equation*}
\begin{gathered}
\begin{tikzcd}
&[2em] H_{\bullet}^M(Y,\;  X \times M) &[2em] \arrow{l}[above]{(j_0)_{*}} H_{\bullet}^M(Y, \; X)
  \\[2em]
 H_{\bullet}^M(Y,\; X \times M) \arrow{ru}[left]{(id_{X \times M})_{*}} \arrow{r}{(j_1 \times id_M)_{*} \;\;}  & H_{\bullet}^M(Y, \; X \times M \times M) \arrow{u}[right]{\big(id_{X \times M} \vee (j_0 \circ J)\big)_{*} } &  \arrow{l}[above]{(j_0 \times id_M)_{*}} H_{\bullet}^M(Y, \; X \times M) \arrow{u}[right]{(id_X \vee J)_{*}}
\end{tikzcd}
\end{gathered}
\end{equation*}
Since $(j_0)_{*} = (j_1)_{*}$ (see Lemma $\ref{homotopyinv1}$), we obtain
\begin{equation*}
\begin{gathered}
(j_0)_{*} \circ (id_X \vee J)_{*} = \big(id_{X \times M} \vee (j_0 \circ J)\big)_{*} \circ (j_0 \times id_M)_{*} =  \\
\big(id_{X \times M} \vee (j_0 \circ J)\big)_{*} \circ (j_1 \times id_M)_{*} = id_{X \times M}.
\end{gathered}
\end{equation*}
As a result, we proved the first equality from $(\ref{requiredtoprove})$. Let us prove the second equality.

\begin{lemma}
The following diagrams are commutative
\begin{equation*}
\begin{gathered}
\begin{tikzcd}
&[2em] Y \times M
  \\[2em]
 Y \times M \arrow{ru}[left]{id_{Y \times M}} \arrow{r}{i_1 \times id_M \;\;}  & Y \times M \times M \arrow{u}[right]{\big(id_{Y \times M} \vee (i_0 \circ I)\big) \circ (id_Y \times \gamma_1)}
\end{tikzcd}
\hspace{0.1in}
\begin{tikzcd}
Y \times M
&[2em] Y  \arrow{l}[above]{i_0}  \\[2em]
Y \times M \times M  \arrow{u}[right]{id_{Y \times M} \vee (i_0 \circ I)}
& Y \times M \arrow{u}[right]{id_Y \vee I}  \arrow{l}[above]{i_0 \times id_M}
\end{tikzcd}
\end{gathered}
\end{equation*}
\end{lemma}

\begin{proof}
The proof mimics the proof of Lemma $\ref{commutativediag}$. We just need to replace $j_0, j_1, X, J$ by $i_0, i_1, Y, I$.
\end{proof}

Arguing as before, we see that $(id_{Y } \times \gamma_1)^{*} = id^{*}_{Y \times M \times M}$. We have the following commutative diagram
\begin{equation*}
\begin{gathered}
\begin{tikzcd}
&[2em]   H_{\bullet}^M(Y \times M \;  X) \arrow{ld}[left]{(id_{Y \times M})^{*}} \arrow{r}[above]{(i_0)^{*}} \arrow{d}[right]{\big(id_{Y \times M} \vee (i_0 \circ I)\big)^{*} } &[2em]   H_{\bullet}^M(Y, \; X) \arrow{d}[right]{(id_Y \vee I)^{*}}
\\[2em]
H_{\bullet}^M(Y \times M,\; X)   & H_{\bullet}^M(Y \times M \times M, \; X) \arrow{l}[above]{(i_1 \times id_M)^{*} \;\;} \arrow{r}[above]{(i_0 \times id_M)^{*}} & H_{\bullet}^M(Y \times M, \; X)
\end{tikzcd}
\end{gathered}
\end{equation*}
Since $(i_0)_{*} = (i_1)_{*}$, we have
\begin{equation*}
\begin{gathered}
 (id_Y \vee I)^{*} \circ (i_0)^{*} =  (i_0 \times id_M)^{*} \circ \big(id_{Y \times M} \vee (i_0 \circ I)\big)^{*} =  \\
 (i_1 \times id_M)^{*} \circ \big(id_{Y \times M} \vee (i_0 \circ I)\big)^{*} = id_{Y \times M}^{*}.
\end{gathered}
\end{equation*}
This proves the second equality from $(\ref{requiredtoprove})$.
\end{proof}

\section{Triangulated persistence category of symplectic manifolds}\label{triangperscat}

Triangulated Persistence Categories are defined by Biran, Cornea, and Zhang in \cite{Octav1}, \cite{Octav2}. We do not give definitions here and the reader  can find all the details in the mentioned papers.

\begin{remark}
In a similar way triangulated persistence category of almost complex manifolds can be constructed. We consider only simplectic manifolds to simplify the paper.
\end{remark}

\subsection{Topological motivation}
Let $Y$, $Z$, $X$ be some nice topological spaces. We defined abelian group $K(Y, X)$ by formula $(\ref{requiredfor})$. Recall that (discussed at the beginning of Section $\ref{operationcorr}$) maps
\begin{equation*}
Z \rightarrow SP^k(Y), \;\;\; Y \rightarrow SP^n(X)
\end{equation*}
can be composed and we obtain a map $Z \rightarrow SP^{nk}(X)$. Extending this composition linearly to linear combinations of such maps we get bilinear operation
\begin{equation*}
K(Y, X) \times K(Z, Y) \xrightarrow{\circ} K(Z, X).
\end{equation*}
Let $\mathcal{T}op$ be the category of topological spaces. Consider a category $\mathcal{T}$, where

\begin{equation*}
Ob(\mathcal{T}) = Ob(\mathcal{T}op), \;\;\; Mor_{\mathcal{T}}(Y, X) = K(Y, X).
\end{equation*}

Let $Y \sqcup Z$ be the disjoint union. Then, the standard inclusions $i_Y: Y \hookrightarrow Y \sqcup Z$, $i_Z: Z \hookrightarrow Y \sqcup Z$ provides a map
\begin{equation*}
K(Y \sqcup Z, X) \xrightarrow{(\circ i_Y, \; \circ i_X )} K(Y,X) \oplus K(Z, X).
\end{equation*}
Easy to see that this map is isomorphism of groups. This implies that $Y \sqcup Z$ is the coproduct of $Y$ and $Z$ in the category $\mathcal{T}$. Since the set of morphisms is abelian and compositions are bilinear, we get that $Y \sqcup Z$ is also a product. Note that the empty set is the zero object of $\mathcal{T}$. As a result, we see that all finite products and coproducts exist and get the following lemma.

\begin{lemma}
The category $\mathcal{T}$ is additive.
\end{lemma}

Consider bounded below cochain complexes $Com^{-}(\mathcal{T})$ and its homotopy category $\mathcal{K}(\mathcal{T})$. It is known that homotopy category of chain complexes of any additive category is triangulated (see \cite[Th. 11.2.6]{categories}).

We started with topological spaces and constructed the triangulated category. The goal of this section to develop similar idea for symplectic manifolds.

\subsection{Triamgulated category of symplectic manifolds}

Recall that in Section $\ref{abelsympcorr}$ we defined bounded symplectic correspondences $SC_{b}(Y, X)$. Let $\widetilde{\mathcal{B}}$ be a category, where

\begin{equation*}
\begin{gathered}
\text{$Ob(\widetilde{\mathcal{B}}) = \;$ \emph{connected}  symplectic manifolds},\\
Mor_{\widetilde{\mathcal{B}}}(Y, X) = SC_{b}(Y, X).
\end{gathered}
\end{equation*}

Let $Z_1, \ldots, Z_n, Z_1', \ldots, Z_k', Z_1'', \ldots, Z_m''$ be connected symplectic manifolds. Define  $n \times k$ and $k \times m$ matrices $F$, $G$ with entries  $(F)_{i, j} \in SC_b(Z_i, Z_j')$ and $(G)_{i, j} \in SC_b(Z_i', Z_j'')$, respectively. Then, we define their composition $G \circ F$ as an $n \times m$ matrix with following entries
\begin{equation}\label{prodmat}
(G \circ F)_{i, j} = \sum\limits_{\ell = 1}^k G_{\ell, j} \circ F_{i, \ell} \in SC_{b}(Z_i, Z_j'').
\end{equation}
where $\circ$ is the defined composition of correspondences (see Section $\ref{operationcorr}$). In other words, we multiply these matrices.

Consider matrices  $F_1, F_2$ and $G_1, G_2$ with entries from $SC_b(Z_i, Z_j')$ and $SC_b(Z_i', Z_j'')$, respectively. We define composition between linear combinations of such matrices linearly, i.e.
\begin{equation*}
(\lambda_1G_1 + \lambda_2G_2) \circ (\mu_1F_1 + \mu_2F_2) = \sum\limits_{i,j = 1}^2\lambda_i\mu_j F_i \circ G_j.
\end{equation*}

If $dim(Y) \neq dim(X)$, then the disjoint union $Y \sqcup X$ is not a manifold, but we want to have a category with all disjoint unions. Let us assume that the empty set is a symplectic manifod. Define a new category $\mathcal{A}$ such that

\begin{equation*}
\begin{gathered}
\text{$Ob(\mathcal{B}) = $ \{all disjoint unions of connected symplectic manifolds $\sqcup_{i=1}^nZ_i$, $\; n$ is arbitrary\}  } \\
Mor_{\mathcal{B}}\big(\sqcup_{i=1}^nZ_i, \; \sqcup_{j=1}^k Z_j'\big) = \; \{\text{finite formal finite linear combinations of the defined} \\
\text{ $n \times k$ matrices with entries from $SC_b(Z_i, Z_j')$}\},\\
Mor_{\mathcal{B}}(\sqcup_{i=1}^nZ_i, \; \emptyset) = Mor_{\mathcal{B}}(\emptyset, \; \sqcup_{i=1}^nZ_i) = 0, \\
\text{Composition is defined by formula $\ref{prodmat}$}.
\end{gathered}
\end{equation*}
In particular, for connected symplectic $Y$, $X$ we have
\begin{equation*}
Mor_{\mathcal{B}}(Y, X) = SC_b(Y, X).
\end{equation*}

\begin{lemma}\label{triangproof}
The category $\mathcal{B}$ is additive, where $Y \oplus X$ is the disjoint union $Y \sqcup X$ and the empty set is the zero object. The category $\widetilde{\mathcal{B}}$ is full subcategory of $\mathcal{B}$.
\end{lemma}

\begin{proof}
Let $Z_1, Z_2$ be symplectic manifolds. We have an $1 \times 2$ matrices $\delta_k \in Mor_{\mathcal{A}}(Z_k, Z_1 \sqcup Z_2)$ such that
\begin{equation*}
\delta_1 = (id_{Z_1}, 0), \;\;\; \delta_2 = (0, id_{Z_2}).
\end{equation*}
We get the following map
\begin{equation*}
Mor_{\mathcal{B}}(Z_1 \sqcup Z_2, \sqcup_{j=1}^m Z_j') \xrightarrow{(\circ \delta_1, \circ \delta_2 )} \bigoplus_{i=1}^2 Mor_{\mathcal{B}}(Z_i, \sqcup_{j=1}^m Z_j').
\end{equation*}
Definitions show that the map above is isomorphism (we write matrices as direct sums of their columns). This implies that $Z_1 \sqcup Z_2$ is coproduct of $Z_1$ and $Z_2$. In the same way we can show that coproduct of arbitrary elements exists.

Since the set of morphisms are abelian groups and compositions are bilinear, we obtain that all coproducts are also products. We see that the empty set is the zero object. As a result, all finite limits and colimits exist and $\mathcal{B}$ is additive.

\vspace{0.07in}
The last part of the lemma follows from definitions.

\end{proof}

Consider a category $Com^{-}(\mathcal{B})$ of bounded below cochain complexes. Let us recall that objects of this category are cochain complexes $(C_{\bullet}, d_C)$
\begin{equation*}
0  \rightarrow C_{-k} \xrightarrow{d^{-k}_C} \ldots \rightarrow C_{n-1} \xrightarrow{d^{n-1}_C} C_n \xrightarrow{d^n_C}  \ldots, \;\;\;\;  C_i \in \mathcal{B}.
\end{equation*}
Morphisms between $(C_{\bullet}, d_C)$ and $(K_{\bullet}, d_K)$ are chain maps, i.e collection of morphisms $F_n \in Mor_{\mathcal{B}}(C_n, K_n)$ such that $F_n \circ d_C^{n-1} = d_K^{n-1} \circ F_{n-1}$ for any $n$. A chain map $F$ is called isomorphism if each $F_n$ is isomorphism.  We say that two chain maps $F, G: (C_{\bullet}, d_C) \rightarrow (K_{\bullet}, d_K)$ are homotopic if there exist morphisms   $H_n \in Mor_{\mathcal{B}}(C_n, K_{n-1})$ such that $F_n - G_n = d_K^{n-1} H_n \pm H_{n+1} d^{n}_C$ for any $n$.

\begin{example}\label{examplechain1}
Let $M = (M, p)$ be a symplectic manifold with a marked point and $M^n$ be its $n$th power. We have embeddings $i_k: M^n \rightarrow M^{k} \times p \times M^{n-k} \subset M^{n+1}$ and define $d^n = \sum\limits_{k=0}^n (-1)^k i_k \in SC_b(M^n, M^{n+1})$. We see that $d^{n+1} d^n = 0$ and define a chain complex
\begin{equation*}
0 \rightarrow Y, \xrightarrow{d^0} Y \times M \xrightarrow{id_Y \times d^1} \ldots \xrightarrow{id_Y \times d^{n-1}} Y \times M^{n} \xrightarrow{id_Y \times d^n} \ldots
\end{equation*}
Here $d^0$ maps $Y$ to $Y \times p \subset Y \times M$. We denote this complex by $Y_{\bullet}(M, p)$. Consider a complex $X_{\bullet}(M, p)$ and note that any $F \in SC_{b}(Y, X)$ defines a chain map from $Y_{\bullet}(M, p) \rightarrow X_{\bullet}(M, p)$ with $F_n = F \times id_{M^n}$.
\end{example}

\begin{example}\label{examplechain2}
Let $M = (M, p_0, p_1)$ be a symplectic manifold with two marked points. We have embeddings $d_{i, \varepsilon}: M^{i} \times p_{\varepsilon} \times M^{n-i} \rightarrow M^{n+1}$, where $\varepsilon = 0, 1$. Define $d_{n+1} = \sum\limits_{\varepsilon =0}^1\sum\limits_{i=0}^{n}(-1)^{\varepsilon + i }d_{i, \varepsilon}$ and note that $d_{n+1} \circ d_n = 0$. We get the following chain complex
\begin{equation*}
0 \rightarrow Y  \xrightarrow{d^0} Y \times M \xrightarrow{id_Y \times d^1} \ldots \xrightarrow{id_Y \times d^{n-1}} Y \times M^{n} \xrightarrow{id_Y \times d^n} \ldots
\end{equation*}
and denote this complex by $Y_{\bullet}(M, p_0, p_1)$. Consider another complex $X_{\bullet}(M, p_0, p_1)$. As in the previous example, any  $F \in SC_b(Y, X)$ defines a chain map from $Y_{\bullet}(M, p_0, p_1) \rightarrow X_{\bullet}(M, p_0, p_1)$ with $F_n = F \times id_{M^n}$.
\end{example}

\vspace*{0.08in}

Consider a category $\mathcal{K}(\mathcal{B})$, where

\begin{equation*}
\begin{gathered}
Ob(\mathcal{K}(\mathcal{B})) = Ob(Com^{-}(\mathcal{B})), \\
Mor_{\mathcal{K}(\mathcal{B})}(C_{\bullet}, K_{\bullet}) = \{ \text{chain maps} \}/\simeq,
\end{gathered}
\end{equation*}
where $F \simeq G$ if they are chain homotopic.

\begin{lemma}
The category $\mathcal{K}(\mathcal{B})$ is triangulated.
\end{lemma}

\begin{proof}
It is known that the homotopy category of chain complexes of any additive category is triangulated (see \cite[Th. 11.2.6]{categories}).
\end{proof}

\subsection{Triangulated persistence category of symplectic manifolds}\label{persistencecat}

In the previous section we constructed the category $\mathcal{B}$ and the triangulated category $\mathcal{K}(\mathcal{B})$. The goal of this section is to construct a refinement of these categories and get a triangulated persistene category $\mathcal{C}$.

\vspace{0.1in}

Let us consider pairs $(Z, r)$, where $Z$ is \emph{connected} symplected manifolds and $r \in \mathbb{R}$. Also, consider formal elements
\begin{equation*}
(Z_1, r_1) \oplus \ldots \oplus (Z_n, r_n),
\end{equation*}
where $Z_1, \ldots, Z_n$ are connected symplectic manifolds. We get a new category $\mathcal{A}$ such that

\begin{equation*}
\begin{gathered}
\text{$Ob(\mathcal{A}) = $ \{elements of the form $(Z_1, r_1) \oplus \ldots \oplus (Z_n, r_n)$, where $n$ is arbitrary\},  } \\
Mor_{\mathcal{A}}\big(\oplus_{i=1}^n(Z_i, r_i), \; \oplus_{j=1}^k (Z_j', r_j')\big) = Mor_{\mathcal{B}}\big(\sqcup_{i=1}^n Z_i, \; \sqcup_{j=1}^k Z_j'\big), \\
Mor_{\mathcal{A}}(\oplus_{i=1}^n(Z_i, r_i), \; \emptyset) = Mor_{\mathcal{A}}(\emptyset, \; \oplus_{i=1}^n(Z_i, r_i)) = 0.
\end{gathered}
\end{equation*}
The set of morphisms are independent of constants in the pairs $(Z_i, r_i)$. The reason to introduce these pairs is explained below, where we discuss a filtration.

\begin{lemma}
The category $\mathcal{A}$ is naturally equivalent to $\mathcal{B}$. The category $\mathcal{A}$  is additive, where the addition of $(Y, r)$ and $(X, s)$ is $(Y,r) \oplus (X, s)$. The empty set is the zero object.
\end{lemma}

\begin{proof}

There is a functor $I: \mathcal{B} \rightarrow \mathcal{A}$
\begin{equation*}
I(\sqcup_{i=1}^n Z_i) = \oplus_{i=1}^n(Z, 0), \;\;\; I(F) = F,
\end{equation*}
where $F$ is a morpism. By definition, this functor is full and faithful. Also, any element $\oplus_{i=1}^n(Z_i, r)$ is isomorphic to $\oplus_{i=1}^n (Z_i, 0) = I(\sqcup_{i=1}^n Z_i)$. This implies that $\mathcal{B}$ is naturally equivalent to $\mathcal{A}$.

The rest of the lemma mimics the proof of Lemma $\ref{triangproof}$  where we need to replace $\sqcup_{i=1}^n Z_i$ by $\oplus_{i=1}^n (Z_i, r_i)$.
\end{proof}

We discussed (see formula $(\ref{filtration}))$ that bounded correspondences $SC_b(Y, Z)$ are filtered. We define the filtration on the set of morphisms of $\mathcal{A}$ in the following way:
\begin{equation*}
\begin{gathered}
Mor_{\mathcal{A}}((Y,r), (X,s); \; k) = SC_b((Y,r), (X,s); \; k) = SC_b(Y, X; \; k+r-s), \\
SC_b\big(\oplus_{i=1}^n (Z_i, r_i), \; \oplus_{j=1}^m (Z_j', r_j'); \; k \big) =  \bigoplus_{i,j} SC_b((Z_i, r_i), (Z_j', r_j'); \; k) = \\
= \bigoplus_{i,j} SC_b(Z_i, Z_j'; \; k+r_i-r_j').
\end{gathered}
\end{equation*}
By definition of the composition of correspondences (see Section $\ref{operationcorr}$), we have
\begin{equation*}
\begin{gathered}
SC_b((Z', r'), (Z'', r''); \; k) \times SC_b((Z,r), (Z',r'); \; \ell) = \\
 SC_b(Z', Z''; \; k+r'-r'') \times SC_b(Z, Z'; \; \ell+r-r') \xrightarrow{\circ} \\
SC_{b}(Z, Z''; \; k + \ell+r - r'') = SC_b((Z, r), (Z'', r''); \; k+\ell).
\end{gathered}
\end{equation*}
There is a subcategory $\mathcal{A}_0 \subset \mathcal{A}$, where
\begin{equation*}
\begin{gathered}
Ob(\mathcal{A}_0) = Ob(\mathcal{A}), \\
Mor_{\mathcal{A}_0}(\oplus_{i=1}^n (Z_i, r_i), \; \oplus_{j=1}^m (Z_j', r_j')) = SC_b(\oplus_{i=1}^n (Z_i, r_i), \; \oplus_{j=1}^m (Z_j', r_j'); \; 0).
\end{gathered}
\end{equation*}

\begin{lemma}
The category $\mathcal{A}_0$ is additive.
\end{lemma}

\begin{proof}
The proof mimics the proof of Lemma $\ref{triangproof}$.
\end{proof}

There are endofunctors $\Sigma^s: \mathcal{A} \rightarrow \mathcal{A}$ for any $s \in \mathbb{R}$ defined in the following way:
\begin{equation*}
\Sigma^s(\oplus_{i=1}^n(Y_i, r_i)) = \oplus_{i=1}^n(Y_i, r_i + s), \;\;\; \text{$\Sigma^s$ is identity on morphisms}.
\end{equation*}
We see that $\Sigma^r \circ \Sigma^s = \Sigma^{r+s}$ and $\Sigma^0 = id_{\mathcal{A}}$.

\begin{example}\label{examplefiltr}
By definition,
\begin{equation*}
SC_b((Y,r), (Y, s); \; s-r) = SC_b(Y, Y; \; 0).
\end{equation*}
We denote the morphism in $SC_b((Y,r), (Y, s); \; s-r)$ that corresponds to the identity map of $Y$ by $id_Y^{r, s}$. Note that
\begin{equation*}
id_Y^{r', s} \circ id_Y^{r, r'} = id_Y^{r, s}.
\end{equation*}
\end{example}

\vspace*{0.06in}

Consider the category of bounded below cochain complexes $Com^{-}(\mathcal{A})$ and define
\begin{equation*}
Hom^r(C_{\bullet}, K_{\bullet}) = \{\text{chain maps with components $F_n \in SC_b(C_n, K_n; \; r)$ } \}/\simeq_r,
\end{equation*}
where the relation $\simeq_r$ is chain homotopy via a homotopy with components $H_n \in SC_b(C_n, K_n; \; r)$. If $r < s$, then any $F \in Hom^r(C_{\bullet}, K_{\bullet})$ belongs also to $Hom^s(C_{\bullet}, K_{\bullet})$. So, we get the following maps
\begin{equation*}
i_{r, s}: Hom^r(C_{\bullet}, K_{\bullet}) \rightarrow Hom^s(C_{\bullet}, K_{\bullet}).
\end{equation*}
Note that $i_{r,s}$ are not inclusions in general because maps may be homotopic via correspondences from $SC_b(C_n, K_n; \;s)$ but not homotopic via correspondences $SC_b(C_n, K_n; \; r)$. On the other hand, if a morphism is equal to zero in $Hom^r(C_{\bullet}, K_{\bullet})$, then it is equal to zero in $Hom^s(C_{\bullet}, K_{\bullet})$.

\begin{lemma}\label{compwelldef}
The composition
\begin{equation*}
Hom^r(Y_{\bullet}, X_{\bullet}) \times Hom^s(Z_{\bullet}, Y_{\bullet}) \xrightarrow{\circ} Hom^{r+s}(Z_{\bullet}, X_{\bullet}).
\end{equation*}
is well-defined.
\end{lemma}

\begin{proof}
Consider $F + \widetilde{F}$ and $G + \widetilde{G}$, where $F \in Hom^r(C_{\bullet}, K_{\bullet})$, $G \in Hom^s(L_{\bullet}, C_{\bullet})$, $\widetilde{F}, \widetilde{G}$ are equal to zero in $Hom^r(C_{\bullet}, K_{\bullet})$ and $Hom^s(L_{\bullet}, C_{\bullet})$, respectively. Then,
\begin{equation*}
(F + \widetilde{F}) \circ (G + \widetilde{G}) = F \circ G + F\circ \widetilde{G} + \widetilde{F} \circ G + \widetilde{F} \circ \widetilde{G}.
\end{equation*}
Let $H$ be a homotopy between $\widetilde{G}$ and $0$, i.e. there are maps $H_n \in C(Z_n, Y_{n-1}; \; s)$ such that $\widetilde{G}_n = d_Y^{n-1}H_n \pm H_{n+1}d_Z^n$. We see that
\begin{equation*}
\begin{gathered}
F_{n-1} \circ \widetilde{H}_n \in C(C_n, K_{n-1}; \; r+s) \\
d_K^{n-1}(F_{n-1} \circ H_n) \pm (F_{n} \circ H_{n+1}) d^n_L = F_n \circ (d_K^{n-1} H_n \pm H_{n+1} d^n_L) = F_n \circ \widetilde{G}_n.
\end{gathered}
\end{equation*}
This means that $F \circ \widetilde{G} = 0$ in $Hom^{r+s}(L_{\bullet}, K_{\bullet})$. In the similar way, we show that $\widetilde{F} \circ G = 0$ in $Hom^{r+s}(L_{\bullet}, K_{\bullet})$. This also shows that $\widetilde{F} \circ \widetilde{G}$ is homotopic to $0$ in $Hom^{r+s}(L_{\bullet}, K_{\bullet})$. As a result, we see that the composition is well-defined.

\end{proof}

We define categories $\mathcal{C}$, $\mathcal{C}_0$, and $\mathcal{C}_{\infty}$, where
\begin{equation*}
\begin{gathered}
Ob(\mathcal{C}) = Ob(\mathcal{C}_0) = Ob(\mathcal{C}_{\infty}) = Ob(Com^{-}(\mathcal{A})), \\
Mor_{\mathcal{C}}(C_{\bullet}, K_{\bullet}) = \coprod\limits_{r} Hom^r(C_{\bullet}, K_{\bullet}), \;\;\; Mor_{\mathcal{C}_0}(C_{\bullet}, K_{\bullet}) = Hom^0(C_{\bullet}, K_{\bullet}), \\
Mor_{\mathcal{C}_{\infty}} = \varinjlim Hom^r(C_{\bullet}, K_{\bullet})
\end{gathered}
\end{equation*}
and the limit is taken with respect to maps $i_{r, s}$.

We have a functor $T: \mathcal{C} \rightarrow \mathcal{C}$
\begin{equation*}
T(C_{\bullet}, d_C) = (C_{\bullet + 1}, -d_C), \;\;\; T(F_n) = F_{n+1},
\end{equation*}
where $F_n$ are the components of a chain map $F$. We have the same functor $\mathcal{C}_0 \rightarrow \mathcal{C}_0$ and denote it also by $T$.

\begin{lemma}
The category $\mathcal{C}_0$ is triangulated, where $T$ is the translation functor.
\end{lemma}

\begin{proof}
The category $\mathcal{C}_0$ coincides with the homotopy category of cochain complexes of $\mathcal{A}_0$. Since $\mathcal{A}_0$ is additive, we get that $\mathcal{C}_0$ is triangulated (see \cite[Th. 11.2.6]{categories})

\end{proof}

\begin{lemma}
The category $\mathcal{C}$ is persistence category.
\end{lemma}

\begin{proof}
Let $r',s'$ be integers such that $r<r'$ and $s<s'$. We need to prove that $i_{r, r'}(F) \circ i_{s, s'}(G) = i_{r+ s, r' + s'}(F \circ G)$. Consider $F + F'$ and $G + G'$, where $F' = 0$ and $G' = 0$ in $Hom^{r'}(C_{\bullet}, K_{\bullet})$ and $Hom^{s'}(L_{\bullet}, C_{\bullet})$, respectively. Arguing as Lemma $\ref{compwelldef}$, we can show that $F' \circ G = 0$, $F \circ G' = 0$, and $F' \circ G' = 0$ in $Hom^{r' + s}(L_{\bullet}, K_{\bullet})$, $Hom^{r+ s'}(L_{\bullet}, K_{\bullet})$, and $Hom^{r' + s'}(L_{\bullet}, K_{\bullet})$, respectively. Since $r' + s$ and $r + s'$ are less than $r' + s'$, we get that all the mentioned terms are equal to zero in $Hom^{r' + s'}(L_{\bullet}, K_{\bullet})$. So, we have
\begin{equation*}
i_{r, r'}(F) \circ i_{s, s'}(G) = i_{r+s, r' + s'}(F \circ G).
\end{equation*}
By definition, this means that $\mathcal{C}$ is persistence category.

\end{proof}

\begin{lemma}
The categories $\mathcal{C}_{\infty}$ and $\mathcal{K}(\mathcal{B})$ (from the previous section) are naturally equivalent.
\end{lemma}

\begin{proof}
There is embedding $I: \mathcal{K}(\mathcal{B}) \rightarrow \mathcal{C}_{\infty}$, where each cochain complex $C_{\bullet}$ with $C_n = \sqcup_{i=1}^{n} Z_i$ goes to cochain with $I(C_{\bullet})_n = \oplus_{i=1}^n (Z_i, 0)$ and $I$ is identity on morphisms. Since any $\oplus_{i=1}^n(Z_i, r_i)$ is isomorphic to $\oplus_{i=1}^n(Z_i, 0)$. We get that any chain complex from $\mathcal{C}_{\infty}$ is isomorphic to an object of $I(\mathcal{B})$. Since the functor $I$ is full and faithful, we see that $I$ is natural equivalence.

\end{proof}

The category $\mathcal{C}$ inherits endofunctors $\Sigma^s: \mathcal{C} \rightarrow \mathcal{C}$. We define $\Sigma^r(C_{\bullet}, d_C)$ such that
\begin{equation*}
(\Sigma^rC_{\bullet})_n = \Sigma^rC_n, \;\; \text{$\Sigma^r$ is identity on morphisms}.
\end{equation*}

\begin{example}\label{identityfiltr}
We see that the identity map $id_{C_{\bullet}}: C_{\bullet} \rightarrow C_{\bullet}$ consists of identity maps $id_{C_n} \in SC_b(C_n, C_n; \;0)$ and belongs to $Hom^0(C_{\bullet}, C_{\bullet})$.

We constructed morphisms $id_{C_n}^{r, s}$ (see Example \ref{examplefiltr}) and they give us chain maps
\begin{equation*}
id_{C_{\bullet}}^{r, s}: \Sigma^rC_{\bullet} \rightarrow \Sigma^sC_{\bullet}.
\end{equation*}
We see that
\begin{equation*}
id_{C_{\bullet}}^{r',s} \circ id_{C_{\bullet}}^{r,r'} = id_{C_{\bullet}}^{r,s}.
\end{equation*}
\end{example}

\vspace*{0.08in}

In general, the identity map provides natural transformation between functors $\Sigma^r$ and $\Sigma^s$. We have the following identities
\begin{equation*}
\begin{gathered}
Hom^s(\Sigma^rC_{\bullet}, K_{\bullet}) = Hom^{s+r}(C_{\bullet}, K_{\bullet}), \;\;\; Hom^s(C_{\bullet}, \Sigma^rK_{\bullet}) = Hom^{s-r}(C_{\bullet}, K_{\bullet}),\\
Hom^s(\Sigma^rC_{\bullet}, \Sigma^kK_{\bullet}) = Hom^{s+r-k}(C_{\bullet}, K_{\bullet}).
\end{gathered}
\end{equation*}

\begin{lemma}
The functor $T$ and endofunctors $\Sigma^r$ commute for any $r$.
\end{lemma}

\begin{proof}
This follows from definitions.

\end{proof}

We define a complex $(C_{\bullet} \oplus K_{\oplus}, d_{C \oplus K})$ as a complex with
\begin{equation*}
(C_{\bullet} \oplus K_{\bullet})_n = C_n \oplus K_n, \;\;\; d_{C \oplus K} =   \begin{pmatrix}
d_C & 0 \\
0 & d_K
\end{pmatrix}.
\end{equation*}
By definition, $(\Sigma^r(C_{\bullet} \oplus K_{\bullet}))_n = \Sigma^rC_n \oplus \Sigma^rK_n$ and $\Sigma^r(d_{C \oplus K}) = d_{C \oplus K}$.

\vspace{0.08in}

Let $F: C_{\bullet} \rightarrow K_{\bullet}$ be a chain map. We define a cone of $F$
\begin{equation*}
Cone(F) = (T(Y_{\bullet}) \oplus Z_{\bullet}, d_{cone}). \;\;\; d_{cone} = \begin{pmatrix}
-d_C & 0 \\
F & d_K
\end{pmatrix}.
\end{equation*}

A cochain complex $C_{\bullet}$ is called $r-$acyclic if the  $id_{C_{\bullet}}^{0,0} = 0$ in $Hom^r(C_{\bullet}, C_{\bullet})$, i.e. $id_{C_{\bullet}}^{0,0}$ homotopic to zero via homotopy with components in $SC_b(C_{n}, C_{n-1}; \;r)$. A chain map $F$ is called $r-$isomorphism if $id_{cone(F)}^{0,0} = 0$ in $Hom^{r}(Cone(F), Cone(F))$.

\begin{lemma}
The cone of the map (see Example $\ref{identityfiltr}$)
\begin{equation*}
id_{C_{\bullet}}^{r,0}: \Sigma^rC_{\bullet} \rightarrow C_{\bullet}
\end{equation*}
is $r-$acyclic for any complex $C_{\bullet}$.
\end{lemma}

\begin{proof}
We need to prove that $Cone(id_{C_{\bullet}}^{r,0})$ is $r-$acyclic. This means that
\begin{equation*}
\begin{pmatrix}
id_{C_{n+1}}^{r,r} & 0 \\
0 &       id_{C_{n}}^{0,0}
\end{pmatrix}
 = 0\;\; \text{in $Hom^r(Cone(id_{C_{\bullet}}^{r,0}), Cone(id_{C_{\bullet}}^{r,0}))$}.
\end{equation*}
We define a homotopy
\begin{equation*}
H_n: \Sigma^rC_{n+1} \oplus C_n \rightarrow \Sigma^rC_n \oplus C_{n-1}, \;\;\; H_n = \begin{pmatrix}
0 & id_{C_{n}}^{0,r} \\
0 & 0
\end{pmatrix},
\end{equation*}
where $id_{C_{n}}^{0,r}: C_n \rightarrow \Sigma^rC_n$. Therefore $H_n \in SC_b(Cone(id_{C_{n}}^{r,0}), Cone(id_{C_{n}}^{r,0}); \; r)$ (because all entries of the matrix belong to $SC_b(\cdot, \cdot; \; r)$). We have
\begin{equation*}
\begin{gathered}
d_{Cone}^{n-1} \circ H_n = \begin{pmatrix}
0 & -d_{C}^{n} \\
0 & id_{C_{n}}^{0,0}
\end{pmatrix}
\;\;\;
H_{n+1} \circ d_{cone}^n = \begin{pmatrix}
id_{C_{n+1}}^{r, r} & d_{C}^n \\
0 & 0
\end{pmatrix}
\\
d_{cone}^{n-1}\circ H_n+ H_{n+1} \circ d_{cone}^n = \begin{pmatrix}
id_{C_{n+1}}^{r,r} & 0 \\
0 &       id_{C_{n}}^{0,0}
\end{pmatrix}
\end{gathered}
\end{equation*}
This shows that the cone is $r-$acyclic.
\end{proof}

\begin{theorem}
The category $\mathcal{C}$ is triangulated persistence category
\end{theorem}

\begin{proof}
This follows from the previous proved lemmas.
\end{proof}

\subsection{Distances between symplectic manifolds}

In the previous section we constructed triangulated persistence category of symplectic manifolds $\mathcal{C}$. This allows us to use the machinery developed by Biran, Cornea, and Zhang (see papers $\cite{Octav1}$, $\cite{Octav2}$) to define distances between symplectic manifolds. Let $\mathcal{F} \subset Ob(\mathcal{C})$ be a subset of objects. We have a  pseudometric $d^{\mathcal{F}}$ between objects of $\mathcal{C}$. We defined chain complexes $Y_{\bullet}(M, p_0)$, $Y_{\bullet}(M, p_0,p_1)$ (see Examples $\ref{examplechain1}$, $\ref{examplechain2}$) and they can be used to define distances between symplectic manifolds.

We fix $(M, p_0, p_1)$ and define distance between $Y$ and $Z$ by the following formulas
\begin{equation*}
\begin{gathered}
dist(Y, Z) = d^{\mathcal{F}}(Y_{\bullet}(M, p_0, p_1), Z_{\bullet}(M, p_0, p_1)).
\end{gathered}
\end{equation*}
There is another option
\begin{equation*}
\begin{gathered}
dist(Y, Z) = d^{\mathcal{F}}(Y_{\bullet}(M, p_0, p_1), Z_{\bullet}(M, p_0, p_1)).
\end{gathered}
\end{equation*}

\section{Some computations}\label{computations}

In this section we prove theorems and lemmas mentioned in Sections $\ref{question1}$ and $\ref{question2}$.

\vspace{0.08in}
\noindent
\textbf{Proof of Theorem $\ref{cohomkahler}$}. First, let us give the following definition.

\begin{definition}
A complex manifold $X$ is called algebraically connected if for any general pair of points $p, q \in X$, there exists a proper, connected and not necessarily irreducible curve in $X$ containing $p$ and $q$.
\end{definition}

We need the following criteria proved by Campana.

\begin{theorem}(see $\cite[p. 212]{Campana})$.
A compact complex manifold is projective if and only if it is Kahler and algebraically connected.
\end{theorem}

Let $\mathbb{Z}[X]$ be a group of finite formal linear combinations of points of $X$. By definition,
\begin{equation*}
JC_0^{\mathbb{C}P^1}(pt, X) = \mathbb{Z}[X], \quad JC_1^{\mathbb{C}P^1}(pt, X) = JC(\mathbb{C}P^{1}, X).
\end{equation*}
If $JH_0^{\mathbb{C}P^1}(pt, X) = \mathbb{Z}$, then it is generated by one point and for any $p, q \in X$ there exists $F = r_1F_1 + \ldots + r_nF_n \in JC(\mathbb{C}P^1, X)$ such that $dF = p - q$.

Consider a correspondence $G = F_1 + \ldots + F_n$. Also, denote by  $G$ the associated subset of $\mathbb{C}P^1 \times X$ and by $\pi$ the projection $G \rightarrow \mathbb{C}P^1$. Since $dF = p-q$, we get that $p$ and $q$ belong to one connected component of $G$. We denote the connected component containing $p$ and $q$ also by $G$.

If $U \subset \mathbb{C}P^1$ is open and small, then there is $N$ such that
\begin{equation*}
\pi^{-1}(U) = \bigcup_{i=1}^N(U, f_i(U)), \;\;\; \text{where $f_i: U  \rightarrow X$ are holomorphic}.
\end{equation*}
Let $G_1, \ldots, G_k$ be irreducible components of $G$. There are numbers $n_1, \ldots, n_k$ such that
\begin{equation*}
G_j|_U = \bigcup_{i=1}^{n_j}(U, g_i(U)), \;\;\; \text{where $g_i: U  \rightarrow X$ are holomorphic}.
\end{equation*}
Since distinct holomorphic functions can be equal only at discrete points, we get that $\pi: G_j \rightarrow \mathbb{C}P^1$ is a branched covering. This means that $G_j$ is a curve of some genus.

Let $pr_X$ be the projection of $\mathbb{C}P^1 \times X$ onto $X$. Since $dG = p - q \pm \ldots$, we see that  $p, q \in pr_X(G)$.

We proved that $X$ is algebraically connected. Then the criteria of Campana says that $X$ is algebraic (we assume that $X$ is Kahler).

\vspace{0.1in}
\noindent
\textbf{Proof of Lemma $\ref{curvepassinglemma}$}. As before,
\begin{equation*}
JC_0^{\mathbb{C}P^1}(pt, X) = \mathbb{Z}[X], \quad JC_1^{\mathbb{C}P^1}(pt, X) = JC(\mathbb{C}P^{1}, X),
\end{equation*}
where $\mathbb{Z}[X]$ is a group of finite formal linear combinations of points of $X$. Recall that $\mathcal{J}\mathcal{S}ymp(\mathbb{C}P^1, X) \subset JC(\mathbb{C}P^{1}, X)$. We assumed that for any $p, q \in X$ there is $f \in \mathcal{J}\mathcal{S}ymp(\mathbb{C}P^1, X)$ and $a, b \in \mathbb{C}P^1$ such that $f(a) = p$ and $f(b) = q$. Let $\gamma: \mathbb{C}P^1 \rightarrow \mathbb{C}P^1$ be an automorphism such that
\begin{equation*}
\gamma([0:1]) = a, \quad \gamma([1:0]) = b.
\end{equation*}
Then, $d(f \circ \gamma) = f \circ \gamma([0:1]) - f \circ \gamma([1:0]) = p-q$. This means that $JH^{\mathbb{C}P^1}_0(pt, X)$ is generated by one point. As a result, $JH^{\mathbb{C}P^1}_0(pt, X) = \mathbb{Z}$.

\vspace{0.1in}
\noindent
\textbf{Proof of Lemma $\ref{sympapplemma}$}. Since there exists symplectic embedding of $M$ into $X^n$, we get that $SCor(M, X) \neq \emptyset$.  By definition,
\begin{equation*}
SC_0^M(Y, X) = SC(Y, X), \quad SC_1^M(Y, X) = SC(Y \times M, X).
\end{equation*}
Let $f, g: Y \rightarrow X$ be our given symplectic embeddings and $f, g \in SCor(Y, X)$ be the associated correspondences. Theorem $\ref{deformtheorem}$ says that if there is symplectic isotopy $\varphi_t$ connecting $f$ and $g$, then  $f$ and $g$ are $M$-homotopic. This means that there is $H \in SC(Y \times M, X)$ such that $H|_{p_0} = f$ and $H|_{p_1} = g$. We also know that $dH = H|_{p_0} - H|_{p_1} = f - g$.

So, if $f$ and $g$ are symplectically isotopic, then $f$ and $g$ represent the same class in cohomology groups. On the other hand, we assumed that $f$ and $g$ represent different classes. This implies that the symplectic isotopy does not exist.


\begin{thebibliography}{100}
\bibitem{Octav1} P. Biran, O. Cornea, J. Zhang, \emph{Triangulation, Persistence, and Fukaya categories}, arXiv:2304.01785.
\bibitem{Octav2} P. Biran, O. Cornea, J. Zhang, \emph{Persistence K-theory}, arXiv:2305.01370.
\bibitem{Campana} F. Campana, \emph{Coreduction algebrique d'un espace analytique faiblement Kahlerien compact}, vol. 63 (1981), 187--223.
\bibitem{DoldThom} A. Dold, R. Thom,  \emph{Quasifaserungen und unendliche symmtrische Produkte}, Annals of Mathematics, vol 67 (1958), 230--281.
\bibitem{Fukaya} K. Fukaya, \emph{Morse homotopy and its quantization}, AMS/IP Studies in Advanced Math.  vol. 2 (1997), 409 -- 440.
\bibitem{Gromov} M. Gromov, {Soft and hard symplectic geometry}, ICM Series, American Mathematical Society, Providence, RI, 1988.
\bibitem{Gromov2} M. Gromov, \emph{Pseudo holomorphic curves in symplectic manifolds}. Inventiones Mathematicae, vol. 82 (1985), 307--347.
\bibitem{categories} M Kashiwara, \emph{Categories and sheaves}, Grundlehren der Mathematischen Wissenschaften, vol. 332 (2006), Springer-Verlag.
\bibitem{sheaves1} S. Mac Lane, I. Moerdijk,   \emph{Sheaves in Geometry and Logic - A first introduction to topos theory}, Springer Verlag (1992).
\bibitem{motivic1} C. Mazza, V. Voevodsky, C. Weibel, \emph{Lectures on Motivic Cohomology}, Clay Mathematics Monographs, vol. 2 (2006).
\bibitem{motivic2} F. Morel, V. Voevodsky,  \emph{$A^1$-homotopy theory of schemes}, Publications Mathematiques de L'Institut des Hautes Scientifiques, vol. 90 (1999), 45--143.
\bibitem{Pardon} J. Pardon, \emph{An algebraic approach to virtual fundamental cycles on moduli spaces of pseudo-holomorphic curves}, Geometry and Topology, vol. 20 (2016), 779--1034.
\bibitem{Seidel1} P. Seidel, \emph{Floer homology and the symplectic isotopy problem}. Ph.D. thesis, Oxford University, 1997.
\bibitem{Seidel2}  P. Seidel, \emph{Lectures on four-dimensional Dehn twists. Symplectic four-manifolds and algebraic surfaces}, Lecture Notes in Mathematics vol. 1938 (2008), 231--267.
\bibitem{Tian} B. Siebert, G. Tian, \emph{On the holomorphicity of genus two Lefschetz fibrations}, Annals of Mathematics, vol. 161 (2005), 959--1020.
\bibitem{Sikorav} J. Sikorav, \emph{The gluing construction for normally generic J-holomorphic curves}, Symplectic and contact topology: interactions and perspectives, vol. 35 (2003)  175--199.
\bibitem{Starkston} L. Starkston, \emph{A new approach to the symplectic isotopy problem}, Journal of Symplectic Geometry, vol. 18 (2020), 939--960.
\bibitem{motivic3} A. Suslin, V. Voevodsky, \emph{Singular homology of abstract algebraic varieties}, Inventiones Mathematicae, vol. 123 (1996), 61--94.
\end{thebibliography}
\end{document}